\documentclass{amsart}
\usepackage{amsfonts}
\usepackage{geometry}
\usepackage{setspace}
\usepackage{color}
\usepackage{cite}

\newtheorem{theorem}{Theorem}
\theoremstyle{plain}

\newtheorem{assumption}[theorem]{Assumption}

\newtheorem{condition}[theorem]{Condition}

\newtheorem{corollary}[theorem]{Corollary}

\newtheorem{definition}[theorem]{Definition}
\newtheorem{example}[theorem]{Example}

\newtheorem{lemma}[theorem]{Lemma}

\newtheorem{proposition}[theorem]{Proposition}
\newtheorem{property}[theorem]{Property}
\newtheorem{remark}[theorem]{Remark}

\numberwithin{equation}{section}
\numberwithin{theorem}{section}

\input{tcilatex}

\begin{document}
\title[LDP for mean field particle systems]{Large Deviation Principle For
Finite-State Mean Field Interacting Particle Systems}
\author{Paul Dupuis}
\author{Kavita Ramanan}
\author{Wei Wu}
\thanks{Division of Applied Mathematics, Brown University and New York
University. This research was supported in part by the Army Research Office
(W911NF-12-1-0222).}
\date{\today }
\maketitle

\begin{abstract}
We establish a large deviation principle for the empirical measure process
associated with a general class of finite-state mean field interacting
particle systems with Lipschitz continuous transition rates that satisfy a
certain ergodicity condition. The approach is based on a variational
representation for functionals of a Poisson random measure. Under an
appropriate strengthening of the ergodicity condition, we also prove a
locally uniform large deviation principle. The main novelty is that more
than one particle is allowed to change its state simultaneously, and so a
standard approach to the proof based on a change of measure with respect to
a system of independent particles is not possible. The result is shown to be
applicable to a wide range of models arising from statistical physics,
queueing systems and communication networks. Along the way, we establish a
large deviation principle for a class of jump Markov processes on the
simplex, whose rates decay to zero as they approach the boundary of the
domain. This result may be of independent interest.
\end{abstract}

\section{Introduction}

Markovian particle systems on finite state spaces under mean field
interactions arise in many different contexts. They appear as approximations
of statistical physics models in higher dimensional lattices (for various
types of spin dynamics, see \cite{M} and references therein), kinetic theory 
\cite{K}, game theory \cite{GMS} and as models of communication networks 
\cite{AFRT}, \cite{GibHunKel90}, \cite{GM}, \cite{T}. The dynamics of these
particle systems have the following common features: a) particles are
exchangeable, that is, their joint distribution is invariant under
permutation of their indices; b) at each time, multiple particles in some
finite subset can switch their states simultaneously; c) the interaction
between particles is global but weak, in the sense that the jump rate of
each group of particles is a function only of the initial and final
configurations of that group of particles, and the empirical measure of all
particles. The precise dynamics of the Markovian $n$-particle system we
consider are described in Section 2.1.

Due to the exchangeability assumption, many essential features of the state
of the particle system can be captured by its empirical measure, which
evolves as a jump Markov process on (a sublattice of) the unit simplex.
Under mild assumptions on the jump rates, standard results on jump Markov
processes (see \cite{Oel}) show that the functional law of large numbers
limit of the sequence of $n$-particle empirical measures is the solution of
a nonlinear ordinary differential equation (ODE) on the unit simplex. The
ODE also characterizes the transition probabilities of a certain
\textquotedblleft nonlinear Markov process\textquotedblright\ that describes
the limiting distribution of a typical particle in the system, as the number
of particles goes to infinity \cite{Ko}, and is commonly referred to as the
McKean-Vlasov limit. In this paper we consider the sample path large
deviation properties of the sequence of empirical measure processes as the
number of particles tends to infinity. In the case of interacting diffusion
processes, such a large deviation principle (LDP) was first established by
Dawson and G\"{a}rtner in \cite{DG}. The sample path large deviation
principle over finite time intervals has a number of applications, including
the study of metastability properties via Freidlin-Wentzell theory \cite{FW}
(see also \cite{OliVarBook05} and \cite{BerLan11} for the reversible case),
and the study of the possible evolution of a Gibbs measure into a non-Gibbs
measure under stochastic (e.g., spin-flip) dynamics (which is referred to as
a Gibbs-non Gibbs transition in \cite{FerdenHolMar13}).

Large deviation principles for jump Markov processes are known if the jump
rates are Lipschitz continuous and uniformly bounded below away from zero
(cf. \cite{SW}). In this case, the large deviation rate function admits an
integral representation in terms of a so-called local rate function.
However, the jump rates in our model do not satisfy this condition.
Specifically, as the empirical measure approaches the boundary of the
simplex, its jump rates along certain directions converge to zero.
Nevertheless, we show that (under general conditions on the jump rates), the
sequence of empirical measure processes satisfies a sample path LDP with the
rate function having the standard integral representation. Under mild
conditions, we also establish a \textquotedblleft locally
uniform\textquotedblright\ refinement \cite{SW}, which characterizes the
decay rate of the probabilities of hitting a convergent sequence of points.
Such a result is of relevance only for discrete Markov processes (and not
for diffusions) and does not follow immediately from the LDP. 
The locally uniform refinement is shown in \cite{BDFR14a,BDFR14b} to be
relevant for the study of stability properties of the nonlinear ODE that
describes the law of large numbers (LLN) limit. All the main results of this
paper are formulated for a more general class of jump Markov processes on
the simplex whose rates diminish to zero at the boundary, and the
interacting particle models are obtained as a special case.

Other works that have studied large deviations for jump Markov processes
with vanishing rates include \cite{SW-p}, \cite{KarRam15}, \cite{Leo95} and 
\cite{BS13}. However, the results in \cite{SW-p} impose special conditions
on the jump rates near the boundary, which do not apply to our model (see
Appendix A of \cite{WW}). On the other hand, the methods used in \cite{Leo95}
and \cite{BS13} are adaptations of the argument used by Dawson and Gartner
in \cite{DG}, which crucially relies on the fact that the measure on path
space induced by the interacting $n$-particle process is absolutely
continuous with respect to that induced by $n$ independent (non-interacting)
particles, each evolving according to a time inhomogenous Markov process.
This property does not hold when multiple particles jump simultaneously.
Simultaneous jumps are a common feature of models used in many applications
(see Example \ref{ARN} and also \cite{T} and \cite[Chapter 8]{D}).

The large deviation upper bound follows from general results in \cite{DEW}
(see Section 5). The subtlety arises in the proof of the large deviation
lower bound. Our strategy for the proof is based on a variational
representation for the $n$-particle empirical measure process and a
perturbation argument near the boundary. The starting point of our
variational representation is a representation formula for functionals of
Poisson random measures \cite{BDM}, and an SDE representation of the
empirical measure process in terms of a sequence of Poisson random measures.
However, the state-dependent nature of the jump rates leads to a somewhat
complicated variational problem. We use the special structure of the SDE to
simplify the representation formula. The perturbation argument takes
inspiration from \cite{DNW}, where an LDP was established for a discrete
time one-dimensional Markov chain. Our model is higher dimensional, where
the perturbation argument becomes substantially more intricate, and geometry
comes into play. The variational representation that we establish holds more
generally for jump Markov processes with bounded jump rates, and could be
useful for obtaining other asymptotics.

The outline of this paper is as follows. In Section \ref{sec-ips} we set up
the mean field interacting particle system, and describe a few examples in
the literature that fit into the framework. In Section \ref{sec-mainres} we
state the main results, namely a sample path LDP for a general class of
weakly interacting particle systems (Theorem \ref{th-ldips}), its locally
uniform refinement (Theorem \ref{th-localldips}) and an LDP for the
corresponding sequence of stationary measures (Theorem \ref{thm:LDforIM}).
In Section \ref{sec-ver} we show that our assumptions on the transition
rates of the mean field interacting particle system imply that the jump
rates of the associated empirical measure process satisfy certain useful
properties, which are the only ones used in the proof of our results. As a
consequence, our main results in fact apply to the larger class of jump
Markov processes on the simplex whose jump rates possess these properties
(see Remark \ref{rem-ldp} for a precise statement). Section \ref{sec-varrep}
establishes the variational representation for the empirical measure
process, and provides an alternative proof for the functional LLN limit.
Some details of the proof of the variational representations are deferred to
the Appendix. The sample path large deviation upper and lower bounds are
derived in Section \ref{sec-pfupper} and Section \ref{sec-pflower},
respectively, while in Section \ref{sec-locrf} we study properties of the
local rate function. Section \ref{sec-locunif} is devoted to the proof of
the locally uniform LDP.

\section{The Interacting Particle Systems}

\label{sec-ips}

\subsection{Model Description}

\label{subs-model}

In this work, we consider an $n$-particle system in which the state of each
individual particle takes values in the finite set $\mathcal{X}\doteq
\left\{ 1,2,...,d\right\} $. For each $i=1,...,n$, let $X^{i,n}\left(
t\right) $ be the state of the $i^{\text{th}}$ particle at time $t$. For
simplicity of notation, we assume that the sequence of processes $%
X^{n}\left( \cdot \right) =\left\{ X^{n}\left( t\right) =\left(
X^{1,n}\left( t\right) ,...,X^{n,n}\left( t\right) \right) ,t\geq 0\right\} $%
, $n\in \mathbb{N}$, are defined on a common probability space $\left(
\Omega ,\mathcal{F},\mathbb{P}\right) $. Each $X^{n}\left( \cdot \right) $
evolves as a c\`{a}dl\`{a}g, $\mathcal{X}^{n}$-valued jump Markov process.
The associated empirical measure is denoted by 
\begin{equation*}
\mu ^{n}\left( t,\omega \right) =\frac{1}{n}\sum_{i=1}^{n}\delta
_{X^{i,n}\left( t,\omega \right) },\text{ \ \ }t\geq 0,\omega \in \Omega ,
\end{equation*}%
where $\delta _{x}$ represents the Dirac mass at $x$. In subsequent
discussions, we often suppress the dependence of $\mu ^{n}$ on $\omega $.

Let ${\mathcal{P}}\left( \mathcal{X}\right) $ denote the space of
probability measures on $\mathcal{X}$. We identify ${\mathcal{P}}\left( 
\mathcal{X}\right) $ with the simplex $\mathcal{S}\doteq \{x\in \mathbb{R}%
^{d}:x_{i}\geq 0,\sum_{i=1}^{d}x_{i}=1\}$ and endow $\mathcal{S}$ with the
topology induced from $\mathbb{R}^{d}$, so that $\mathcal{S}={\mathcal{P}}%
\left( \mathcal{X}\right) $ is equipped with the Euclidean norm $\left\Vert
\cdot \right\Vert $. Define ${\mathcal{P}}_{n}\left( \mathcal{X}\right)
\doteq \left\{ \frac{1}{n}\sum_{i=1}^{n}\delta _{x_{i}}:x\in \mathcal{X}%
^{n}\right\} \subset {\mathcal{P}}\left( \mathcal{X}\right) $. Then ${%
\mathcal{P}}_{n}\left( \mathcal{X}\right) $ can be similarly identified with
the lattice $\mathcal{S}_{n}\doteq \mathcal{S}\cap \frac{1}{n}\mathbb{Z}^{d}$%
, and clearly $\mu ^{n}(\cdot )=\{\mu ^{n}(t),t\geq 0\}$ is an $\mathcal{S}%
_{n}$-valued stochastic process.

The possible transitions of $X^{n}$ are as follows. It is assumed that there
exists $K\in \mathbb{N}$ such that at most $K$ particles jump
simultaneously. When $K=1$, almost surely at most one particle can
instantaneously change its state. For $i,j\in \mathcal{X}$, $i\neq j$, and $%
t\geq 0$, the rate at which a particle changes its state from $i$ to $j$ at
time $t$ is assumed to be $\Gamma _{ij}^{n}\left( \mu ^{n}\left( t\right)
\right) $, where $\left\{ \Gamma ^{n}\left( x\right) ,x\in \mathcal{S}%
_{n}\right\} $ is a family of nonnegative $d\times d$ matrices, and we set $%
\Gamma _{ii}^{n}\left( x\right) \doteq -\sum_{j=1,j\neq i}^{d}\Gamma
_{ij}^{n}\left( x\right) $ for $i=1,\ldots ,d$. For general $K\in \mathbb{N}$
(in which case, we will always assume without loss of generality that $n\geq
K$), for each $k\in \{1,\ldots ,K\}$, an ordered collection of $k$ particles
among all possible ordered $k$-tuples of the $n$-particle system can
simultaneously change its configuration from $\mathbf{i}=\left(
i_{1},..,i_{k}\right) \in \mathcal{X}^{k}$ to $\mathbf{j}=\left(
j_{1},...,j_{k}\right) \in \mathcal{X}^{k}$, where $i_{l}\neq j_{l}$, for $%
l=1,...,k$.

Note that it is possible that multiple particles in the $k$-tuple may be in
the same state. Let $\mathcal{J}\doteq \cup _{k=1}^{K}\mathcal{J}^{k}$,
where for $k=1,\ldots K$, 
\begin{equation}
\mathcal{J}^{k}\doteq \left\{ \left( \mathbf{i,j}\right) \in \mathcal{X}%
^{k}\times \mathcal{X}^{k}:i_{l}\neq j_{l}\text{ for }l=1,...,k\right\}
\label{def-kjumpset}
\end{equation}%
is the collection of all possible pairs of initial and final configurations
for an ordered $k$-tuple of particles. At time $t$, the rate of a
simultaneous transition of a $k$-tuple from $\mathbf{i}\in \mathcal{X}^{k}$
to $\mathbf{j}\in \mathcal{X}^{k}$ is given by $\Gamma _{\mathbf{ij}%
}^{k,n}\left( \mu ^{n}\left( t\right) \right) $, where for each $(\mathbf{i},%
\mathbf{j})\in {\mathcal{J}}^{k}$, $\Gamma _{\mathbf{ij}}^{k,n}$ is a
function from ${\mathcal{S}}_{n}$ to $[0,\infty )$. 
We also assume that the transition rate is independent of the ordering of
the particles: if $\mathbb{S}_{k}$ denotes the group of permutations on $%
\left\{ 1,...,k\right\} $, then 
\begin{equation}
\Gamma _{\mathbf{ij}}^{k,n}\left( x\right) =\Gamma _{\sigma \left( \mathbf{i}%
\right) \sigma \left( \mathbf{j}\right) }^{k,n}\left( x\right) ,\text{ for
any }n\in \mathbb{N},k=1,...,K,x\in \mathcal{S}_{n}\text{ and }\sigma \in 
\mathbb{S}_{k}.  \label{perm}
\end{equation}

\subsection{Dynamics of the Empirical Measure Process}

If the initial configuration $X^n(0) = (X^{1,n}(0), \ldots, X^{n,n}(0))$ is
exchangeable, then it is clear that at any time $t$, the configuration $%
X^n(t)$ of the $n$-particle system described above is also exchangeable, and
thus essential features of its state at that time can be described by the
empirical measure $\mu ^{n}(t)$. We now identify the generator $\mathcal{L}%
_{n}$ of the empirical measure process $\left\{ \mu ^{n}\left( t\right)
,t\geq 0\right\} $, which is an $\mathcal{S}_{n}$-valued c\`{a}dl\`{a}g jump
Markov process. Let $\left\{ e_{i},i=1,...,d\right\} $ represent the
standard basis of $\mathbb{R}^{d}$. When $K=1$, the possible jump directions
of $\mu ^{n}\left( \cdot \right) $ lie in the set $\frac{1}{n}\mathcal{V}%
_{1} $, where $\mathcal{V}_{1} \doteq \left\{ e_{j}-e_{i},\left( i,j\right)
\in \mathcal{J}^{1}\right\} $. Moreover, the number of particles in state $i$
when the empirical measure is equal to $x\in \mathcal{S}_{n}$ is $nx_{i}$.
Hence, the jump rate of $\mu ^{n}\left( \cdot \right) $ in the direction $%
\frac{1}{n}\left( e_{j}-e_{i}\right) $ is $nx_{i}\Gamma _{ij}^{1,n}\left(
x\right) $, and $\mathcal{L}_{n}$ takes the form%
\begin{equation}  \label{gen-single}
\mathcal{L}_{n}\left( f\right) \left( x\right) =n\sum_{\left( i,j\right) \in 
\mathcal{J}^{1}}x_{i}\Gamma _{ij}^{1,n}\left( x\right) \left[ f\left( x+%
\frac{1}{n}\left( e_{j}-e_{i}\right) \right) -f\left( x\right) \right]
\end{equation}%
for any function $f:\mathcal{S}_{n}\mapsto \mathbb{R}$.

In the general case of simultaneous transitions with $K\in \mathbb{N}$, for
fixed $1\leq k\leq K$, $\mathbf{i}=\left\{ i_{1},...,i_{k}\right\} \in 
\mathcal{X}^{k}$, $n\in \mathbb{N}$, $n\geq K$, and $x\in \mathcal{S}_{n}$,
define $A_{k}\left( n,\mathbf{i},x\right) $ to be the number of ordered $k$%
-tuples of particles with configuration $\mathbf{i}=\left\{
i_{1},...,i_{k}\right\} $ when the empirical measure of the $n$-particle
system is $x$. In other words, $A_{k}\left( n,\mathbf{i},x\right) $ is the
number of ordered $k$-tuples $\{r_{1},\ldots ,r_{k}\}\subset \{1,\ldots ,n\}$
such that $(i^\prime_{r_{1}},\ldots ,i^\prime_{r_{k}})=(i_{1},\ldots ,i_{k})$%
, so that the $l^{th}$ particle in the $k$-tuple is in state $i_{l}$, for
some configuration $\mathbf{i}^\prime=(i^\prime_{1},\ldots ,i^\prime_{n})\in 
\mathcal{X}^{n}$ of the $n $-particle system whose empirical measure is $x$: 
$\frac{1}{n}\sum_{l=1}^{n}\mathbb{I}_{\{i^\prime_{l}=m\}}=x_{m}$ for $%
m=1,\ldots ,d$. It is easily seen that this quantity depends on $\mathbf{i}%
^\prime$ (and hence, $\mathbf{i}$) only through the empirical measure $x$
and takes the form 
\begin{equation}
A_{k}\left( n,\mathbf{i},x\right) =n^{k}\displaystyle\prod\limits_{l=1}^{k}
x_{i_{l}}+O\left( n^{k-1}\right) ,  \label{a}
\end{equation}%
where the error term is non-zero precisely when the states $\left\{
i_{l}\right\} _{l=1}^{k}$ are not all distinct.

For $k=1,...,K$ and $\mathbf{i}=(i_{1},\ldots ,i_{k})\in \mathcal{X}^{k}$,
denote $e_{\mathbf{i}}\doteq \sum_{l=1}^{k}e_{i_{l}}$. Also, recall that ${%
\mathcal{J}}\doteq \cup _{k=1}^{K}{\mathcal{J}}^{k}$ with ${\mathcal{J}}^{k}$
defined by \eqref{def-kjumpset}, and set 
\begin{equation*}
\mathcal{V}\doteq \left\{ e_{\mathbf{j}}-e_{\mathbf{i}}\text{: }\left( 
\mathbf{i},\mathbf{j}\right) \in \mathcal{J}\right\} .
\end{equation*}%
We call $v=e_{\mathbf{j}}-e_{\mathbf{i}}$ the jump direction associated with
the transition $\left( \mathbf{i},\mathbf{j}\right) \in {\mathcal{J}}$. In
what follows, $|B|$ denotes the cardinality of a set $B$.

\begin{lemma}
\bigskip The generator of the Markov process $\mu ^{n}\left( \cdot \right)$
is given by%
\begin{equation}
\mathcal{L}_{n}\left( f\right) \left( x\right) =n\sum_{k=1}^{K}\sum_{\left( 
\mathbf{i},\mathbf{j}\right) \in \mathcal{J}^{k}}\alpha _{\mathbf{ij}%
}^{k,n}\left( x\right) \left[ f\left( x+\frac{1}{n}e_{\mathbf{j}}-\frac{1}{n}%
e_{\mathbf{i}}\right) -f\left( x\right) \right]  \label{Ln}
\end{equation}%
for any function $f:\mathcal{S}_{n}\mapsto \mathbb{R}$, with 
\begin{equation}
\alpha _{\mathbf{ij}}^{k,n}\left( x\right) \doteq \frac{1}{n(k!)}A_{k}\left(
n,\mathbf{i},x\right) \Gamma _{\mathbf{ij}}^{k,n}\left( x\right) ,\text{ }%
x\in \mathcal{S}_{n}\text{.}  \label{alphan}
\end{equation}
Alternatively, the generator can be rewritten as 
\begin{equation}  \label{Lnl}
{\mathcal{L}}_n \left( f\right) \left( x\right) =n \sum_{v \in {\mathcal{V}}%
} \lambda_v^n (x) \left[ f\left(x + \frac{1}{n}v \right) - f(x) \right],
\end{equation}
where 
\begin{equation}  \label{lambdan}
\lambda_v^n (x) \doteq \sum_{k=1}^K \sum_{ \overset{(\mathbf{i}, \mathbf{j})
\in {\mathcal{J}}^k:}{e_{\mathbf{j}} - e_{\mathbf{i}} = v}} \alpha_{\mathbf{i%
} \mathbf{j}}^{k,n} (x),
\end{equation}
with $\mathcal{J}^k$ given by \eqref{def-kjumpset}.
\end{lemma}

\begin{proof}
Fix $k\in \{1,\ldots ,K\}$ and define an equivalence relation on $\mathcal{J}%
^{k}$ as follows: for $\left( \mathbf{i}_{1}\mathbf{,j}_{1}\right) ,\left( 
\mathbf{i}_{2}\mathbf{,j}_{2}\right) \in \mathcal{J}^{k}$, $\left( \mathbf{i}%
_{1}\mathbf{,j}_{1}\right) \sim \left( \mathbf{i}_{2}\mathbf{,j}_{2}\right) $
if and only if there exists $\sigma \in \mathbb{S}_{k}$ such that $\sigma
\left( \mathbf{i}_{1}\right) =\mathbf{i}_{2}$, $\sigma \left( \mathbf{j}%
_{1}\right) =\mathbf{j}_{2}$. Let $\left[ \mathbf{i,j}\right] $ denote the
equivalence class containing $\left( \mathbf{i,j}\right) $, let $\left[ 
\mathcal{J}^{k}\right] $ denote the collection of equivalence classes, and
define $\mathbb{S}_{k}\left[ \mathbf{i,j}\right] =\left\{ \sigma \in \mathbb{%
S}_{k}:\sigma \left( \mathbf{i}\right) =\mathbf{i}\text{, }\sigma \left( 
\mathbf{j}\right) =\mathbf{j}\right\} $. Since the particles are assumed
indistinguishable, when $\left( \mathbf{i}_{1}\mathbf{,j}_{1}\right) \sim
\left( \mathbf{i}_{2}\mathbf{,j}_{2}\right) $, the jump direction associated
with $\left( \mathbf{i}_{1}\mathbf{,j}_{1}\right) $ coincides with that
associated with $\left( \mathbf{i}_{2}\mathbf{,j}_{2}\right) $. Therefore,
when the empirical measure of the $n$-particle system is $x\in {\mathcal{S}}%
_{n}$, given $\left( \mathbf{i},\mathbf{j}\right) \in \mathcal{J}^{k}$, the
number of distinguishable ordered $k$-tuple transitions from configuration $%
\mathbf{i}$ to $\mathbf{j}$ is equal to $A_{k}\left( n,\mathbf{i},x\right)
/\left\vert \mathbb{S}_{k}\left[ \mathbf{i,j}\right] \right\vert $, where $%
A_k(n,\mathbf{i}, x)$ satisfies \eqref{a}. By the permutation symmetry (\ref%
{perm}), we can set $\Gamma _{\left[ \mathbf{i,j}\right] }^{k,n}\left( \cdot
\right) =\Gamma _{\mathbf{ij}}^{k,n}\left( \cdot \right) $, and the
generator of the Markov process $\left\{ \mu ^{n}\left( \cdot \right)
\right\} $ is given by 
\begin{equation}
\mathcal{L}_{n}\left( f\right) \left( x\right) =\sum_{k=1}^{K}\sum_{\left[ 
\mathbf{i,j}\right] \in \left[ \mathcal{J}^{k}\right] }\frac{A_{k}\left( n,%
\mathbf{i},x\right) }{\left\vert \mathbb{S}_{k}\left[ \mathbf{i,j}\right]
\right\vert }\Gamma _{\left[ \mathbf{i,j}\right] }^{k,n}\left( x\right) %
\left[ f\left( x+\frac{1}{n}e_{\mathbf{j}}-\frac{1}{n}e_{\mathbf{i}}\right)
-f\left( x\right) \right] ,\text{ }x\in \mathcal{S}_{n},  \label{Lgrp}
\end{equation}%
for any function $f:\mathcal{S}_{n}\mapsto \mathbb{R}$. An alternative way
to write the generator (\ref{Lgrp}) is as a sum over $\mathcal{J}^{k}$
rather than over $\left[ \mathcal{J}^{k}\right] $. Using \eqref{perm} and
noting that $\left\vert \left[ \mathbf{i,j}\right] \right\vert =\left\vert 
\mathbb{S}_{k}\right\vert /\left\vert \mathbb{S}_{k}\left[ \mathbf{i,j}%
\right] \right\vert =k!/\left\vert \mathbb{S}_{k}\left[ \mathbf{i,j}\right]
\right\vert $, we can rewrite (\ref{Lgrp}) as in (\ref{Ln}), since the sum
in (\ref{Lgrp}) for a given $\left( \mathbf{i,j}\right) $ corresponds to $%
\left\vert \left[ \mathbf{i,j}\right] \right\vert $ summands in (\ref{Ln}).
Finally, \eqref{Lnl} is a direct consequence of \eqref{Ln} and the
definition of $\lambda _{v}^{n}(\cdot )$ in \eqref{lambdan}.
\end{proof}

We will refer to $\lambda_v^n$ as the jump rate (of the empirical measure $%
\mu^n$) in the direction $v$.

\subsection{The Law of Large Numbers Limit}

\label{subs-lln}

We now describe the functional LLN limit for the sequence of jump Markov
processes $\{\mu ^{n}\}_{n\in \mathbb{N}}$ under a suitable assumption on
the particle transition rates.

\begin{assumption}
\label{intsys} For every $k=1,...,K$ and $\left( \mathbf{i,j}\right) \in 
\mathcal{J}^{k}$, there exists a Lipschitz continuous function $\Gamma _{%
\mathbf{ij}}^{k}:\mathcal{S}\rightarrow \mathbb{R}$ such that for every $%
x\in \mathcal{S}$ and sequence $x_{n}\in \mathcal{S}_{n}$, $n\in \mathbb{N}$%
, such that $\lim_{n\rightarrow \infty }x_{n}=x$, 
\begin{equation}  \label{gamma-limit}
\Gamma _{\mathbf{ij}}^{k}\left( x\right) =\lim_{n\rightarrow \infty
}n^{k-1}\Gamma _{\mathbf{ij}}^{k,n}\left( x_{n}\right).
\end{equation}
\end{assumption}


Note that Assumption \ref{intsys} implies that the transition rates are
uniformly bounded: 
\begin{equation}
R_{0}\doteq \max_{k=1,\ldots ,K}\max_{(\mathbf{i},\mathbf{j})\in {\mathcal{J}%
}^{k}}\max_{x\in {\mathcal{S}}}\Gamma _{\mathbf{i}\mathbf{j}}^{k}(x)<\infty ,
\label{def-R0}
\end{equation}%
and that the associated jump rates $\{\lambda _{v}^{n},v\in {\mathcal{V}}\}$
of the empirical measure process $\mu ^{n}(\cdot )$, given by \eqref{lambdan}%
, satisfy the following property.

\begin{property}
\label{prop-lambda} For every $v\in {\mathcal{V}}$, there exists a Lipschitz
continuous function $\lambda _{v}:{\mathcal{S}}\rightarrow \lbrack 0,\infty
) $ such that given any sequence $x_{n}\in {\mathcal{S}}_{n}$, $n\in \mathbb{%
N} $, such that $x_{n}\rightarrow x\in {\mathcal{S}}$ as $n\rightarrow
\infty $, $\lambda _{v}^{n}(x_{n})\rightarrow \lambda _{v}(x)$.
\end{property}

To see why this is true, for every $k=1,\ldots ,K$ and $\left( \mathbf{i,j}%
\right) \in \mathcal{J}^{k}$, define $\alpha _{\mathbf{ij}}^{k}\left( \cdot
\right) $ by 
\begin{equation}
\alpha _{\mathbf{ij}}^{k}\left( x\right) \doteq \frac{1}{k!}\left(
\prod\limits_{l=1}^{k}x_{i_{l}}\right) \Gamma _{\mathbf{ij}}^{k}\left(
x\right) ,\quad x\in {\mathcal{S}}.  \label{alpha1}
\end{equation}%
If Assumption \ref{intsys} holds, then \eqref{alphan} and \eqref{a} together
imply that $\alpha _{\mathbf{ij}}^{k}$ is Lipschitz continuous and $\alpha _{%
\mathbf{ij}}^{k}\left( x\right) =\lim_{n\rightarrow \infty }\alpha _{\mathbf{%
ij}}^{k,n}\left( x_{n}\right) $ for $x\in \mathcal{S}$. Together with %
\eqref{lambdan}, this shows that Property \ref{prop-lambda} is satisfied
with 
\begin{equation}
\lambda _{v}\left( x\right) \doteq \sum_{k=1}^{K}\sum_{\substack{ \left( 
\mathbf{i},\mathbf{j}\right) \in \mathcal{J}^{k}\text{:}  \\ e_{\mathbf{j}%
}-e_{\mathbf{i}}=v}}\alpha _{\mathbf{ij}}^{k}\left( x\right) ,\quad x\in {%
\mathcal{S}},  \label{lambda}
\end{equation}%
for $v\in {\mathcal{V}}$.

For future purposes, we also define 
\begin{equation}
R\doteq \sup_{v\in \mathcal{V},x\in \mathcal{S}}\lambda _{v}\left( x\right)
<\infty ,  \label{m}
\end{equation}%
where $R$ is finite because ${\mathcal{S}}$ is compact and the rates $%
\lambda _{v}(\cdot ),v\in {\mathcal{V}}$, are continuous. Since the jump
rates $\{\lambda _{v}(\cdot ),v\in {\mathcal{V}}\}$ satisfy Property \ref%
{prop-lambda}, the LLN limit for $\{\mu ^{n}\}_{n\in \mathbb{N}}$ follows
from a general result due to \cite{Kur} (see also \cite{Oel}).

\begin{theorem}
\label{2.1} Suppose that the sequence $\{\lambda_v^n (\cdot), v \in {%
\mathcal{V}}\}$ of jump rates associated with the sequence of empirical
measure processes $\{\mu^n(\cdot)\}_{n \in \mathbb{N}}$ satisfies Property %
\ref{prop-lambda}, and let $\lambda_v, v \in \mathcal{V}$, be the associated
limit jump rates defined in \eqref{lambda}. Also, assume $\mu^{n}\left(
0\right) $ converges in probability to $\mu _{0}\in {\mathcal{P}}\left( 
\mathcal{X}\right) $ as $n$ tends to infinity. Then $\left\{ \mu ^{n}\left(
\cdot \right) \right\} _{n\in \mathbb{N}}$ converges (uniformly on compact
time intervals) in probability to $\mu \left( \cdot \right) $, where $\mu
\left( \cdot \right) $ is the unique solution to the nonlinear Kolmogorov
forward equation%
\begin{equation}
\dot{\mu}\left( t\right) =\sum_{v\in \mathcal{V}}v\lambda _{v}\left( \mu
\left( t\right) \right) ,\quad \mu \left( 0\right) =\mu _{0}.  \label{1}
\end{equation}
In particular, the above assertion holds when the sequence of transition
rates $\{\Gamma _{\mathbf{ij}}^{k,n}:\left( \mathbf{i,j}\right) \in \mathcal{%
J}^k, k = 1, \ldots, K\}_{n \in \mathbb{N}}$ satisfies Assumption \ref%
{intsys} and $\mu^{n}\left( 0\right) $ converges in probability to $\mu
_{0}\in {\mathcal{P}}\left( \mathcal{X}\right) $ as $n$ tends to infinity.
\end{theorem}

Since properties of the LLN trajectory will be used in the large deviation
proof, we present an alternative proof of Theorem \ref{2.1} in Section \ref%
{subs-pflln}. In the single jump case ($K=1$), substituting \eqref{alpha1}
and \eqref{lambda} into \eqref{1} and rearranging terms, it is easy to see
that the nonlinear ODE describing the LLN limit can be rewritten in the form 
\begin{equation}  \label{lln-single}
\dot{\mu}\left( t\right) = \mu (t) \Gamma \left( \mu(t)\right),\quad \mu
\left( 0\right) =\mu _{0}.
\end{equation}
where $\Gamma(\cdot) = \Gamma^1 (\cdot)$ is the transition rate matrix $%
\{\Gamma_{ij}(\cdot), i, j = 1, \ldots, d\}$. We now show that the LLN limit
of the empirical measure of an interacting particle system with $K > 1$ can
be viewed as the LLN limit of the empirical measure of a corresponding
particle system with no simultaneous transitions (i.e., with $K=1$).

\begin{remark}
\label{rem-singlejmp} \emph{Given a jump Markov process with generator (\ref%
{Ln}), consider the associated \textquotedblleft single
transition\textquotedblright\ interacting particle process, with transition
rate matrix 
\begin{equation}
\Gamma _{ij}^{n,\mathrm{eff}}\left( x\right) \doteq
\sum_{k=1}^{K}\sum_{\left( \mathbf{i},\mathbf{j}\right) \in \mathcal{J}%
^{k}}\sum_{l=1}^{k}\frac{\alpha _{\mathbf{ij}}^{k,n}\left( x\right) }{x_{i}}%
\mathbb{I}_{\left\{ i=i_{l},j=j_{l}\right\} },\text{ \ \ \ }x\in \mathcal{S}%
,\left( i,j\right) \in \mathcal{J}^{1},  \label{Geff}
\end{equation}%
for $i\neq j$ and $\Gamma _{ii}^{n,\mathrm{eff}}\left( x\right) \doteq
-\sum_{j=1,j\neq i}^{d}\Gamma _{ij}^{n,\mathrm{eff}}\left( x\right) $, $n\in 
\mathbb{N}$. In \eqref{Geff}, for $\mathbf{i}=(i_{1},\ldots ,i_{k})$, if $%
x_{i_{\ell }}=0$ for some $l\in \{1,\ldots ,k\}$, then $\alpha _{\mathbf{ij}%
}^{k,n}\left( x\right) /x_{i_{l}}$ is understood as the pointwise limit of $%
\alpha _{\mathbf{ij}}^{k,n}\left( y\right) /y_{i_{l}}$ when $y$ lies in the
relative interior of $\mathcal{S}$ and $y\rightarrow x$ in the Euclidean
norm; the form of $\alpha _{\mathbf{ij}}^{k,n}(\cdot )$ in \eqref{alphan}
and (\ref{a}) guarantees the existence of this pointwise limit. From
Assumption \ref{intsys} and \eqref{alpha1}, it is clear that for each $x\in 
\mathcal{S}$ and $i,j\in {\mathcal{X}}$, $i\neq j$, as $n\rightarrow \infty $%
, $\Gamma _{ij}^{n,\mathrm{eff}}\left( x\right) $ converges to 
\begin{equation}
\Gamma _{ij}^{\mathrm{eff}}\left( x\right) \doteq \sum_{k=1}^{K}\sum_{\left( 
\mathbf{i},\mathbf{j}\right) \in \mathcal{J}^{k}}\sum_{l=1}^{k}\left(
\prod\limits_{\substack{ r=1  \\ r\neq l}}^{k}x_{i_{r}}\right) \Gamma _{%
\mathbf{ij}}^{k}\left( x\right) \mathbb{I}_{\left\{ i=i_{l},j=j_{l}\right\}
},  \label{GEFF}
\end{equation}%
where a product over an empty set is to be interpreted as $1$. If, as usual,
we set $\Gamma _{ii}^{\mathrm{eff}}\doteq -\sum_{j\neq i,j\in {\mathcal{X}}%
}\Gamma _{ij}^{\mathrm{eff}}$, then it is easy to see that the LLN limit for
the simultaneous transitions case, which has the form \eqref{1} with $%
\{\lambda _{v}(\cdot ),v\in {\mathcal{V}}\}$ as in \eqref{lambda}, coincides
with the LLN limit for the single transition case in \eqref{lln-single}, but
with the matrix $\Gamma $ replaced by the matrix $\Gamma ^{\mathrm{eff}}$.
The superscript \textquotedblleft eff\textquotedblright\ in \eqref{Geff} and %
\eqref{GEFF} stands for \textquotedblleft effective\textquotedblright , and
is used to indicate that the $n$-particle system with simultaneous
transitions and the corresponding single-transition particle system have the
same LLN limit. However, it is important to note that the two systems have
different dynamics and large deviation behavior (for instance, see Example
3.1.26 of \cite{WW}). }
\end{remark}

\subsection{Examples}

\label{subs-egs}

The particle systems that we describe naturally occur in a wide range of
areas, including statistical mechanics (Curie-Weiss model), graphical models
and algorithms, networks and queueing systems (rerouting, loss networks). We
present two illustrative examples below.

\begin{example}
\label{CW}\textit{The opinion dynamics or Curie-Weiss model \cite{DM}.} 
\emph{This is a mean field model on a complete graph. As before, let $n$ be
the number of particles or individuals and let $X^{i,n}\left( t\right) \in {%
\mathcal{X}}\doteq \left\{ -1,1\right\} $ denote the opinion of the $i^{%
\text{th}}$ individual at time $t$, and let $\beta > 0$ be a parameter that
measures the proclivity of an individual to change opinion. (Note that $d
\doteq |{\mathcal{X}}| = 2$, but we write ${\mathcal{X}} = \{-1, 1\}$
instead of ${\mathcal{X}} = \{1,2\}$.) At time $0$, each individual adopts
an opinion $X^{i,n}\left( 0\right) $ in ${\mathcal{X}}$ independently and
uniformly at random. Each individual has an independent and identically
distributed (iid) Poisson clock of rate $1$. If the clock of individual $i$
rings at time $t$, he/she computes the opinion imbalance $M^{\left( i\right)
}=\sum_{j\neq i}X^{j,n}\left( t-\right) $, and changes opinion with
probability 
\begin{equation*}
\mathbb{P}_{\mathrm{flip}}\left( X^{i,n}\left( t\right) \right) =\left\{ 
\begin{array}{ll}
\exp \left( -2\beta \left\vert M^{\left( i\right) }(t-)\right\vert /n\right)
& \text{if }M^{\left( i\right) }(t-)X^{i,n}\left( t-\right) >0, \\ 
1 & \text{otherwise.}%
\end{array}%
\right.
\end{equation*}%
The empirical measure process $\mu ^{n}$ only takes jumps of the form $%
\mathcal{V}=\left\{ e_{j}-e_{i},i,j\in \left\{ -1,1\right\} \right\} ,$ with
points in $\mathcal{S}$ denoted by }$(x_{-1},x_{1})$\emph{. The particle
transition rates satisfy Assumption \ref{intsys} with $\Gamma =\Gamma ^{1}$
taking the form 
\begin{eqnarray*}
\Gamma _{1,-1}\left( x\right) &=&\left\{ 
\begin{array}{ll}
\exp \left( -2\beta \left( x_{1}-x_{-1}\right) \right) & \text{if }%
x_{1}-x_{-1}>0, \\ 
1 & \text{otherwise,}%
\end{array}%
\right. \\
\Gamma _{-1,1}\left( x\right) &=&\left\{ 
\begin{array}{ll}
\exp \left( -2\beta \left( x_{-1}-x_{1}\right) \right) & \text{if }%
x_{1}-x_{-1}<0, \\ 
1 & \text{otherwise,}%
\end{array}%
\right.
\end{eqnarray*}%
and, as usual, $\Gamma _{1,1}\left( x\right) =-\Gamma _{1,-1}\left( x\right) 
$ and $\Gamma _{-1,-1}\left( x\right) =-\Gamma _{-1,1}\left( x\right) $. A
generalization of this example is the Curie-Weiss-Potts model with Glauber
dynamics, the mixing time of which has interesting phase transition
properties (see \cite{LLP}). }
\end{example}

Interacting particle systems with simultaneous transitions arise naturally
as models of communication networks. We now provide one such example, from 
\cite{GibHunKel90}. More examples can be found in \cite{T}, \cite{M} and 
\cite{GOc}.

\begin{example}
\label{ARN}\textit{Alternative rerouting networks }\cite{GibHunKel90}. \emph{%
Consider a network that consists of $n$ links, each with finite capacity $C$%
. Let $\mathcal{X}=\left\{ 0,\ldots ,C\right\} $ and let $X^{i,n}(t)$ denote
the number of packets (or customers) using link $i$ at time $t$. Packets
arrive at each link as a Poisson process with rate $\gamma >0$. If a packet
arrives at a link with spare capacity, then it is accepted to the link and
occupies one unit of capacity for an exponentially distributed time with
mean one. On the other hand, if a packet arrives at a link that is fully
occupied, two other links are chosen uniformly at random from amongst the
remaining $n-1$ links. If both chosen links have a unit of spare capacity
available, the packet occupies one unit of capacity on each of the two
links, for two independent, exponential clocks with mean one. Otherwise, the
packet is lost. This model seeks to understand the impact of allowing
alternative routes that occupy a greater number of resources on the
performance of the network. }

\emph{The empirical measure process $\mu ^{n}$ is a jump Markov process with
jump rates summarized as follows: for any $i,j\in \mathcal{X}$:%
\begin{equation*}
\mu ^{n}\rightarrow \left\{ 
\begin{array}{lll}
\mu ^{n}+\frac{1}{n}\left( e_{i+1}-e_{i}\right) & \text{at rate }n\gamma \mu
_{i}^{n} & 0\leq i\leq C-1, \\ 
\mu ^{n}+\frac{1}{n}\left( e_{i-1}-e_{i}\right) & \text{at rate }ni\mu
_{i}^{n} & 1\leq i\leq C, \\ 
\mu ^{n}+\frac{1}{n}\left( e_{i+1}-e_{i}+e_{j+1}-e_{j}\right) & \text{at
rate }2\gamma \frac{n^{3}\mu _{c}^{n}\mu _{i}^{n}\mu _{j}^{n}}{\left(
n-1\right) \left( n-2\right) } & 0\leq i\neq j\leq C-1, \\ 
\mu ^{n}+\frac{2}{n}\left( e_{i+1}-e_{i}\right) & \text{at rate }\gamma 
\frac{n^{2}\mu _{c}^{n}\mu _{i}^{n}\left( n\mu _{i}^{n}-1\right) }{\left(
n-1\right) \left( n-2\right) } & 0\leq i\leq C-1.%
\end{array}%
\right.
\end{equation*}%
The transition rates of the particle system satisfy Assumption \ref{intsys}
with $K=2$, and 
\begin{equation*}
\Gamma _{ii+1}^{1}\left( x\right) =\gamma ,\text{ \ \ \ }\Gamma
_{ii-1}^{1}\left( x\right) =i,\text{ \ \ \ }\Gamma _{\left( i,j\right)
\left( i+1,j+1\right) }^{2}\left( x\right) =\gamma x_{c},
\end{equation*}%
and $\Gamma _{\mathbf{ij}}^{k}=0$ for all other transitions $(\mathbf{i},%
\mathbf{j})\in {\mathcal{J}}^{k}$, $k=1,2$. By (\ref{GEFF}), we can
calculate the effective transition rate as 
\begin{equation*}
\Gamma _{ij}^{\mathrm{eff}}\left( x\right) =\Gamma _{ij}^{1}\left( x\right)
+\sum_{\substack{ i^{\prime }\neq i,j^{\prime }\neq j  \\ i^{\prime }\neq
j^{\prime }}}2x_{i^{\prime }}\Gamma _{\left( i,i^{\prime }\right) \left(
j,j^{\prime }\right) }^{2}\left( x\right) +x_{i}\Gamma _{\left( i,i\right)
\left( j,j\right) }^{2}\left( x\right) ,
\end{equation*}%
which gives $\Gamma _{ii+1}^{\mathrm{eff}}\left( x\right) =\gamma
+4\sum_{i^{\prime }\neq i}x_{i^{\prime }}\lambda x_{c}+2x_{i}\gamma
x_{c}=\gamma \left[ 1+2x_{c}\left( 2-x_{i}\right) \right] $, $\Gamma
_{ii-1}^{\mathrm{eff}}\left( x\right) =\gamma$, and $\Gamma _{ij}^{\mathrm{%
eff}}\left( x\right) =0$ for all other $(i,j)\in {\mathcal{J}}^{1}$. }
\end{example}

\section{Main Results}

\label{sec-mainres}

Throughout the rest of the paper, we always assume, without explicit
mention, that the transition rates associated with the sequence of $n$%
-particle systems satisfy the symmetry condition \eqref{perm}. We also
assume that they satisfy Assumption \ref{intsys} with associated limit
transition rates $\{\Gamma _{\mathbf{ij}}^{k}(\cdot ),(\mathbf{i,j)}\in {%
\mathcal{J}}^{k},k=1,\ldots ,K\}$. Then, as follows from Theorem \ref{2.1},
the corresponding sequence of empirical measure processes $\{\mu
^{n}\}_{n\in \mathbb{N}}$ has a LLN limit $\mu \left( \cdot \right) $ whose
evolution is governed by the limit jump rates $\{\lambda _{v},v\in {\mathcal{%
V}}\},$ defined in \eqref{lambda}. In practice one is often interested in
estimating the tail probabilities $\mathbb{P}\left( \mu ^{n}\left( \cdot
\right) \in A\right) $ for certain sets of paths $A$ that do not contain the
LLN limit. This can be studied in the framework of an LDP. First, in Section %
\ref{subs-nparticle} we introduce additional assumptions on the limit
transition rate functions and then in Section \ref{subs-ldp} state the
sample path large deviation principle for the sequence $\left\{ \mu
^{n}\right\} _{n\in \mathbb{N}}$. Asymptotics of the tail probabilities at a
given time $t$ will follow from the contraction principle. In Section \ref%
{subs-luldp}, we introduce an additional condition that allows us to
establish a locally uniform refinement to the LDP and in Section \ref%
{subs-invldp} we discuss the LDP for the associated sequence of invariant
measures. As a by-product of our proof technique, we in fact establish these
large deviation results for a larger class of sequences $\{\mu ^{n}\}_{n\in 
\mathbb{N}}$ of jump Markov processes on the simplex. A precise statement of
this more general result is given in Remark \ref{rem-ldp}. For simplicity we
assume from now on that $t\in \left[ 0,1\right] $, while all results in this
paper can be established for $t$ in any compact time interval by the same
argument.

\subsection{Assumptions on the Limit Transition Rates}

\label{subs-nparticle}

Below, we introduce three additional assumptions on the limit transition
rates of the interacting particle system: a uniformity condition (Assumption %
\ref{ue}), a type of ergodicity (Assumption \ref{ass-kerg}) and a mild
restriction on the type of simultaneous jumps allowed (Assumption \ref%
{ass-simjumps}). For $k\in \{1,\ldots ,K\}$, denote 
\begin{equation}
\mathfrak{M}_{\mathbf{ij}}^{k}\doteq \inf_{x\in \mathcal{S}}\Gamma _{\mathbf{%
ij}}^{k}\left( x\right) \text{,}\quad \mbox{ for }\left( \mathbf{i,j}\right)
\in \mathcal{J}^{k},  \label{n}
\end{equation}%
and let the set 
\begin{equation}
\mathcal{J}^{k}_+\doteq \left\{ \left( \mathbf{i,j}\right) \in \mathcal{J}%
^{k}:\mathfrak{M}_{\mathbf{ij}}^{k}>0\right\}  \label{nk}
\end{equation}
denote the set of $k$-tuple transitions whose transition rates are uniformly
bounded away from zero. Also, set 
\begin{equation}
c_{0}\doteq \min_{k=1,\ldots ,K}\left\{ \mathfrak{M}_{\mathbf{ij}%
}^{k}:\left( \mathbf{i,j}\right) \in {\mathcal{J}}^{k}_+\right\} .
\label{def-c0}
\end{equation}

The first assumption states that each transition rate function is either
identically zero, or uniformly bounded below away from zero on the simplex.

\begin{assumption}
\label{ue} For $k=1,...,K$ and $\left( \mathbf{i,j}\right) \in \mathcal{J}%
^{k}$, either $\left( \mathbf{i,j}\right) \in \mathcal{J}^{k}_+$
(equivalently, $\mathfrak{M}_{\mathbf{ij}}^{k} > 0$) or $\Gamma _{\mathbf{ij}%
}^{k}\left( x\right) =0$ for every $x\in \mathcal{S}$.
\end{assumption}

Note that, nevertheless, the limit jump rates $\lambda _{v}(\cdot ),v\in {%
\mathcal{V}}$, of the associated sequence of empirical measure processes
will not be bounded away from zero on the simplex. More precisely, for $v\in 
{\mathcal{V}}$, let ${\mathcal{N}}_{v}$ be the set of coordinates of $v$
that are strictly negative: 
\begin{equation}
{\mathcal{N}}_{v}\doteq \left\{ i\in {\mathcal{X}}:\langle v,e_{i}\rangle
<0\right\} .  \label{def-nv}
\end{equation}%
Note that for every $v\in {\mathcal{V}}$, $v\neq 0$, the fact that $%
\sum_{i\in {\mathcal{X}}}v_{i}=0$ implies ${\mathcal{N}}_{v}\neq \emptyset $%
. Now, we claim (and justify below) that $\lambda _{v}(x)\rightarrow 0$
whenever $x_{i}\rightarrow 0$ for any $i\in {\mathcal{N}}_{v}$. Indeed, the
claim can be deduced from the form of $\lambda _{v}(\cdot )$ in %
\eqref{lambda}, the fact that for any $k=1,\ldots ,K$ and $(\mathbf{i},%
\mathbf{j})\in \mathcal{J}^{k}$, we have 
\begin{equation}
e_{\mathbf{j}}-e_{\mathbf{i}}=v\quad \Rightarrow \quad \mbox{ for every
}i\in {\mathcal{N}}_{v},\,|\{l=1,\ldots ,k:i_{l}=i\}|\geq |\langle
v,e_{i}\rangle | \geq 1,  \label{nvinc}
\end{equation}%
and the property that $\alpha _{\mathbf{ij}}^{k}(x)\rightarrow 0$ if $%
x_{i_{l}}\rightarrow 0$ for some $l=1,\ldots ,k$, where the latter assertion
follows from \eqref{alpha1} and the uniform boundedness of $\Gamma _{\mathbf{%
i}\mathbf{j}}^{k}$ on ${\mathcal{S}}$, which is a consequence of the
continuity of $\Gamma _{\mathbf{i}\mathbf{j}}^{k}$ specified in Assumption %
\ref{intsys}.

Next, we impose a type of ergodicity property on the transition rates
specified below.

\begin{definition}
\label{k-acc}For two states $u, w\in \mathcal{X}$, $w$ is said to be $K$-%
\textbf{accessible} from $u$ if there exist $M \in \left\{2,...,d\right\} $
and a sequence of distinct states in $\mathcal{X}$: $%
u=u_{1},u_{2},...,u_{M}=w$, such that for $m=1,...,M-1$, the following three
properties hold:

\begin{enumerate}
\item[(i)] there exist $k_{m}\in \left\{ 1,...,K\right\} $, $\left( \mathbf{i%
}_{m}\mathbf{,j}_{m}\right) \in \mathcal{J}^{k_{m}}$, and $%
l_{m},l_{m}^{^{\prime }}\in \left\{ 1,...,k_{m}\right\} $, such that $%
u_{m}=i_{m,l_{m}}$ and $u_{m+1}=j_{m,l_{m}^{^{\prime }}}$;

\item[(ii)] for $l=1,...,k_{m}$, $i_{m,l}\in \left\{ u_{1},...,u_{m}\right\}$%
;

\item[(iii)] $\mathfrak{M}_{\mathbf{i}_{m}\mathbf{j}_{m}}^{k_{m}}>0$, i.e., $%
(\mathbf{i}_{m},\mathbf{j}_{m})\in \mathcal{J}_+^{k_{m}}$.
\end{enumerate}

We say the family $\{ \Gamma _{\mathbf{ij}}^k \left( \cdot \right), \left( 
\mathbf{i},\mathbf{j}\right) \in \mathcal{J}^{k}, k =1, \ldots, K \}$ is $K$-%
\textbf{ergodic} if for any $u, w\in \mathcal{X}$, $w$ is $K$-accessible
from $u$.
\end{definition}

The $K$-ergodicity condition, roughly speaking, requires that one can reach
any state $w \in {\mathcal{X}}$ from any state $u \in {\mathcal{X}}$ via a
finite sequence of states, where each adjacent pair of states represents a
state transition that can be effected by a simultaneous $k$-tuple transition
with a strictly positive rate. Note that in general, the adjacent pair need
not represent initial and final states of any one particle involved in the $%
m $th simultaneous transition; the latter is true only when $l_m =
l^\prime_m $ in Definition \ref{k-acc}.i), which in particular always holds
when $K=1$. Instead, the first state in the pair could be the initial state
of one particle and the other state could be the final state of another
particle involved in the simultaneous transition. However, as stipulated in
property ii) above, $K$-ergodicity also requires that the initial states of
all particles involved in the $m$th (simultaneous) transition must be a
subset of the previous states $u_1, \ldots, u_m$ in the sequence. The latter
property, which is trivially satisfied when $K=1$, ensures that at the $m$th
stage ``mass'' is moved exclusively from the subset of states $\left\{
u_{1},...,u_{m}\right\}$ to $u_{m+1}$, which helps in the construction of
so-called communicating paths for the associated empirical measure process
between different states on the simplex (see Definition \ref{def-compath}
and Proposition \ref{ver1}).

\begin{assumption}
\label{ass-kerg} The family $\{ \Gamma _{\mathbf{ij}}^k \left( \cdot
\right), \left( \mathbf{i},\mathbf{j}\right) \in \mathcal{J}^{k}, k =1,
\ldots, K\}$ is $K$-ergodic.
\end{assumption}

To provide further insight into the $K$-ergodicity property, we now state a
simpler, and perhaps more intuitive, condition that (in the presence of
Assumption \ref{ue}) implies $K$-ergodicity. Recall that $\Gamma^{\mathrm{eff%
}}$ is the effective transition rate matrix introduced in \eqref{GEFF}.

\begin{assumption}
\label{g2} For every $x\in \mathcal{S}$, the Markov process on $\mathcal{X}$
with transition rate matrix $\Gamma^{\mathrm{eff}}\left( x\right) $ is
ergodic.
\end{assumption}

\begin{lemma}
\label{lem-verk} If the family $\{ \Gamma^{k} _{\mathbf{ij}}\left(
\cdot\right), \left( \mathbf{i},\mathbf{j}\right) \in {\mathcal{J}}^k, k =
1, \ldots, K\}$ satisfies Assumptions \ref{ue} and \ref{g2}, then it also
satisfies Assumption \ref{ass-kerg}, that is, it is $K$-ergodic.
\end{lemma}

\begin{proof}
Take any $u, w\in \mathcal{X}$, $u\neq w$. Since $\Gamma ^{\mathrm{eff}%
}\left( e_{u}\right) $ is ergodic by Assumption \ref{g2}, there exist $M \in
\{2, \ldots, d\}$ and a sequence of distinct states $u=u_{1},...,u_{M}=w$
such that $\Gamma _{u_{m}u_{m+1}}^{\mathrm{eff}}\left( e_{u}\right) >0$, $%
m=1,...,M-1$. By the definition of $\Gamma ^{\mathrm{eff}}$ given in (\ref%
{GEFF}), this implies that for $m=1,...,M-1$, there exist $k_{m}\in \left\{
1,...,K\right\} $, $\left( \mathbf{i}_{m}\mathbf{,j}_{m}\right) \in \mathcal{%
J}^{k_{m}}$ and $l_{m}\in \left\{ 1,...,k_{m}\right\}$ such that $%
u_{m}=i_{m,l_{m}}$, $u_{m+1}=j_{m,l_{m}}$ and 
\begin{equation*}
\prod\limits_{\substack{ r=1  \\ r\neq l_{m}}}^{k_{m}}\left\langle
e_{u},e_{i_{m,r}}\right\rangle \Gamma _{\mathbf{i}_{m}\mathbf{j}%
_{m}}^{k_{m}}\left( e_{u}\right) >0.
\end{equation*}%
By Assumption \ref{ue}, this implies that $\mathfrak{M}_{\mathbf{i}_{m}%
\mathbf{j}_{m}}^{k_{m}}>0$ and $\mathbf{i}_{m,r}=u$ for every $r\neq l_{m}$.
In other words, the $l_{m}^{\text{th}}$ component of $\mathbf{i}_{m}$ is
equal to $u_{m}$, and all other components are equal to $u$. Therefore,
Definition \ref{k-acc}.i) is satisfied with $l_{m}=l_{m}^{^{\prime }}$,
Definition \ref{k-acc}.ii) is satisfied with $i_{m,l}\in \left\{
u,u_{m}\right\}$ for $l=1,...,k_{m}$ and Definition \ref{k-acc}.iii) also
holds. Since $u, w$ are arbitrary, the lemma follows.
\end{proof}

\begin{remark}
\label{rem-1erg} \emph{Notice that when $K=1$, property (ii) of Definition %
\ref{k-acc} is trivially satisfied and $\Gamma^{\mathrm{eff}}$ coincides
with the single transition rate matrix $\Gamma = \Gamma^1$ in %
\eqref{lln-single}. Thus, when $K=1$ and Assumption \ref{ue} is satisfied, $%
K $-ergodicity, Assumption \ref{g2} and Assumption \ref{g1} below (which
requires that $\Gamma^1(x)$ be ergodic for every $x \in {\mathcal{S}}$) are
all equivalent.}
\end{remark}

However, as the following example illustrates, when $K > 1$, $K$-ergodicity
is strictly weaker than Assumption \ref{g2}.

\begin{example}
\label{eg3}\emph{Let $d=4$, $K=2$, and define the generator of the Markov
process $\left\{ \mu ^{n}\right\} _{n\in \mathbb{N}}$ as in (\ref{Ln}) with $%
\alpha _{\mathbf{ij}}^{k,n}$ defined as in (\ref{alphan}), in terms of $%
\Gamma ^{1,n}$ and $\Gamma ^{2,n}$ given by 
\begin{equation*}
\begin{array}{lll}
\Gamma _{12}^{1,n}\left( x\right) =c_{1}, & \Gamma _{21}^{1,n}\left(
x\right) =c_{2}, & \Gamma _{34}^{1,n}\left( x\right) =c_{3}, \\ 
\Gamma _{43}^{1,n}\left( x\right) =c_{4}, & \Gamma _{\left( 1,2\right)
\left( 3,4\right) }^{2,n}\left( x\right) =\frac{1}{n}c_{5}, & \Gamma
_{\left( 3,4\right) \left( 1,2\right) }^{2,n}\left( x\right) =\frac{1}{n}%
c_{6}%
\end{array}%
\end{equation*}%
with $c_{i}>0,$ $i=1,...,6$, and $\Gamma _{\mathbf{ij}}^{k,n}=0$ for all
other ($\mathbf{i,j)}\in {\mathcal{J}}^{k}$, $k=1,2$, $n\in \mathbb{N}$.
Note that Assumption \ref{intsys} trivially holds with $\Gamma _{\mathbf{ij}%
}^{k}\doteq \Gamma _{\mathbf{ij}}^{k,n}$ for $(\mathbf{i,j})\in {\mathcal{J}}%
^{k}$, $k=1,2$, and Assumption \ref{ue} is also satisfied. Also, the
associated limit jump rates $\{\lambda _{v},v\in {\mathcal{V}}\}$ defined in %
\eqref{lambda} take the form 
\begin{equation*}
\begin{array}{lll}
\lambda _{e_{2}-e_{1}}(x)=x_{1}c_{1},\quad & \lambda
_{e_{4}-e_{3}}(x)=x_{3}c_{3},\quad & \lambda
_{e_{3}+e_{4}-e_{1}-e_{2}}(x)=x_{1}x_{2}c_{5}/2 \\ 
\lambda _{e_{1}-e_{2}}(x)=x_{2}c_{2}, & \lambda _{e_{3}-e_{4}}(x)=x_{4}c_{4},
& \lambda _{e_{1}+e_{2}-e_{3}-e_{4}}(x)=x_{3}x_{4}c_{6}/2.%
\end{array}%
\end{equation*}%
}

\emph{Furthermore, the effective transition rate matrix $\Gamma^{\mathrm{eff}%
}$ defined in \eqref{GEFF} takes the form 
\begin{equation*}
\begin{array}{llll}
\Gamma _{12}^{\mathrm{eff}}\left( x\right) =c_{1}, & \text{\ }\Gamma _{21}^{%
\mathrm{eff}}\left( x\right) =c_{2},\text{ \ } & \text{\ }\Gamma _{34}^{%
\mathrm{eff}}\left( x\right) =c_{3}, & \text{\ }\Gamma _{43}^{\mathrm{eff}%
}\left( x\right) =c_{4} \\ 
\Gamma _{13}^{\mathrm{eff}}\left( x\right) =x_{2}c_{5}, & \text{\ }\Gamma
_{31}^{\mathrm{eff}}\left( x\right) =x_{4}c_{6}, & \text{ }\Gamma _{24}^{%
\mathrm{eff}}\left( x\right) =x_{1}c_{5}, & \text{\ }\Gamma _{42}^{\mathrm{%
eff}}\left( x\right) =x_{3}c_{6}.%
\end{array}%
\end{equation*}
Thus, $\Gamma ^{\mathrm{eff}}$ is not ergodic on the part of the boundary
given by $\left\{ x\in \mathcal{S}:x_{3}=x_{4}=0\text{ or }%
x_{1}=x_{2}=0\right\} $, and Assumption \ref{g2} fails to hold. }

\emph{We now show that nevertheless, this particle system is $2$-ergodic. To
verify the $2$-ergodicity of Example \ref{eg3}, first consider the case $u=1$
and $w\in \{2, 3, 4\}$. If $w = 2$, then we can take $M=2$, $k_{1}=1$ and $%
\left( \mathbf{i}_{1}\mathbf{,j}_{1}\right) =\left( 1,2\right)$. If $w=3$,
one might be tempted to set $M=2$ again and use the simultaneous jump $(1,2)
\mapsto (3,4)$. However, this would violate property (ii) of Definition \ref%
{k-acc}. Instead, we take $M=3$, $u_{1}=1$, $u_{2}=2$, $u_{3}=3$, $k_{1}=1$, 
$\left( \mathbf{i}_{1}\mathbf{,j}_{1}\right) =\left( 1,2\right) $, $k_{2}=2$
and $\left( \mathbf{i}_{2}\mathbf{,j}_{2}\right) =\left( \left( 1,2\right)
,\left( 3,4\right) \right) $. The case $w=4$ is similar to the case $w=3$,
except that $u_{3}=4$. It is easy to check in each case that the sequence of
states $\left\{ u_{m}, m = 1, \ldots, M\right\}$ satisfy conditions i), ii),
iii) of Definition \ref{k-acc}. The symmetry of the problem allows one to
deal with the case $u \in \{2, 3, 4\}$ in an analogous fashion (we omit the
details) to show that the example is $2$-ergodic. }
\end{example}

As explained in the introduction and shown in Section \ref{sec-pfupper}, by
applying a general result for jump-diffusion Markov processes that was
obtained in \cite{DEW}, it is possible to establish a large deviation upper
bound for jump Markov processes with generator \eqref{Lnl} in the form of
the integral of a so-called \textquotedblleft local rate
function\textquotedblright ; see \eqref{I} and \eqref{L} below. The more
delicate part of the sample path LDP is the proof of the large deviation
lower bound. Since each jump rate $\lambda _{v}^{n}(\cdot )$ of $\mu ^{n}$
tends to zero as $\mu ^{n}$ approaches some part of the boundary of $%
\mathcal{S}$, the local rate function can approach infinity, which makes the
analysis difficult. The third assumption we require is a mild technical
restriction on the type of simultaneous particle transitions that are
allowed, which allows us to overcome this difficulty. This assumption is
used only to show that the LLN trajectory moves into the (relative) interior
of the simplex sufficiently quickly (see Property \ref{cond-lln} for a
precise statement). However, as elaborated in Remark \ref{rem-ldp}, our
proof applies to the broader class of systems for which the LLN trajectory
still possesses this property, even if Assumption \ref{ass-simjumps} may
fail to hold. Assumption \ref{ass-simjumps} simply serves to identify a
large class of systems for which this property of the LLN trajectory can be 
\emph{a priori} verified. We recall that $\mathcal{J}_{+}^{k}$ was defined
in (\ref{nk}) and let $\mathcal{J}_{+}\doteq \cup _{k=1}^{K}\mathcal{J}%
_{+}^{k}$.

\begin{assumption}
\label{ass-simjumps} For every $v\in {\mathcal{V}}\backslash \{0\}$ such
that $\lambda _{v}$ is not identically zero, at least one of the following
two properties is true:

\begin{enumerate}
\item There exists $(\mathbf{i}^{\ast },\mathbf{j}^{\ast })\in \mathcal{J}%
_{+}$ such that $e_{\mathbf{j}^{\ast }}-e_{\mathbf{i}^{\ast }}=v$ and 
\begin{equation*}
\left\vert \left\{ l:i_{l}=j\right\} \right\vert =\left\{ 
\begin{array}{ll}
|\langle v,e_{j}\rangle | & \mbox{ if }j\in {\mathcal{N}}_{v}, \\ 
0 & \mbox{ otherwise, }%
\end{array}%
\right.
\end{equation*}%
where we recall that ${\mathcal{N}}_{v} = \{i\in \mathcal{X}:v_{i}<0\}$,

\item There exist $r_{j}\geq 1,j\in {\mathcal{N}}_{v},$ such that given any $%
(\mathbf{i},\mathbf{j})\in \mathcal{J}_{+}$, we have $(\mathbf{i},\mathbf{j}%
)\in {\mathcal{J}}_{+}^{k_{\ast }}$ where $k^{\ast }=\sum_{j\in {\mathcal{N}}%
_{v}}r_{j}$, and 
\begin{equation*}
r_{j}=|\{l=1,\ldots ,k^{\ast }:i_{l}=j\}|,\quad j\in {\mathcal{N}}_v.
\end{equation*}
\end{enumerate}
\end{assumption}

To better understand what this assumption says, consider a particle system
with $d=4$ and suppose $v=2e_{1}+e_{3}-2e_{2}-e_{4}$. There are many
transitions $(\mathbf{i},\mathbf{j})$ that could lead to the jump direction $%
v$, including, for example, (a) $(2,2,4)\mapsto (1,1,3)$; (b) $%
(2,2,4)\mapsto (1,3,1)$, (c) $(2,2,4,4)\mapsto (4,1,3,1)$; (d) $%
(2,2,4,4)\mapsto (4,3,1,1)$; (e) $(2,2,4,1)\mapsto (1,1,1,3)$. Here ${%
\mathcal{N}}_{v}=\{2,4\}$. First consider Assumption \ref{ass-simjumps}(1).
For $j=2$, the number of particles jumping from type $2$ should be $|\langle
v,e_{2}\rangle |=2$. Similarly, the number of particles of type $4$ before
the jump should be $1$. In particular, transitions (a) or (b) above would
meet this requirement, and a system in which all the above transitions [and
their permuted versions, by virtue of \eqref{perm}] have strictly positive
rates would also satisfy the assumption. In contrast, a system in which the
only transitions associated with $v$ that have positive rate are of type (c)
or (d) would not satisfy Assumption \ref{ass-simjumps}(1) because for such
transitions the first inequality in \eqref{nvinc} is strict for $j=4$, thus
violating the stipulated condition. However, this system would satisfy
Assumption \ref{ass-simjumps}(2) since the vector of initial particle values
for both these transitions have the same "type", namely containing a pair of 
$2$'s and a pair of $4$'s, whereas a system that has both transitions (a)
and (c) would not satisfy the second condition in Assumption \ref%
{ass-simjumps} [although, as mentioned above, it would satisfy Assumption %
\ref{ass-simjumps}(1)]. An example of a particle system that would violate
both conditions in Assumption \ref{ass-simjumps} is one in which only
transitions of type (e) have positive rate. In this case, $\left\vert
\left\{ l:i_{l}=j\right\} \right\vert $ is non-zero for $j=1$, which does
not lie in ${\mathcal{N}}_{v}=\{2,4\}$ and, moreover, $(i,j) \in {\mathcal{J}%
}_+^4$, but $\sum_{j \in {\mathcal{N}}_v} |\{l: i_l = j\}| = 3$. The
technical problem with such a system is that one could have $\lambda
_{v}(x)\rightarrow 0$ as one approaches parts of the boundary where $x_{i}=0$
even though $v_{i}>0$. In other words, the jump rates in a certain direction 
$v$ could diminish to zero at certain points on the boundary when jumps in
the direction $v$ from such points would take the empirical measure back to
the interior of the domain. The difficulty is that we rely on such jumps to
move the LLN limit to the interior of the simplex quickly.

\subsection{Large Deviation Principles}

\label{subs-ldp}

We now state our first large deviation result, which is the sample path LDP.
To define the rate function, we need some notation. For $x\in {\mathcal{S}}$%
, let 
\begin{equation}
\ell \left( x\right) \doteq \left\{ 
\begin{array}{ll}
x\log x-x+1 & x\geq 0, \\ 
\infty & x<0,%
\end{array}%
\right.  \label{l}
\end{equation}%
be the local rate function associated with the standard Poisson process. Let 
$\Delta ^{d-1}\doteq \{x\in \mathbb{R}^{d}:\sum_{i=1}^{d}x_{i}=0\}$. Then
for $x\in \mathcal{S}$ and $\beta \in \Delta ^{d-1}$, we define 
\begin{equation}
L\left( x,\beta \right) \doteq \inf_{\overset{q_{v}\geq 0,v\in {\mathcal{V}}:%
}{\sum_{v\in \mathcal{V}}vq_{v}=\beta }}\sum_{v\in \mathcal{V}}\lambda
_{v}\left( x\right) \ell \left( \frac{q_{v}}{\lambda _{v}\left( x\right) }%
\right) \text{.}  \label{L}
\end{equation}%
For $t\in \left[ 0,1\right] $ and an absolutely continuous function $\gamma :%
\left[ 0,t\right] \mapsto \mathcal{S}$, define 
\begin{equation}
I_{t}\left( \gamma \right) \doteq \int_{0}^{t}L\left( \gamma \left( s\right)
,\dot{\gamma}\left( s\right) \right) ds,  \label{I}
\end{equation}%
and in all the other cases, set $I_{t}\left( \gamma \right) =\infty $. We
write $I\left( \gamma \right) $ to denote $I_{1}\left( \gamma \right) $.

In what follows we equip $D\left( \left[ 0,1\right] :\mathcal{S}\right) $
with the Skorokhod $J_1$-topology, and let $\mathcal{B}\left( D\left( \left[
0,1\right] :\mathcal{S}\right) \right) $ be the associated Borel sets.

\begin{theorem}
\label{th-ldips} Suppose the family $\{\Gamma _{\mathbf{i}\mathbf{j}%
}^{k}(x),x\in {\mathcal{S}},(\mathbf{i},\mathbf{j})\in {\mathcal{J}}%
^{k},k=1,\ldots ,K\}$ satisfies Assumptions \ref{intsys}, \ref{ue}, \ref%
{ass-kerg} and \ref{ass-simjumps}. Also, assume that that the initial
conditions $\left\{ \mu ^{n}\left( 0\right) \right\} _{n\in \mathbb{N}}$ are
deterministic and satisfy $\mu ^{n}\left( 0\right) \rightarrow \mu _{0}\in {%
\mathcal{P}}\left( \mathcal{X}\right) $ as $n$ tends to infinity. Then the
associated sequence of empirical measure processes 
$\left\{ \mu ^{n}\right\} _{n\in \mathbb{N}}$ 
satisfies the sample path LDP with rate function $I$. Specifically, for any
measurable set $A\in \mathcal{B}\left( D\left( \left[ 0,1\right] :\mathcal{S}%
\right) \right) $, we have the large deviation upper bound 
\begin{equation}
\limsup_{n\rightarrow \infty }\frac{1}{n}\log \mathbb{P}\left( \mu ^{n}\in
A\right) \leq -\inf \left\{ I\left( \gamma \right) :\gamma \in \bar{A}%
,\gamma \left( 0\right) =\mu _{0}\right\} ,  \label{uppbd}
\end{equation}%
and the large deviation lower bound 
\begin{equation}
\liminf_{n\rightarrow \infty }\frac{1}{n}\log \mathbb{P}\left( \mu ^{n}\in
A\right) \geq -\inf \left\{ I\left( \gamma \right) :\gamma \in A^{\circ
},\gamma \left( 0\right) =\mu _{0}\right\} .  \label{lowbd}
\end{equation}%
Moreover, for any compact set ${\mathcal{K}}\subset \mathcal{S}$ and $%
M<\infty ,$ the set%
\begin{equation}
\left\{ \gamma \in D\left( \left[ 0,1\right] :\mathcal{S}\right) :I\left(
\gamma \right) \leq M,\gamma \left( 0\right) \in {\mathcal{K}}\right\}
\label{ini}
\end{equation}%
is compact.
\end{theorem}

The proof of the upper bound \eqref{uppbd} and the compactness of the set in %
\eqref{ini}, which only uses Property \ref{prop-lambda} (which is implied by
Assumption \ref{intsys}) is given at the end of Section \ref{sec-pfupper}.
The proof of the lower bound \eqref{lowbd} is given at the end of Section %
\ref{sec-pflower}.

Theorem \ref{th-ldips}, together with an application of the contraction
principle (see, e.g., \cite{Var}), yields the following variational
representation for the rate function of $\left\{ \mu ^{n}\left( t\right)
\right\} _{n\in \mathbb{N}}$ for any $t\in \left[ 0,1\right] $.

\begin{corollary}
\label{2.3}Suppose the conditions of Theorem \ref{th-ldips} hold. Then for
each $t\in \left[ 0,1\right] $, the sequence of random variables $\left\{
\mu ^{n}\left( t\right) \right\} _{n\in \mathbb{N}}$ satisfies an LDP with
rate function 
\begin{equation}
J_{t}\left( \mu _{0},x\right) \doteq \inf \left\{ I_{t}\left( \gamma \right)
:\gamma \in D\left( \left[ 0,1\right] :\mathcal{S}\right) ,\gamma \left(
0\right) =\mu _{0},\gamma \left( t\right) =x\right\}.  \label{J}
\end{equation}
\end{corollary}

\subsection{A locally uniform refinement}

\label{subs-luldp}

In applications, it is often useful to estimate the probability that $\mu
^{n}$ hits a specific point $x_{n}\in \mathcal{S}_{n}$ at some given time,
where $x_{n}\rightarrow x\in \mathcal{S}$ as $n\rightarrow \infty $. The
ordinary LDP does not imply an asymptotic rate for this hitting probability
since it applies only to fixed sets, and the \textquotedblleft
moving\textquotedblright\ set $\left\{ x_{n}\right\} $ in the present case
has empty interior. To obtain such a \textquotedblleft locally
uniform\textquotedblright\ result we need a strengthening of the $K$%
-ergodicity condition. Recall the single-transition rate matrix $\Gamma^1$.

\begin{assumption}
\label{g1} For every $x\in \mathcal{S}$, the Markov process on $\mathcal{X}$
with transition rate matrix $\Gamma^{1}\left( x\right) $ is ergodic.
\end{assumption}

Note that Assumption \ref{g1} implies Assumption \ref{g2} and, thus, is
stronger than Assumption \ref{g2}, which in itself (in the presence of
Assumption \ref{ue}) is a strengthening of $K$-ergodicity (see Lemma \ref%
{lem-verk}). We now state the locally uniform LDP result, which is proved in
Section \ref{sec-locunif}.

\begin{theorem}
\label{th-localldips} Suppose $\{\Gamma _{\mathbf{i}\mathbf{j}}^{k}(x),x\in {%
\mathcal{S}},(\mathbf{i},\mathbf{j})\in {\mathcal{J}}^{k},k=1,\ldots ,K\}$
satisfies Assumptions \ref{intsys}, \ref{ue}, \ref{ass-simjumps} and \ref{g1}%
, and let $\{\mu ^{n}\}_{n\in \mathbb{N}}$ be the associated sequence of
empirical measure processes. Also, assume the initial conditions $\left\{
\mu ^{n}\left( 0\right) \right\} _{n\in \mathbb{N}}$ are deterministic, and $%
\mu ^{n}\left( 0\right) \rightarrow \mu _{0}\in {\mathcal{P}}\left( \mathcal{%
X}\right) $ as $n$ tends to infinity. Let $\left\{ x_{n}\right\} _{n\in 
\mathbb{N}}$ $\subset \mathcal{S}_{n}$, $x\in \mathcal{S}$, be such that $%
x_{n}\rightarrow x$ as $n\rightarrow \infty $. Then for any $t\in \lbrack
0,1)$, 
\begin{equation*}
\lim_{n\rightarrow \infty }\frac{1}{n}\log \mathbb{P}\left( \mu ^{n}\left(
t\right) =x_{n}\right) =-J_{t}\left( \mu _{0},x\right) ,
\end{equation*}%
where $J_{t}$ is as defined in (\ref{J}).
\end{theorem}

For the $n$-particle systems we study, it is also natural to start with
random initial conditions. Depending on the large deviation rate of the
sequence of initial conditions, this gives rise to an additional cost in the
rate function. The LDP for empirical measure processes with random initial
conditions are stated in the following corollary.

\begin{corollary}
\label{cor-randin} Suppose that Assumptions \ref{intsys}, \ref{ue}, \ref%
{ass-simjumps} and \ref{g1} are satisfied. Also, assume that the sequence of
initial conditions $\left\{ \mu ^{n}\left( 0\right) \right\} _{n\in \mathbb{N%
}}$ converges to $\mu _{0}$ in such a way that they satisfy an LDP with rate
function $J_{0}\left( \cdot \right) $. Then the corresponding sequence of
empirical measure processes $\left\{ \mu ^{n}\right\} _{n\in \mathbb{N}}$
satisfies the sample path LDP with rate function $J_{0}\left( \gamma \left(
0\right) \right) +I\left( \gamma \right) $.
\end{corollary}

The proof of the corollary relies on the continuity of the following
functional: given a bounded and continuous functional $h$ on $D\left( [0,1]:%
\mathcal{S}\right) $, define 
\begin{equation}
U\left( y\right) =\inf \left\{ I\left( \gamma \right) +h\left( \gamma
\right) :\gamma \in D\left( [0,1]:\mathcal{S}\right) ,\gamma \left( 0\right)
=y\right\} ,\quad y\in {\mathcal{S}}.  \label{fn-u}
\end{equation}%
Then it follows from Lemma \ref{lem-ucont} that $U$ is continuous.

\begin{proof}[Proof of Corollary \protect\ref{cor-randin}.]
Given a bounded and continuous function $h:D\left( [0,1]:\mathcal{S}\right)
\mapsto \mathbb{R}$, for any $y\in \mathcal{S}_{n}$ denote $U^{n}\left(
y\right) \doteq -\frac{1}{n}\log \mathbb{E}_{y}[e^{-nh\left( \mu ^{n}\right)
}]$. Since $U$ is continuous and $\{\mu^n\}_{n \in \mathbb{N}}$ satisfies an
LDP (Theorem \ref{th-ldips}), the equivalence between the LDP and the
Laplace principle \cite[Theorems 1.2.1 and 1.2.3]{DE} implies that $U^{n}$
converges to $U$ uniformly on ${\mathcal{S}}$. In particular, this shows
that if $y_n \rightarrow y$ in ${\mathcal{S}}$, then $U^{n}\left(
y_{n}\right) \rightarrow U(y)$. Let $\nu^{n}$ denote the law of $\mu
^{n}\left( 0\right) $. Then%
\begin{eqnarray*}
\lim_{n\rightarrow \infty }-\frac{1}{n}\log \mathbb{E}_{\mu ^{n}\left(
0\right) }\left[ e^{-nh\left( \mu ^{n}\right) }\right] &=&\lim_{n\rightarrow
\infty }-\frac{1}{n}\log \sum_{y_{n}\in \mathcal{S}_{n}}e^{-nU^{n}\left(
y_{n}\right) }\nu ^{n}\left\{ y_{n}\right\} \\
&=&\lim_{n\rightarrow \infty }-\frac{1}{n}\log \int e^{-n\left( U\left(
y\right) +o\left( 1\right) \right) }\nu ^{n}\left( dy\right) \\
&=&\inf_{y\in \mathcal{S}}\left\{ U\left( y\right) +J_{0}\left( y\right)
\right\} \\
&=&\inf_{\gamma \in D\left( [0,1]:\mathcal{S}\right) }\left\{ h\left( \gamma
\right) +J_{0}\left( \gamma \left( 0\right) \right) +I\left( \gamma \right)
\right\} ,
\end{eqnarray*}%
where the third equality follows from the assumed LDP for deterministic
initial conditions and the continuity of $U$, and the fourth equality
follows from the definition of $U$ in \eqref{fn-u}. The conclusion of the
corollary then follows from the equivalence between the LDP and the Laplace
principle.
\end{proof}

\begin{remark}
\emph{An example of initial conditions in the $n$-particle system that
satisfy the assumptions of Corollary \ref{cor-randin} is the case when
particles are initially distributed as iid $\mathcal{X}-$valued random
variables, with common distribution $\nu $. Then by Sanov's theorem, $%
J_{0}\left( \mu _{0}\right) =R\left( \mu _{0}\left\Vert \nu \right. \right)
=\sum_{i=1}^{d}\mu _{0,i}\log \frac{\mu _{0,i}}{\nu _{i}}$.}
\end{remark}

\begin{remark}
\emph{The assumptions of the locally uniform case are used in the proof of
Lemma \ref{lem-ucont} to establish that }$U$ \emph{is continuous on }${%
\mathcal{S}}$\emph{. Any set of conditions implying this continuity can also
be used, and under the conditions of Theorem \ref{th-ldips}, }$U$ \emph{is
continuous on the interior of }${\mathcal{S}}$\emph{, and hence the
corollary holds if the distributions of initial conditions have support in a
compact subset of the relative interior of }${\mathcal{S}}$\emph{. }
\end{remark}

\bigskip

\subsection{LDP for Invariant Measures}

\label{subs-invldp}

We now discuss some ramifications of the locally uniform LDP. In \cite{FW} a
uniform (with respect to initial conditions) sample path LDP for small noise
diffusions is used to study its metastability properties, including the mean
exit time and most likely exit location from a given domain, and to
establish an LDP for the sequence of invariant measures with the rate
function given by the so-called quasipotential. The program of \cite{FW} was
carried out for non-degenerate diffusions in $\mathbb{R}^{d}$; here we have
a sequence of jump processes on lattice approximations of a compact set.
However, we remark here that the same arguments carry through without
essential change in the presence of a certain communication property, namely
Property \ref{prop-dcommunicate}.i) in Section \ref{subs-propverii}, which
is shown to be implied by Assumptions \ref{intsys}, \ref{ue}, \ref{g1} and %
\ref{ass-simjumps} in Lemma \ref{move} (see also \cite{BS13} for details in
the case of empirical measures arising from single-jump interacting particle
systems, that is, systems with $K=1$). In \cite{FW} extra conditions are
assumed to guarantee that the process does not escape to infinity with
significant probability; for our model, since the state space is compact,
this is automatic.

When Assumption \ref{g1} is satisfied, for each $n\in \mathbb{N}$, all
states in $\mathcal{S}_{n}$ communicate under the dynamics of $\mu ^{n}$,
and hence there exists a unique invariant measure $\pi ^{n}$ for this Markov
process. In our setting, the quasipotential is defined by 
\begin{equation*}
V\left( x,y\right) =\inf \left\{ I_{t}\left( \gamma \right) :\gamma \in
D\left( [0,t]:\mathcal{S}\right) ,\gamma \left( 0\right) =x,\gamma \left(
t\right) =y,t<\infty \right\} ,\text{ for }x,y\in \mathcal{S}.
\end{equation*}

For the results of \cite{FW} to carry over to our setting, we need the
quasipotential to be continuous on its domain.

\begin{lemma}
\label{uc}If Assumption \ref{intsys} and Assumption \ref{g1} hold, then $%
V\left( \cdot ,\cdot \right) $ is jointly continuous in $\mathcal{S}\times 
\mathcal{S}$.
\end{lemma}

The proof of Lemma \ref{uc} is given in Section \ref{subs-luldpupper}; it
essentially follows from the property that for any $x,y\in \mathcal{S}$ that
are sufficiently close, one can construct a path that connects $x$ to $y$
with arbitrary small cost (see Lemma \ref{7.1} for a precise statement). We
now state the LDP for invariant measures in the case when the LLN limit $%
\mu(\cdot)$ has a unique fixed point. In view of \eqref{1}, a fixed point of
the LLN dynamics is a measure $\pi^* \in {\mathcal{S}}$ with the property
that $\sum_{v \in {\mathcal{V}}} v \lambda_v (\pi^*) = 0$. Moreover, the
fixed point is said to be globally attracting if for every $x_0 \in {%
\mathcal{S}}$, the solution $\mu(\cdot)$ to \eqref{1} with initial condition 
$x_0$ satisfies $\mu(t) \rightarrow \pi^*$ as $t \rightarrow \infty$.

\begin{theorem}
\label{thm:LDforIM}Assume that $x_{0}$ is the unique fixed point of the LLN
dynamics (\ref{1}), and is globally attracting (in $\mathcal{S}$). Also
assume 
Assumptions \ref{intsys}, \ref{ue}, \ref{ass-simjumps} and \ref{g1} are
satisfied. Then for any $n>0$, there exists a unique invariant measure $\pi
^{n}$ of the Markov process with generator (\ref{Lnl}). Moreover, the
sequence $\left\{ \pi ^{n}\right\} _{n\in \mathbb{N}}$ satisfies an LDP with
rate function $V\left( x_{0},\cdot \right) $.
\end{theorem}

When the LLN limit (\ref{1}) has multiple stable equilibria, following the
same approach as in the case of non-degenerate diffusions (see \cite{FW},
Chapter 6.4), a generalization of Theorem \ref{thm:LDforIM} can be obtained.

\section{\protect\bigskip Properties of the Limit Jump Rates}

\label{sec-ver}

In this section we establish certain important properties of the limit jump
rates $\{\lambda _{v}(\cdot ),v\in {\mathcal{V}}\}$ associated with
interacting particle systems whose transition rates satisfy the assumptions
introduced in the last section. First, in Section \ref{subs-communicate} we
describe certain communication conditions (Property \ref{prop-communicate})
that are required to avoid singularities in the large deviation analysis. In
Sections \ref{subs-prelim} and \ref{subs-pfveri} we show that these
communication conditions are satisfied by the limit jump rates $\{\lambda
_{v}(\cdot ),v\in {\mathcal{V}}\}$ associated with any interacting particle
system model that satisfies Assumptions \ref{intsys}, \ref{ue} and \ref%
{ass-kerg}. Next, in Section \ref{subs-ratioest} we show that Assumptions %
\ref{intsys}, \ref{ue} and \ref{ass-simjumps} together imply certain
estimates (Lemma \ref{lem-estimates}) on the jump rates $\{\lambda
_{v}(\cdot ),v\in {\mathcal{V}}\}$. Then, in Section \ref{subs-lln-inward}
we show that if the jump rates $\{\lambda _{v}(\cdot ),v\in {\mathcal{V}}\}$
satisfy some of the estimates from Lemma \ref{lem-estimates} and the
communication property (Property \ref{prop-communicate}), then one can
obtain a suitable upper bound on the time taken by the LLN path to hit a
compact subset of the interior of the simplex ${\mathcal{S}}$. The latter
property plays a crucial role in the proof of the large deviation lower
bound. In fact, as made precise in Remark \ref{rem-ldp}, the above
properties of $\{\lambda _{v}(\cdot ),v\in {\mathcal{V}}\}$ are the only
ones used to establish the LDP, and thus the conclusion of Theorem \ref%
{th-ldips} in fact holds for the larger class of sequences of jump Markov
processes that satisfy Property \ref{prop-lambda} and the above-stated
properties. Finally, in Section \ref{subs-propverii} we establish a discrete
version of the communication condition that is used (only) in the proof of
the locally uniform LDP in Section \ref{sec-locunif}. This is a technical
section of the paper. Readers only interested in the LDP proof may want to
skip Section \ref{subs-propverii}.

\subsection{Communication Conditions}

\label{subs-communicate}

In Definition \ref{def-compath}, we first introduce the notion of a
communicating path associated with the limit jump rates $\{\lambda_v(\cdot),
v \in {\mathcal{V}}\}$.

\begin{definition}
\label{def-compath} Given rates $\{\lambda _{v}(\cdot ),v\in {\mathcal{V}}\}$%
, for any $x,y\in \mathcal{S}$ and $t\in (0,1]$, a \textbf{communicating path%
} on $[0,t]$ from $x$ to $y$ with constants $c>0,p<\infty $, and $F\in 
\mathbb{N}$ is a piecewise linear function $\phi :[0,t]\mapsto {\mathcal{S}}$
that satisfies $\phi \left( 0\right) =x$, $\phi \left( t\right) =y$ and the
following two properties:

i). there exist $\left\{ v_{m}\right\} _{m=1}^{F}\subset \mathcal{V}$, $%
0=t_{0}<t_{1}<\cdots <t_{F}=t$, $\left\{ U_{m}\right\} _{m=1}^{F}\subset 
\mathbb{R}_{+}$, such that 
\begin{equation}
\dot{\phi}\left( s\right) =\sum_{m=1}^{F}U_{m}v_{m}\mathbb{I}%
_{[t_{m-1},t_{m})}\left( s\right) ,\text{ \ a.e. }s\in \left[ 0,t\right] ,
\label{CP}
\end{equation}

ii). for $m=1,...,F$, 
\begin{equation*}
\lambda _{v_{m}}\left( \phi \left( s\right) \right) \geq c\left(
\min_{i=1,...,d}y_{i}\right) ^{p}\text{ for }s\in \lbrack t_{m-1},t_{m}).
\end{equation*}
\end{definition}

\begin{remark}
\emph{Definition \ref{def-compath}.ii) implies that for }$y$\emph{\ in a
compact subset of }$S^\circ$\emph{, }$\lambda _{v_{m}}\left( \phi \left(
s\right) \right) $\emph{\ is uniformly bounded from below. In fact, one can
weaken Definition \ref{def-compath}.ii) (and correspondingly, Property \ref%
{prop-communicate}) in this way and the proof of the LDP lower bound still
holds. Nevertheless, we choose to define a communicating path using the
slightly stronger condition in Definition \ref{def-compath} because it is
naturally satisfied by interacting particle systems with }$K$-\emph{ergodic
jump rates (see Definition \ref{k-acc}), and it is analogous to the
corresponding condition in Definition \ref{def1''} of a strongly
communicating path, which is used in the proof of the locally uniform LDP in
Section \ref{subs-propverii}.}
\end{remark}

Let $AC\left( \left[ 0,T\right] :\mathcal{S}\right) $ denote the absolutely
continuous functions from $[0,T]$ to $\mathcal{S}$. In what follows, given $%
t>0$ and a path $\phi \in AC\left( \left[ 0,t\right] :\mathcal{S}\right) $,
let 
\begin{equation}
\text{\textrm{Len}}\left( \phi \right) \doteq \int_{0}^{t}\left\vert
\left\vert \dot{\phi}\left( s\right) \right\vert \right\vert ds  \label{r}
\end{equation}%
denote the length of $\phi $. We now state the communication condition on
the jump rates $\{\lambda _{v}(\cdot ),v\in {\mathcal{V}}\}$.

\begin{property}
\label{prop-communicate} The rates $\{\lambda _{v}(\cdot ),v\in {\mathcal{V}}%
\}$ are such that there exist constants $c>0$, $C^{\prime }<\infty $, $%
p<\infty $ and $F\in \mathbb{N}$, such that for every $x\in \mathcal{S}$ and 
$y\in \mathrm{\ int}\left( \mathcal{S}\right) $, there exist $t\in (0,1]$,
and a communicating path $\phi $ on $[0,t]$ from $x$ to $y$ exists with the
given $c,p,F$ such that 
\begin{equation}
\mathrm{Len}\left( \phi \right) \leq C^{\prime }||x-y||.  \label{length}
\end{equation}
\end{property}

For the locally uniform LDP, we need the following strengthening of the
notion of a communicating path.


\begin{definition}
\label{def1''} Given $x,y\in \mathcal{S}$ and $t\in (0,1]$, a piecewise
linear function $\phi :[0,t]\mapsto {\mathcal{S}}$ is said to be a \textbf{%
strongly communicating path} on $[0,t]$ from $x$ to $y$ with constants $c>0$%
, $p<\infty $, $F\in \mathbb{N}$, $c_{1}>0,p_{1}<\infty $ if it is a
communicating path on $[0,t]$ from $x$ to $y$ with constants $c,p,F$ and, in
addition, 

iii). if $\phi $ has the representation (\ref{CP}), then for $m=1,...,F$,%
\begin{equation}
\lambda _{v_{m}}\left( \phi \left( s\right) \right) \geq c_{1}\left( %
\displaystyle\prod\limits_{j\in \mathcal{N}_{v_{m}}}\phi _{j}\left( s\right)
\right) ^{p_{1}},\text{ \ }s\in \lbrack t_{m-1},t_{m}),  \label{lvm}
\end{equation}%
where, as defined in \eqref{def-nv}, $\mathcal{N}_{v_{m}}=\{i\in {\mathcal{X}%
}:\langle v_{m},e_{i}\rangle <0\}$.
\end{definition}

Note that when $y$ lies in $\partial \mathcal{S}$, $\min_{i=1,...,d}y_{i}=0$%
, and therefore property (ii) in Definition \ref{def-compath} is trivially
satisfied. In constrast, the polynomial lower bound in (\ref{lvm}) imposes
stronger requirements on the paths that are not automatically satisfied even
when $y \in \partial \mathcal{S}$.

\begin{remark}
\label{flex}\emph{There is some flexibility in the choice of $t$, $\left\{
U_{m}\right\} _{m=1}^{F}$ and $\left\{ t_{m}\right\} _{m=1}^{F}$ both in
Definition \ref{def-compath} and Definition \ref{def1''}. By a
reparametrization of the respective paths, we can always assume $t=1$ and $%
U_{m}=U$ for every $m$. Moreover, if $t$ is allowed to take any values in }$%
\left( 0,\infty \right) $\emph{, we can always choose $U_{m}=1$. }
\end{remark}

\subsection{A Preliminary Result}

\label{subs-prelim}

Here, we show that Assumption \ref{ue} and Assumption \ref{g1}, 
together imply a certain strong controllability property of the associated
(limit) jump rates $\{\lambda_v (\cdot), v \in {\mathcal{V}}\}$. This result
is used both in the verification (under suitable assumptions) of the
communication condition in Section \ref{subs-pfveri} and of its
strengthening (under more restrictive assumptions) in Section \ref%
{subs-propverii}.

\begin{lemma}
\label{move} Suppose the transition rates $\{\Gamma _{\mathbf{i}\mathbf{j}},(%
\mathbf{i},\mathbf{j})\in {\mathcal{J}}\}$ satisfy Assumption \ref{ue} and
Assumption \ref{g1}. Then the associated jump rates $\{\lambda _{v}(\cdot
),v\in {\mathcal{V}}\}$ defined in \eqref{lambda} have the property that
there exist constants $C^{\prime }<\infty $, $c,c_{1}>0$, $p,p_{1}<\infty $
and $F\in \mathbb{N}$, such that for every $x,y\in {\mathcal{S}}$ there
exists a strongly communicating path $\phi $ from $x$ to $y$ with constants $%
c,c_{1},p,p_{1}$ and $F$ such that, in addition, \eqref{length} is
satisfied. Moreover, if the requirement that $t\in (0,1]$ is dropped, the
path $\phi $ can be chosen so that its derivatives all lie in the set $%
\mathcal{V}_{1}=\{e_{j}-e_{i},i,j\in {\mathcal{X}},i\neq j\}$. Furthermore,
for any $n\in \mathbb{N}$, if $x,y\in {\mathcal{S}}_{n}$, then there exists
a strongly communicating path $\phi ^{P{,n}}$ from $x$ to $y$ whose
representation \eqref{CP} satisfies $U_{m}=1/n$, $\phi (t_{m})\in {\mathcal{S%
}}_{n}$ and $t_{m}-t_{m-1}\in \mathbb{N}$, for $m=1,\ldots ,F$.
\end{lemma}

\begin{proof}
We will prove the result by a recursive construction. We claim that for any $%
r\in \left\{ 2,...,d\right\} $, there exists $C_{r}=C_{r}(d,K)<\infty $ such
that for every $x,y\in {\mathcal{S}}$, there exists ${\mathcal{X}}%
_{r}\subset {\mathcal{X}}$ with $|{\mathcal{X}}_{r}|\geq r-1$, $z^{(r)}\in {%
\mathcal{S}}$ with $z_{i}^{(r)}=y_{i}$ for $i\in {\mathcal{X}}_{r}$, $0\leq
t_{r-1}<\infty $ and a strongly communicating path $\phi ^{(r)}$ on $%
[0,t_{r-1}]$ from $x$ to $z^{(r)}$ such that $\mathrm{Len}(\phi ^{(r)})\leq
C_{r}||x-y||$ and all the derivatives of $\phi^{(r)}$ lie in ${\mathcal{V}}%
_1 $. The first two assertions of the lemma then follows on taking $r=d$
because the fact that both $z^{(r)}$ and $y$ lie on the simplex implies that
they are equal if and only if they agree on $d-1$ coordinates.

We will prove the claim by induction. We first consider the case $r=2$,
which is easy. We assume without loss of generality that $x_{i}\neq y_{i}$
for some $i\in {\mathcal{X}}$, for otherwise the construction is trivial.
Choose $u_{1}\in \mathcal{X}$ such that $x_{u_{1}}\geq y_{u_{1}}$ (such a $%
u_{1}$ always exists because $x,y\in {\mathcal{S}}$), then set ${\mathcal{X}}%
_{1}\doteq \{u_{1}\}$, $t_{1}\doteq x_{1}-y_{1}$ and choose any state $%
u_{2}\in \mathcal{X}\backslash \left\{ u_{1}\right\} $ such that $\mathfrak{M%
}_{u_{1}u_{2}}^{1}>0$ (the existence of $u_{2}$ is implied by Assumption \ref%
{ue} and Assumption \ref{g1}). Define 
\begin{equation*}
\phi ^{\left( 2\right) }\left( t\right) \doteq x+\left(
e_{u_{2}}-e_{u_{1}}\right) t,\text{ }\quad t\in \left[ 0,t_{1}\right] .
\end{equation*}%
Then clearly, $\phi ^{\left( 2\right) }\left( 0\right) =x$, $z^{\left(
2\right) }=\phi ^{\left( 2\right) }\left( t_{1}\right) \in \mathcal{S}$, $%
z_{u_{1}}^{\left( 2\right) }=y_{1}$ and $\mathrm{Len}(\phi ^{(2)})=\sqrt{2}%
t_{1}\leq \sqrt{2}||x-y||$. Moreover, by (\ref{alpha1}) and (\ref{lambda}),
setting $v_{1}\doteq e_{u_{2}}-e_{u_{1}}$, we have 
\begin{equation*}
\lambda _{v_{1}}\left( x\right) =\sum_{k=1}^{K}\sum_{\substack{ \left( 
\mathbf{i},\mathbf{j}\right) \in \mathcal{J}^{k}\text{:}  \\ e_{\mathbf{j}%
}-e_{\mathbf{i}}=v_{1}}}\alpha _{\mathbf{ij}}^{k}\left( x\right) \geq
x_{u_{1}}\Gamma _{u_{1}u_{2}}^{1}\left( x\right) \geq c_{0}x_{u_{1}},
\end{equation*}%
where we have used the fact that $\mathfrak{M}_{u_{1}u_{2}}^{1}>0$ and $%
c_{0} $ is defined by (\ref{def-c0}). Thus, the lower bound (\ref{lvm})
holds with $\mathcal{N}_{v_{1}}=\{u_{1}\}$, $p_{1}=1$, and some $%
c_{1}=c_{0}>0$. Thus, $\phi ^{\left( 2\right) }$ is a path of the desired
form.

Now, assume the claim holds for $r=l<d$, and let ${\mathcal{X}}_{l}$, $%
z^{\left( l\right) }$, $t_{l-1}$, and $\phi ^{\left( l\right) }$ be the
corresponding quantities in the claim. We now prove the claim for $r=l+1$.
By the induction hypothesis, we have $\phi ^{\left( l\right) }\left(
t_{l-1}\right) =z^{\left( l\right) }$ and $z_{i}^{(l)}=y_{i}$ for $i\in {%
\mathcal{X}}_{l}$, and $|{\mathcal{X}}_{l}|\geq l-1$. We can assume without
loss of generality that ${\mathcal{X}}_{l}=\{i\in {\mathcal{X}}%
:z_{i}^{(l)}=y_{i}\}$ satisfies $|{\mathcal{X}}_{l}|=l-1$, for otherwise the
claim clearly also holds for $r=l+1$. Under this assumption, we have $%
\sum_{i\in {\mathcal{X}}\setminus {\mathcal{X}}_{l}}z_{i}^{(l)}=\sum_{i\in {%
\mathcal{X}}\setminus {\mathcal{X}}_{l}}y_{i}$ and there exists $j\in {%
\mathcal{X}}\setminus {\mathcal{X}}_{l}$ such that $z_{j}^{(l)}>y_{j}$. For
notational simplicity, we assume without loss of generality that ${\mathcal{X%
}}_{l}=\{1,\ldots ,l-1\}$ and $j=l$. Then $z_{l}^{\left( l\right) }>y_{l}$,
and to prove the claim we will move mass from state $l$ to some state in $%
\left\{ l+1,\ldots ,d\right\} $, without changing the mass in any state with
a lower index. In other words, we will construct a path $\psi \in AC\left(
[0,t_{\ast }]:\mathcal{S}\right) $ for some $t_{\ast }<\infty $, such that $%
\psi _{i}\left( t_{\ast }\right) =\psi _{i}\left( 0\right) $ for every $%
i=1,\ldots ,l-1$ and $\psi _{l}\left( t_{\ast }\right) -\psi _{l}\left(
0\right) =-(z_{l}^{\left( l\right) }-y_{l})$. To do this, take any $w\in $ $%
\{l+1,...,d\}$. By Assumption \ref{ue} and Assumption \ref{g1}, there exist $%
M\leq d$ and a sequence of distinct states $u_{0}=l,...,u_{M}=w$, such that
for $1\leq m\leq M-1$, $\mathfrak{M}_{u_{m}u_{m+1}}^{1}>0.$ Now, let $%
M_{\ast }\doteq \min \left\{ m\geq 1:u_{m}\in \{l+1,...,d\}\right\} $, and
note that $u_{M_{\ast }}$ is the first state in the sequence that lies
outside $\{1,...,l\}$, and $M_{\ast }$ is the number of steps it took to get
there. Define $t_{\ast }\doteq M_{\ast }(z_{l}^{\left( l\right) }-y_{l})$,
and let $\psi \in {\mathcal{A}C}([0,t_{\ast }]:{\mathcal{S}})$ be defined by 
$\psi (0)=z^{\left( l\right) }$ and 
\begin{equation}
\dot{\psi}\left( t\right) =v_{m}\doteq e_{u_{m}}-e_{u_{m-1}},\quad t\in
(\left( m-1\right) (z_{l}^{\left( l\right) }-y_{l}),m(z_{l}^{\left( l\right)
}-y_{l})),\,m=1,\ldots ,M_{\ast }.  \label{phip}
\end{equation}%
Since the states are distinct, $M_{\ast }\leq d$, and we have 
\begin{eqnarray*}
\mathrm{Len}(\psi ) &\leq &\sqrt{2}d\left\vert z_{l}^{\left( l\right)
}-y_{l}\right\vert \\
&\leq &\sqrt{2}d\left( \left\vert x_{l}-y_{l}\right\vert +\left\vert
z_{l}^{\left( l\right) }-x_{l}\right\vert \right) \\
&\leq &\sqrt{2}d\left( \left\Vert x-y\right\Vert +\int_{0}^{t_{l-1}}\Vert 
\dot{\phi}^{\left( l\right) }\left( s\right) \Vert ds\right) \\
&\leq &\sqrt{2}d\left( 1+C_{l}\right) \left\Vert x-y\right\Vert ,
\end{eqnarray*}%
where the last inequality follows from the induction assumption for $r=l$.

Define $t_{l+1}\doteq t_{l}+t_{\ast }$, and let $\phi ^{\left( l+1\right)
}\in {\mathcal{A}C}\left( \left[ 0,t_{l+1}\right] :\mathcal{S}\right) $ be
the concatenation of $\phi ^{\left( l\right) }$ and $\psi $. As we show
below, $\phi ^{\left( l+1\right) }$ is a path of the desired form. Clearly,
if $z^{(l+1)}\doteq \phi ^{\left( l+1\right) }(t_{l+1})$, then ${\mathcal{X}}%
_{l+1}\doteq \{j\in {\mathcal{X}}:z_{j}^{(l+1)}=y_{j}\}=\{1,\ldots ,l\}$.
Further, since the derivatives of both $\psi $ and $\phi ^{(l)}$ all lie in $%
{\mathcal{V}}_{1}$, the same holds true for $\phi ^{(l+1)}$. Now, given any $%
v\in \mathcal{V}_{1}$ of the form $v=e_{j}-e_{i}$ for some $i,j\in \mathcal{X%
}$, $i\neq j$, with $\mathfrak{M}_{ij}^{1}>0$, as before, we have $\lambda
_{v}\left( x\right) \geq x_{i}\Gamma _{ij}^{1}\left( x\right) \geq
c_{0}x_{i} $, and therefore, for a.e.\ $s\in \left[ t_{l},t_{l+1}\right] $, $%
\lambda _{v}(\phi ^{\left( l+1\right) }\left( s\right) )\geq c_{0}\phi
_{i}^{\left( l+1\right) }\left( s\right) $, where $v=\dot{\phi}^{(l+1)}(s)$
and ${\mathcal{N}}_{v}=\{i\}$, from which it is easy to see that $\phi
^{(l+1)}$ is a strongly communicating path. Furthermore, we note that 
\begin{eqnarray*}
\mathrm{Len}(\phi ^{(l+1)})=\int_{0}^{t_{l+1}}\Vert \dot{\phi}^{\left(
l+1\right) }\left( s\right) \Vert ds &=&\mathrm{Len}(\psi )+\mathrm{Len}%
(\phi ^{(l)}) \\
&\leq &\left( \sqrt{2}d+\left( \sqrt{2}d+1\right) C_{l}\right) \left\Vert
x-y\right\Vert ,
\end{eqnarray*}%
which establishes \eqref{length}, with $\phi $ replaced by $\phi ^{(l+1)}$,
and $C^{\prime }$ replaced by $C_{l+1}\doteq \sqrt{2}d+(\sqrt{2}d+1)C_{l}$.
By induction, it follows that the claim holds for $r=d$, thus completing the
proof of the first two assertions of the lemma.

To prove the last assertion (which is only used in the proof of the locally
uniform LDP), suppose we restrict to $x,y\in {\mathcal{S}}_{n}$ for some $%
n\in \mathbb{N}$, and let $\phi $ be the strongly communicating path
constructed above. Then, it is easy to see from the construction that the
lengths of the intervals on which $\phi $ has constant derivative that all
lie in $N_{n}\doteq \{\tau /n:\tau =1,2,\ldots \}$ and thus, the value of $%
\phi $ at the end of each such interval lies in ${\mathcal{S}}_{n}$. A
simple time reparametrization (see Remark \ref{flex}) then yields a path
with the stated properties. This concludes the proof of the lemma.
\end{proof}

\subsection{Verification of the Communication Condition}

\label{subs-pfveri}

This section is devoted to establishing the following result.

\begin{proposition}
\label{ver1} Suppose the family $\{ \Gamma_{\mathbf{i} \mathbf{j}}(\cdot), (%
\mathbf{i},\mathbf{j}) \in {\mathcal{J}}^k, k = 1, \ldots, K\}$ satisfies
Assumption \ref{ue} and Assumption \ref{ass-kerg}. Then the associated jump
rates $\{\lambda_v (\cdot), v \in {\mathcal{V}}\}$ defined via \eqref{lambda}
satisfy Property \ref{prop-communicate}.
\end{proposition}

The proof of the proposition consists of three steps. First, given $x,y\in {%
\mathcal{S}}$, $y\not\in \partial {\mathcal{S}}$, we show that Assumption %
\ref{ue} and $K$-ergodicity (Assumption \ref{ass-kerg}) allow one to move
from any point $x$ on the boundary $\partial S$ to some compact convex
subset of int$\left( \mathcal{S}\right) $ containing $y$ along a piecewise
linear path, each of whose segments is parallel to a jump direction $v\in {%
\mathcal{V}}$ whose rate is uniformly bounded below away from zero on that
segment (Lemma \ref{in}). Then we show that we have stronger controllability
within the compact subset, which allows us to move along any coordinate
direction, again with rates that are uniformly bounded below away from zero
(Lemma \ref{repre}). This property is then used in a straightforward manner
to construct a communicating path in this compact subset (Lemma \ref{commint}%
). The proof of the proposition is completed at the end of the section by
concatenating the two paths constructed above.

We now define two compact subsets of $\mathrm{{int}({\mathcal{S}})}$: for $%
a\in \lbrack 0,1)$, define 
\begin{equation}
\tilde{{\mathcal{S}}}^{a}\dot{=}\left\{ x\in \mathcal{S}:x_{i}\geq a\text{, }%
i=1,...,d\right\} ,  \label{tS}
\end{equation}%
and 
\begin{equation}
{\mathcal{S}}^{a}\doteq \left\{ x\in {\mathcal{S}}:\mathrm{dist}(x,\partial
S)\geq a\right\} =\left\{ x\in {\mathcal{S}}:\inf_{z\in \partial
S}\left\Vert x-z\right\Vert \geq a\right\} .  \label{S}
\end{equation}
Note that for $x\in {\mathcal{S}}^{a}$, $x_{i}\geq a/\sqrt{2}$ for $%
i=1,\ldots ,d$, and therefore $\tilde{{\mathcal{S}}}^{a}\subset {\mathcal{S}}%
^{a}\subset \tilde{{\mathcal{S}}}^{a/\sqrt{2}}$. Thus, the two sets have
similar properties, but it will be more convenient to use one or the other
depending on the context.

\begin{lemma}
\label{in} Suppose $K\geq 2$ and that the family $\{\Gamma _{\mathbf{i}%
\mathbf{j}}(\cdot ),(\mathbf{i},\mathbf{j})\in {\mathcal{J}}^{k},k=1,\ldots
,K\}$ satisfies Assumption \ref{ue} and Assumption \ref{ass-kerg}, and let $%
c_0$ be as defined in \eqref{def-c0}. Then for any $x\in \mathcal{S}$ and $%
a\in \lbrack 0,1/(\left( K+1\right) ^{d-1}d)]$, there exist $z\in \tilde{%
\mathcal{S}}^{a}$, $t_{0}<\infty $ and a communicating path $\phi $ on $%
[0,t_{0}]$ from $x$ to $z$ with constants $c=c_{0}/K!$, $p=d$ and $F\leq
d^{2}$ that also satisfies the following two properties:

\begin{enumerate}
\item for any $s\in \left[ 0,t_{0}\right] $ and $i\in \mathcal{X}$ such that 
$\dot{\phi}_{i}\left( s\right) <0$, the inequality $\phi _{i}\left( s\right)
\geq a$ holds.

\item $\mathrm{Len} (\phi) \leq C_x \mathrm{dist}(x,\tilde{{\mathcal{S}}}^a)$
for some $C_x < \infty$, which does not depend on $a$.
\end{enumerate}

Furthermore, the family of paths can be chosen so that $C\doteq \sup_{x\in {%
\mathcal{S}}}C_{x}<\infty $.
\end{lemma}

Before proving the lemma in general, we first illustrate the argument for
Example \ref{eg3}, which has $d=4$ and $K=2$.

\medskip \noindent \textbf{Example \ref{eg3} cont'd.} We assume that $%
x_{1}\geq x_{2}\geq x_{3}\geq x_{4}$ (the other cases can be treated in an
exactly analogous fashion). Note that then $x_{1}\geq 1/4$. Fix $a\in
\lbrack 0,1/108]$. If $x_{i}\geq a$ for all $i=1,\ldots ,4$ then we can set $%
z=x$ and the null path $\phi $ is trivially a communicating path. Otherwise,
we consider three mutually exclusive and exhaustive cases and discuss the
construction of the path in each case. We set $\left( \mathbf{i}_{1}\mathbf{%
,j}_{1}\right) =\left( 1,2\right) $, and $\left( \mathbf{i}_{2}\mathbf{,j}%
_{2}\right) =\left( \left( 1,2\right) ,\left( 3,4\right) \right) $.

\emph{Case I.} $x_{3}\geq a>x_{4}$. Take $\dot{\phi}\left( s\right) =\left(
e_{\mathbf{j}_{1}}-e_{\mathbf{i}_{1}}\right) \mathbb{I}_{(0,2\left(
a-x_{4}\right) )}\left( s\right) +\left( e_{\mathbf{j}_{2}}-e_{\mathbf{i}%
_{2}}\right) \mathbb{I}_{(2\left( a-x_{4}\right) ,3\left( a-x_{4}\right)
)}\left( s\right) $ and $\phi \left( 0\right) =x.$ Then $\phi $ clearly
satisfies property i) of Definition \ref{def-compath} with $F=2$. Moreover,
note that for $s\in \left[ 0,3\left( a-x_{4}\right) \right] $, \ $\phi
_{1}\left( s\right) \geq \phi _{1}\left( 0\right) -2\left( a-x_{4}\right)
-\left( a-x_{4}\right) \geq x_{1}-3a\geq a$, $\phi _{2}\left( s\right) \geq
\phi _{2}\left( 0\right) \geq a$, $\phi _{3}$ and $\phi _{4}$ are
nondecreasing 
and $\phi _{4}\left( 3\left( a-x_{4}\right) \right) =a$. Thus, we have $%
z\doteq \phi \left( 3\left( a-x_{4}\right) \right) \in \tilde{{\mathcal{S}}}%
^{a}$. Moreover, $\lambda _{v}(\phi (s))=c_{1}\phi _{1}(s)\geq c_{1}a$ for $%
v=e_{\mathbf{j}_{1}}-e_{\mathbf{i}_{1}}$ and $s\in \lbrack 0,2\left(
a-x_{4}\right) ]$ and $\lambda _{v}(\phi (s))=\frac{c_{5}}{2}\phi
_{1}(s)\phi _{2}(s)\geq \frac{c_{5}}{2}a^{2}$ for $v=e_{\mathbf{j}_{2}}-e_{%
\mathbf{i}_{2}}$ and $s\in \lbrack 2\left( a-x_{4}\right) ,3\left(
a-x_{4}\right) ]$. Thus $\phi $ also satisfies property (ii) of Definition %
\ref{def-compath} with $p=2$ and $c=\min (c_{1},c_{5}/2)$, and is thus a
communicating path from $x$ to $z$. The only $i$ for which $\dot{\phi}%
_{i}\left( s\right) <0$ for any $s$ are $i=1,2$. However, we already
verified that $\phi _{i}\left( s\right) \geq a$ for all $s\in \left[
0,3\left( a-x_{4}\right) \right] $ and $i=1,2,3$. Moreover, it is clear that 
$\mathrm{Len}(\phi )\leq 4\sqrt{2}(a-x_{4})\leq 4\sqrt{2}$dist$(x,\tilde{{%
\mathcal{S}}}^{a})$, and thus $\phi $ also satisfies properties (1) and (2)
of Lemma \ref{in}, with $C_{x}\doteq \mathrm{{dist}(x,\tilde{S}_{a})}$
satisfying $\sup_{x\in {\mathcal{S}}}C_{x}<\infty $.

\emph{Case II.} $x_{2}\geq a>x_{3}$. Set $\dot{\phi}^{\left( 1\right)
}\left( s\right) =\left( e_{\mathbf{j}_{1}}-e_{\mathbf{i}_{1}}\right) 
\mathbb{I}_{(0,2\left( a-x_{3}\right) )}\left( s\right) +\left( e_{\mathbf{j}%
_{2}}-e_{\mathbf{i}_{2}}\right) \mathbb{I}_{(2\left( a-x_{3}\right) ,3\left(
a-x_{3}\right) )}\left( s\right) $ and $\phi ^{\left( 1\right) }\left(
0\right) =x.$ As in Case 1, it is easy to verify that $\phi ^{\left(
1\right) }$ is a communicating path on $[0,3\left( a-x_{3}\right) ]$ from $x$
to $z^{\left( 1\right) }\doteq {\phi }^{\left( 1\right) }\left( 3\left(
a-x_{3}\right) \right) $ that satisfies property 1 of Lemma \ref{in} and has 
$\mathrm{Len}(\phi ^{\left( 1\right) })\leq 4\sqrt{2}|x_{3}-a|\leq 4\sqrt{2}%
\mathrm{dist}(x,\tilde{{\mathcal{S}}}^{a})$. Moreover, we have $%
z_{1}^{\left( 1\right) }\geq a$, $z_{2}^{\left( 1\right) }\geq a$, $%
z_{3}^{\left( 1\right) }=a$ and $a-z_{4}^{\left( 1\right) }\leq a-x_{4}$,
and hence, $\mathrm{dist}(z^{\left( 1\right) },\tilde{{\mathcal{S}}}%
^{a})\leq $ dist$(x,\tilde{{\mathcal{S}}}^{a})$. Then, using the
construction in Case 1 (if $z_{1}^{\left( 1\right) }>z_{2}^{\left( 1\right)
} $, and if not then the construction in Case 1 should be modified by
setting $(\mathbf{i}_{1},\mathbf{j}_{1})=(2,1)$), there exists a
communicating path $\phi ^{\left( 2\right) }$ from $z^{\left( 1\right) }$ to
a point $z^{\left( 2\right) }\in $ $\mathcal{\tilde{S}}^{a}$ that satisfies
properties (1) and (2) of Lemma \ref{in} and has $\mathrm{Len}(\phi ^{\left(
2\right) })\leq 4\sqrt{2}\mathrm{dist}(z^{\left( 1\right) },\mathcal{\tilde{S%
}}^{a})\leq 4\sqrt{2}\mathrm{dist}(x,\mathcal{\tilde{S}}^{a})$. The path $%
\phi $ obtained from concatenating $\phi ^{\left( 1\right) }$ and $\phi
^{\left( 2\right) }$ is then easily seen to satisfy the properties of the
lemma.

\emph{Case III.} $x_{1}>a>x_{2}.$ In this case, set $\phi ^{\left( 1\right)
}\left( 0\right) =x,$ $\dot{\phi}^{\left( 1\right) }\left( s\right) =\left(
e_{\mathbf{j}_{1}}-e_{\mathbf{i}_{1}}\right) \mathbb{I}_{(0,a-x_{2})}\left(
s\right) $. Then $\phi ^{\left( 1\right) }$ is a communicating path from $x$
to $z^{(1)}\doteq \phi ^{\left( 1\right) }(a-x_{2})$. It satisfies property
(2) of Lemma \ref{in} since $\mathrm{Len}(\phi ^{\left( 1\right) })\leq 2%
\sqrt{2}(a-x_{2})\leq 2\sqrt{2}\mathrm{dist}(x,\tilde{{\mathcal{S}}}^{a})$.
Since $x_{1}\geq 1/4$, $z_{1}^{(1)}>x_{2}=a>x_{3}\geq x_{4}$. Thus property
1 holds ($\dot{\phi}_{i}\left( s\right) <0$ only for $i=1$), $\mathrm{dist}%
(z^{(1)},\tilde{{\mathcal{S}}}^{a})\leq \mathrm{dist}(x,\tilde{{\mathcal{S}}}%
^{a})$, and $z^{(1)}$ satisfies the conditions of Case II. So, the desired
path can be obtained by concatenating $\phi ^{(1)}$ with a path $\phi ^{(2)}$
from $z^{(1)}$ to $\tilde{{\mathcal{S}}}^{a}$ constructed as in Case II.


The construction in the above example can be generalized into the following
proof.

\begin{proof}[Proof of Lemma \protect\ref{in}]
If $a=0$ or $x\in \tilde{\mathcal{S}}^a$, we can choose $z=x$ and there is
nothing to prove. Therefore, we assume $a \in (0, 1/((K+1)^{d-1}d)]$ and $%
x\notin \mathcal{\tilde{S}}^{a}$, which in particular implies that $x \neq
(1/d, \ldots, 1/d)$. Then, assume without loss of generality that $x_{1}\geq
x_{2}\geq \cdots \geq x_{d}$, and let 
\begin{equation}  \label{def-nx}
N = N(x) \doteq |\{ l = 1, \ldots, d: x_l < a \}|.
\end{equation}
We will prove the lemma by induction on the quantity $N$. 

We first construct a family of paths that will be used in the inductive
argument. Since Assumption \ref{ass-kerg} implies that the state $d$ is $K$%
-accessible from the state $1$, there exist $M\in \left\{ 2,...,d\right\} $
and a sequence of distinct states $1=u_{1},u_{2},...,u_{M}=d$, such that for 
$m=1,...,M-1$, there exist $k_{m}\in \left\{ 1,...,K\right\} $, $\left( 
\mathbf{i}_{m}\mathbf{,j}_{m}\right) \in \mathcal{J}^{k_{m}}$, and $%
l_{m},l_{m}^{^{\prime }}\in \left\{ 1,...,k_{m}\right\} $, such that $%
u_{m}=i_{m,l_{m}}$, $u_{m+1}=j_{m,l_{m}^{^{\prime }}}$, and $\mathfrak{M}_{%
\mathbf{i}_{m}\mathbf{j}_{m}}^{k_{m}}>0$. Now, for any $m_0 \in \{1, \ldots,
M\}$, we introduce the constants 
\begin{equation*}
c_{m,m_0}\doteq \left\{ 
\begin{array}{ll}
\left( K+1\right) ^{m_0-2}-\left( K+1\right) ^{m_0-1-m} & \text{ if }m\in
\left\{ 1,...,m_0-1\right\} , \\ 
\left( K+1\right) ^{m_0-2} & \text{ if }m=m_0,%
\end{array}%
\right. \text{ }
\end{equation*}%
and note that for every $m = 2, \ldots, m_0-1$, 
\begin{equation}  \label{cm1}
c_{m,m_0}-c_{m-1,m_0}\geq K\sum_{r=m}^{m_0-1}\left(
c_{r+1,m_0}-c_{r,m_0}\right).
\end{equation}
Next, fix $0 < h \leq a$, let $t_{0} = t_0 (m_0)\doteq c_{m_0,m_0} h =
\left( K+1\right)^{m_0-2} h$, and on $[0,t_0]$, define the piecewise linear
path $\phi$ (associated with $m_0$ and $h$) with initial condition $x$ as
follows: $\phi \left( 0\right) =x$, and 
\begin{equation}
\dot{\phi}\left( s\right) = e_{\mathbf{j}_{m}}-e_{\mathbf{i}_{m}} 
\mbox{ for
} s \in (c_{m,m_0}h, c_{m+1,m_0}h), \quad m = 1, \ldots, m_0-1.  \label{cons}
\end{equation}%
The proof proceeds via three main claims.

\noindent \emph{Claim 1.} The path $\phi$ associated with $m_0$, $h$ and
initial condition $x$ satisfies the following properties:

\begin{enumerate}
\item[a)] $\phi_1 (t) > a$ for $t \in [0,t_0]$.

\item[b)] $\phi_{u_m} (t) \geq x_{u_m}$ for $t \in [0,t_0]$ and $m = 2,
\ldots, m_0-1$.

\item[c)] $\phi_u$ is non-decreasing on $[0,t_0]$ for $u \in \{1, \ldots,
d\} \setminus \{u_1, \ldots, u_{m_0-1}\}$.

\item[d)] $\phi _{u}(t_{0})\geq x_{u}+h$ for every $u\in
\{j_{m_{0}-1,l},l=1,\ldots ,k_{m_{0}-1}\}\setminus \{u_{1},\ldots
,u_{m_{0}-1}\}$.

\item[e)] For $m = 1, \ldots, m_0-1$, if $v_m \doteq e_{\mathbf{j}_m} - e_{%
\mathbf{i}_m}$, then for $s \in [0,t_0]$, 
\begin{equation*}
\lambda_{v_m} (\phi(s)) \geq \frac{c_0}{k_m!} \left( \min_{m=1, \ldots,
m_0-1} a \wedge x_{u_m} \right)^{k_m},
\end{equation*}
where $c_0$ is defined by \eqref{def-c0}.

\item[f)] If $\min_{m=1,\ldots ,m_{0}-1}x_{u_{m}}\geq a$ and $h\leq d(x,%
\tilde{{\mathcal{S}}}^{a})$ then $\phi $ is a communicating path from $x$ to 
$\phi (t_{0})$ with constants $c_{0}/K!,d$ and $m_{0}$, and $\phi $ also
satisfies properties (1) and (2) of the lemma with $C_{x}\doteq \sqrt{2K}%
(K+1)^{m_{0}}$.
\end{enumerate}

\emph{Proof of Claim 1. } We start with the proof of property a). Recall
that the assumed ordering of the components of $x\in {\mathcal{S}}$ and the
assumption that $x \neq (1/d, 1/d, \ldots, 1/d)$ implies that $x_{1} > 1/d$,
and that we also have the inequalities $\left\langle e_{\mathbf{i}%
_{m},}e_{1}\right\rangle \leq K$, $h\leq a\leq 1/(\left( K+1\right)
^{d-1}d)<1/d$, and $c_{m_{0},m_{0}}=\left( K+1\right) ^{m_{0}-2}\leq \left(
K+1\right) ^{d-2}$. Substituting this into \eqref{cons}, we obtain for $t\in %
\left[ 0,t_{0}\right] $, 
\begin{align*}
\phi _{1}\left( t\right) =x_{1}+\int_{0}^{t}\dot{\phi}_{1}\left( s\right)
ds& \geq x_{1}-\sum_{m=1}^{m_{0}-1}\left\langle e_{\mathbf{i}%
_{m},}e_{1}\right\rangle \left( c_{m+1,m_{0}}-c_{m,m_{0}}\right) h \\
& > \frac{1}{d}-Kc_{m_{0},m_{0}}\frac{1}{\left( K+1\right) ^{d-1}d} \\
& \geq a.
\end{align*}

For the next property, note that for $m\in \{1,\ldots ,m_{0}-1\}$, $\langle
e_{\mathbf{i}_{r}},e_{u_{m}}\rangle \leq K$ for all $r=1,\ldots ,m$. Recall
that by Definition \ref{k-acc}(ii), for any $m=1,\ldots ,M-1$ mass is only
moved from indices $\{u_{1},\ldots ,u_{m}\}$, in that the components $%
i_{m,l},l=1,\ldots ,k_{m},$ of $\mathbf{i}_{m}$ must be from this set. Since
the $u_{m},m=1,\ldots ,m_{0}-1,$ are distinct, this means that if $m>1$,
then $\langle e_{\mathbf{i}_{m^{\prime }}},e_{u_{m}}\rangle =0$ for $%
m^{\prime }<m$, which in turn implies that $\phi _{u_{m}}$ is non-decreasing
on $[0,c_{m,m_{0}}h]$. On the other hand, since $u_{m}=j_{m-1,l_{m-1}^{%
\prime }}$ and $\langle e_{\mathbf{i}_{r}},e_{u_{m}}\rangle \leq K$ for all $%
r$, for $t\in \lbrack c_{m,m_{0}}h,t_{0}]$ we have 
\begin{eqnarray*}
\phi _{u_{m}}\left( t\right) &\geq &x_{u_{m}}+\left\langle e_{\mathbf{j}%
_{m-1}},e_{u_{m}}\right\rangle \left( c_{m,m_{0}}-c_{m-1,m_{0}}\right)
h-\sum_{r=m}^{m_{0}-1}\left\langle e_{\mathbf{i}_{r},}e_{u_{m}}\right\rangle
\left( c_{r+1,m_{0}}-c_{r,m_{0}}\right) h \\
&\geq &x_{u_{m}}+\left( c_{m,m_{0}}-c_{m-1,m_{0}}\right)
h-K\sum_{r=m}^{m_{0}-1}\left( c_{r+1,m_{0}}-c_{r,m_{0}}\right) h,
\end{eqnarray*}%
which implies property b) due to \eqref{cm1}.

Property c) is a simple consequence of \eqref{cons} and the fact that $%
\langle e_{\mathbf{i}_{m}},e_{u}\rangle >0$ only if $u\in \{u_{1},\ldots
,u_{m}\}$ due to Definition \ref{k-acc}(ii). The latter property also
implies that for $u\in \{j_{m_{0}-1,l},l=1,\ldots ,k_{m_{0}}\}\setminus
\{u_{1},\ldots ,u_{m_{0}-1}\}$, $\langle e_{\mathbf{i}_{m}},e_{u}\rangle =0$
for $m=1,\ldots ,m_{0}-1$. For any such $u$, clearly we also have $\langle
e_{\mathbf{j}_{m_{0}-1}},e_{u}\rangle \geq 1$ (where the strict inequality $%
> $ holds if more than one particle transitions to state $u$ during the
simultaneous transition), and hence, 
\begin{eqnarray*}
\phi _{u}(t_{0}) &=&x_{u}+\sum_{m=1}^{m_{0}-1}\left( \langle e_{\mathbf{j}%
_{m}},e_{u}\rangle -\langle e_{\mathbf{i}_{m}},e_{u}\rangle \right)
(c_{m+1,m_{0}}-c_{m,m_{0}})h \\
&\geq &x_{u}+(c_{m_{0},m_{0}}-c_{m_{0}-1,m_{0}})h \\
&=&x_{u}+h,
\end{eqnarray*}%
where the last inequality uses the identity $c_{m_0,m_0} - c_{m_0-1,m_0} =1$%
. This establishes property d).

Next, for $m=1,\ldots m_{0}-1$, setting $v_{m}\doteq e_{\mathbf{j}_{m}}-e_{%
\mathbf{i}_{m}}$, \eqref{lambda} \eqref{alpha1}, Definition \ref{k-acc}(iii)
and \eqref{def-c0} 
show that for $s\in \lbrack 0,t_{0}]$, 
\begin{equation*}
\lambda _{v_{m}}(\phi (s))\geq \alpha _{\mathbf{i}_{m}\mathbf{j}%
_{m}}^{k_{m}}(\phi (s))=\frac{1}{k_{m}!}\left( \prod_{l=1}^{k_{m}}\phi
_{i_{m,l}}(s)\right) \Gamma _{\mathbf{i}_{m}\mathbf{j}_{m}}^{k_{m}}(\phi
(s))\geq \frac{c_{0}}{k_{m}!}\left( \min_{l=1,\ldots ,k_{m}}\phi
_{i_{m,l}}(s)\right) ^{k_{m}},
\end{equation*}
where $c_0 > 0$ due to Assumption \ref{ue}. When combined with properties a)
and b) and the fact that $\{i_{m,l},l=1,\ldots ,k_{m}\}\subset
\{u_{1},\ldots ,u_{m-1}\}$, $m=1,\ldots ,m_{0}-1$, this proves property e).
Furthermore, when $\min_{m=1,\ldots ,m_{0}-1}x_{u_{m}}\geq a$, \eqref{cons},
properties a), b), e) and the fact that $k_{m}\leq K\leq d$ show that $\phi $
is a communicating path with constants $c_{0}/K!$, $d$ and $m_{0}$, whereas
properties a)--c) and the fact that f) assumes $\min_{m=1,\ldots
,m_{0}-1}x_{u_{m}}\geq a$ show that property 1 of the lemma is satisfied.
Lastly, \eqref{cons} and the definition of $t_{0}$ directly imply that $%
\mathrm{Len}(\phi )\leq \sqrt{2K}(K+1)^{m_{0}}h$, which shows that $\phi $
satisfies property 2 of the lemma with $C_{x}=\sqrt{2K}(K+1)^{m_{0}}$ if $%
h\leq \mathrm{dist}(x,\tilde{{\mathcal{S}}}^{a})$. This completes the proof
of property f) and hence, of Claim 1.

We now proceed with the induction argument. 
Recall the definition of $N(x)$ given in \eqref{def-nx}. As our induction
hypothesis, we assume that there exists $N_{0}\in \{1,\ldots ,d\}$ and $%
C_{N_{0}}<\infty $ such that for every $x\in {\mathcal{S}}$ with $N(x)\leq
N_{0}$, there exist $y\in \tilde{{\mathcal{S}}}^{a}$, $t_{0}>0$, $C_{x}\leq
C_{N_{0}}$ and a communicating path $\phi $ on $[0,t_{0}]$ from $x$ to $y$
with constants $c=c_{0}$, $p=d$ and $F\leq N_{0}d$ that satisfy properties
(1) and (2) of the lemma.

\noindent \emph{Claim 2. } The induction hypothesis holds with $N_{0}=1$. 
\newline
\emph{Proof of Claim 2. } Suppose $N=N(x)=N_{0}=1$, where recall that $N(x)$
is defined by \eqref{def-nx}. Then since $x_{i}\geq x_{i+1},$ $x_{d}$ is the
only component such that $x_{d}<a$. Now, set $h\doteq a-x_{d}>0$, $%
t_{0}\doteq c_{M,M}h$ and let $\phi $, as constructed prior to Claim 1, be a
path on $[0,t_{0}]$ associated with $M$ and $h$ and with initial condition $%
x $. Also, define $y\doteq \phi (t_{0})$. Since the $\{u_{m},m=1,\ldots ,M\}$
are distinct and $u_{M}=d$, $\{u_{m},m=1,\ldots ,M-1\}\subset \{1,\ldots
,d-1\}$ and thus, $\min_{m=1,\ldots ,M-1}x_{u_{m}}\geq a$. Moreover, $%
d=u_{M}=j_{M-1,l_{M-1}^{\prime }}$ and hence, property d) of Claim 1 shows
that $\phi _{d}(t_{0})\geq x_{d}+h=a$. The last two assertions, when
combined with properties a) and b) of Claim 1, imply that $y\in \tilde{{%
\mathcal{S}}}^{a}$ and $\min_{m=1,\ldots ,M-1}\inf_{t\in \lbrack
0,t_{0}]}\phi _{u_{m}}(t)\geq \min_{m=1,\ldots ,M-1}x_{u_{m}}\geq a$, thus
verifying property 1 of the lemma. Since we also have $h\leq d(x,\tilde{{%
\mathcal{S}}}^{a})$, property f) of Claim 1 
shows that Claim 2 holds with $C_{1}\doteq \sqrt{2K}(K+1)^{M}\leq \sqrt{2K}%
(K+1)^{d}$.

\noindent \emph{Claim 3.} If the induction hypothesis holds for some $%
N_{0}\in \{1,\ldots ,d-1\}$, then it also holds for $N_{0}+1$. \newline
\emph{Proof of Claim 3.} Due to the induction hypothesis, it suffices to
consider $x$ such that $N=N(x)=N_{0}+1$. To prove the claim, 
we will first construct a communicating path that goes from $x$ to some $%
\bar{y}\in {\mathcal{S}}$ such that $N(\bar{y})\leq N_{0}$ 
and then invoke the induction hypothesis to construct a communicating path
from $\bar{y}$ to some $y\in \tilde{{\mathcal{S}}}^{a}$. 
The assumed ordering of $x$ and the fact that $N(x)=N_{0}+1$ imply that $%
x_{i}<a$ if and only if $i\geq d-N_{0}$. Define $\bar{m}\doteq \min \{m\in
1,\ldots ,M:u_{m}\geq d-N_{0}\}$, which is well defined because $u_{M}=d$,
and set $\bar{h}\doteq a-x_{u_{\bar{m}}}>0$ and $\bar{t}_{0}\doteq c_{\bar{m}%
,\bar{m}}\bar{h}$. Now, let $\bar{\phi}$ be the path on $[0,\bar{t}_{0}]$
and with initial condition $x$ as defined in \eqref{cons}, but with $\bar{m}$
and $\bar{h}$ taking the roles of $m_{0}$ and $h$, and set $\bar{y}\doteq 
\bar{\phi}(\bar{t}_{0})$. Then, by the choice of $\bar{m}$, 
\begin{equation}
\min_{m=1,\ldots ,\bar{m}-1}x_{u_{m}}\geq \min_{i=1,\ldots
,d-N_{0}-1}x_{i}\geq a.  \label{minx}
\end{equation}%
Since, in addition, $\bar{h}\leq d(x,\tilde{{\mathcal{S}}}^{a})$, property
f) of Claim 1 shows that $\bar{\phi}$ is a communicating path from $x$ to $%
\bar{y}$ with constants $c_{0},d$ and $\bar{m}\leq d$, which satisfies
property 1 of the lemma and also 
\begin{equation}
\mathrm{Len}(\bar{\phi})\leq C_{1}\mathrm{dist}(x,\tilde{{\mathcal{S}}}^{a}).
\label{len-bphi}
\end{equation}%
Recall that in the construction of the paths we have $u_{\bar{m}+1}=j_{\bar{m%
},l_{\bar{m}}^{\prime }}$ for $l_{\bar{m}}^{\prime }\in \left\{ 1,...,k_{%
\bar{m}}\right\} $. Thus property d) of Claim 1 shows that $\bar{y}_{u_{\bar{%
m}}}\geq x_{u_{\bar{m}}}+\bar{h}=a$, whereas \eqref{minx} and properties
a)--c) of Claim 1 show that $\bar{y}_{i}\geq a$ for $i\in \{1,\ldots
,d-N_{0}-1\}$. Indeed, for $i\in \{1,\ldots ,d-N_{0}-1\}\backslash \left\{
u_{1},...,u_{\bar{m}-1}\right\} $, $\bar{y}_{i}=x_{i}\geq a$, and for $i\in
\left\{ u_{1},...,u_{\bar{m}-1}\right\} $, properties a)--b) of Claim 1
implies $\bar{y}_{i}=\phi _{i}\left( \bar{t}_{0}\right) \geq \min \left\{
a,x_{i}\right\} \geq a$. This, in turn, implies, that $N(\bar{y})\leq
N-1=N_{0}$. Thus, applying the induction assumption, there exists a
communicating path $\tilde{\phi}$ from $\bar{y}$ to $y\in \tilde{{\mathcal{S}%
}}^{a}$ with constants $c_{0},d$ and $F\leq N_{0}d$ that satisfies property
1 of the lemma and for which 
\begin{equation}
\mathrm{Len}(\tilde{\phi})\leq C_{N_{0}}\mathrm{dist}(\bar{y},\tilde{{%
\mathcal{S}}}^{a}).  \label{len-tphi}
\end{equation}%
Let $\phi $ be the concatenation of $\bar{\phi}$ and $\tilde{\phi}$. Then it
is clear that $\phi $ is a communicating path from $x$ to $y\in \tilde{{%
\mathcal{S}}}^{a}$ with constants $c_{0}$, $d$ and $F\leq (N_{0}+1)d$, and
also satisfies property 1 of the lemma. Moreover, combining \eqref{len-tphi}
and \eqref{len-bphi} with the inequalities $\mathrm{dist}(\bar{y},\tilde{{%
\mathcal{S}}}^{a})\leq ||y-x||+\mathrm{dist}(x,\tilde{{\mathcal{S}}}^{a})$
and $||y-x||\leq \mathrm{Len}(\bar{\phi})$, it follows that 
\begin{equation*}
\mathrm{Len}(\phi )=\mathrm{Len}(\bar{\phi})+\mathrm{Len}(\tilde{\phi})\leq
\lbrack C_{1}+C_{N_{0}}(1+C_{1})]\mathrm{dist}\left( x,\tilde{{\mathcal{S}}}%
^{a}\right) .
\end{equation*}%
Thus, the induction hypothesis is satisfied for $N_{0}+1$ with $%
C_{N_{0}+1}\doteq \lbrack C_{1}+C_{N_{0}}(1+C_{1})]<\infty $. By induction,
the hypothesis holds for $N_{0}=d$, which proves the lemma.
\end{proof}

We now establish a uniform controllability property within any compact
subset of $\mathrm{int}$$({\mathcal{S}})$.

\begin{lemma}
\label{repre} Suppose that the family $\{\Gamma _{\mathbf{i}\mathbf{j}%
}(\cdot ),(\mathbf{i,}\mathbf{j})\in {\mathcal{J}}^{k},k=1,\ldots ,K\}$
satisfies Assumptions \ref{ue} and \ref{ass-kerg}. Then for any $u,w\in 
\mathcal{X}$, $u\neq w$, there exists a finite constant $M=M_{u,w}<\infty $,
and for $m=1,...,M-1$, there exist $k_{m}\in \left\{ 1,...,K\right\} $, $%
a_{m}\geq 0$ and $\left( \mathbf{i}_{m}\mathbf{,j}_{m}\right) \in \mathcal{J}%
^{k_{m}}$ such that $\mathfrak{M}_{\mathbf{i}_{m}\mathbf{j}_{m}}^{k_{m}}>0$
and 
\begin{equation}
e_{w}-e_{u}=\sum_{m=1}^{M-1}a_{m}\left( e_{\mathbf{j}_{m}}-e_{\mathbf{i}%
_{m}}\right) .  \label{repr}
\end{equation}
\end{lemma}

We first verify this assertion for Example \ref{eg3}. Without loss of
generality, we set $w=4$ and $u=1$. As before, we take $\left( \mathbf{i}_{1}%
\mathbf{,j}_{1}\right) =\left( 1,2\right) $, and $\left( \mathbf{i}_{2}%
\mathbf{,j}_{2}\right) =\left( \left( 1,2\right) ,\left( 3,4\right) \right) $%
. Then $\left( e_{\mathbf{j}_{1}}-e_{\mathbf{i}_{1}}\right) +\left( e_{%
\mathbf{j}_{2}}-e_{\mathbf{i}_{2}}\right) =e_{3}+e_{4}-2e_{1}$. To cancel
the term $e_{3}$, we further take $\left( \mathbf{i}_{3}\mathbf{,j}%
_{3}\right) =\left( 3,4\right) $. Then $\sum_{m=1}^{3}\left( e_{\mathbf{j}%
_{m}}-e_{\mathbf{i}_{m}}\right) =2\left( e_{4}-e_{1}\right) $, or $%
e_{4}-e_{1}=\sum_{m=1}^{3}\frac{1}{2}\left( e_{\mathbf{j}_{m}}-e_{\mathbf{i}%
_{m}}\right) $.

\begin{proof}[Proof of Lemma \protect\ref{repre}]
Fix $u, w \in \mathcal{X},$ $u\neq w$. Due to the assumed $K$-ergodicity, $w 
$ is $K$-accessible from $u$ and hence, there exist $M = M_{u,v}\in \{2,
\ldots, d\}$ and a sequence of distinct states $u=u_{1},...,u_{M}=w$, $%
k_{m}\in \left\{ 1,...,K\right\} $, $\left( \mathbf{i}_{m}\mathbf{,j}%
_{m}\right) \in \mathcal{J}^{k_{m}}$ for $m=1,...,M-1$, that satisfy the
properties in Definition \ref{k-acc}. If $K=1$, then $u=u_1$, $\mathbf{i}_m
= u_m$ and $\mathfrak{M}_{u_{m}u_{m+1}}^{1}>0$ for $m = 1, \ldots, M-1$, and
also $\mathbf{j}_{M-1} = u_M = w$. Thus, we can simply take $%
e_{w}-e_{u}=\sum_{m=1}^{M-1}\left( e_{u_{m+1}}-e_{u_{m}}\right)$.

Now, suppose $K\geq 2$. The proof of \eqref{repr} is more subtle in this
case, and consists of two main steps. In the first step, we show that for
any sequence $\{\left( \mathbf{i}_{m}\mathbf{,j}_{m}\right) \}_{m=1}^{M-1}$
as above, there exist nonnegative (and in fact strictly positive)
coefficients $\{b_{m}^{\left( u\right) }\}_{m=1}^{M-1}$ such that 
\begin{equation}
\sum_{m=1}^{M-1}b_{m}^{\left( u\right) }\left( e_{\mathbf{j}_{m}}-e_{\mathbf{%
i}_{m}}\right) =\sum_{i=1,i\neq u}^{d}c_{i}^{(u)}e_{i}-\left(
\sum_{i=1,i\neq u}^{d}c_{i}^{(u)}\right) e_{u},  \label{sum1}
\end{equation}%
where the constants $c_{i}^{(u)}\doteq \langle \sum_{m=1}^{M-1}b_{m}^{\left(
u\right) }\left( e_{\mathbf{j}_{m}}-e_{\mathbf{i}_{m}}\right) ,e_{i}\rangle,$
$i \in \{1, \ldots, d\}\setminus \{u\}$ satisfy 
\begin{equation}
c_{i}^{(u)}\geq 0,i\in \{1,\ldots ,d\}\setminus \{u, w\},\quad \mbox{ and }%
\quad c_{w}^{(u)}>0.  \label{ciu}
\end{equation}%
The second step shows that \eqref{repr} can be deduced from the fact that a
representation of the form \eqref{sum1}-\eqref{ciu} holds for all $u\neq w$.

We now turn to the proof of \eqref{sum1}-\eqref{ciu}. If $M=2$, it is
directly implied by Definition \ref{k-acc} and $\langle e_{\mathbf{i}%
_{1}},e_{u}\rangle \geq 1$ that \eqref{sum1}-\eqref{ciu} holds with $%
b_{1}^{(u)}=1$. Next, suppose $M>2$. The construction is now a bit more
involved. First, assume $\{b_{m}^{\left( u\right) }\}_{m=1}^{M-1}$ is any
strictly positive sequence. Then, since Definition \ref{k-acc} implies $%
i_{m,l}\in \{u_{1},\ldots ,u_{m}\}$ for $l=1,\ldots ,k_{m}$, $m=1,\ldots
,M-1 $ and $\langle e_{\mathbf{j}_{M-1}},e_{u_{M}}\rangle \geq 1$, we have 
\begin{equation}
c_{i}^{(u)}=\left\langle \sum_{m=1}^{M-1}b_{m}^{\left( u\right) }\left( e_{%
\mathbf{j}_{m}}-e_{\mathbf{i}_{m}}\right) ,e_{i}\right\rangle =\left\langle
\sum_{m=1}^{M-1}b_{m}^{\left( u\right) }e_{\mathbf{j}_{m}},e_{i}\right%
\rangle \geq 0,\quad i\in \{1,\ldots ,d\}\setminus \left\{ u_{m}\right\}
_{m=1}^{M-1},  \label{eqn:b_relation0}
\end{equation}%
with strict inequality for $i=u_{M}=w$, implying that $c_{w}^{(u)}>0$. Next,
we must argue that we can pick strictly positive $b_{m}^{(u)},m=1,\ldots
.M-1 $, such that the corresponding $\{c_{i}^{(u)}\}$ also satisfy $%
c_{i}^{(u)}\geq 0$ for $i=u_{2},u_{3},\ldots ,u_{M-1}$. For $m_{0}\in
\{2,\ldots ,M-1\}$, recall from Definition \ref{k-acc}(ii) that $\langle e_{%
\mathbf{i}_{m}},e_{i}\rangle \leq K$ for $i\in \{1,\ldots ,d\}$ and $\langle
e_{\mathbf{i}_{m}},e_{u_{m_{0}}}\rangle =0$ if $m\in \{1,\ldots ,m_{0}-1\}$,
and hence 
\begin{eqnarray}
c_{u_{m_{0}}}^{(u)}\doteq \left\langle \sum_{m=1}^{M-1}b_{m}^{\left(
u\right) }\left( e_{\mathbf{j}_{m}}-e_{\mathbf{i}_{m}}\right)
,e_{u_{m_{0}}}\right\rangle &\geq &b_{m_{0}-1}^{\left( u\right)
}\left\langle e_{\mathbf{j}_{m_{0}}-1},e_{u_{m_{0}}}\right\rangle
-\sum_{m=m_{0}}^{M-1}b_{m}^{\left( u\right) }\left\langle e_{\mathbf{i}%
_{m}},e_{u_{m_{0}}}\right\rangle  \notag  \label{eqn:b_relation} \\
&\geq &\kappa _{m_{0}-1}k_{m_{0}-1}b_{m_{0}-1}^{\left( u\right)
}-K\sum_{m=m_{0}}^{M-1}b_{m}^{(u)},
\end{eqnarray}%
where $\kappa _{m}\doteq \left\langle e_{\mathbf{j}_{m}},e_{u_{m+1}}\right%
\rangle /k_{m}$ for $m=1,\ldots ,M-2$. Defining $\{b_{m}^{\left( u\right)
}\}_{m=1}^{M-1}$ recursively according to 
\begin{eqnarray}
b_{M-1}^{\left( u\right) } &=&1  \notag  \label{bm} \\
\kappa _{M-2}b_{M-2}^{\left( u\right) } &=&Kb_{M-1}^{\left( u\right) } 
\notag \\
&\vdots & \\
\kappa _{1}b_{1}^{\left( u\right) } &=&K\left( b_{M-1}^{\left( u\right)
}+b_{M-2}^{\left( u\right) }+\cdots +b_{2}^{\left( u\right) }\right) , 
\notag
\end{eqnarray}%
it is clear that $\{b_{m}^{\left( u\right) }\}_{m=1}^{M-1}$ are strictly
positive, and $c_{u_{m_{0}}}\geq 0$ for all $m_{0}\in \{2,\ldots ,M-1\}$.
Together with \eqref{eqn:b_relation0} this shows that \eqref{sum1} and %
\eqref{ciu} hold for this choice of $\{b_{m}^{\left( u\right)
}\}_{m=1}^{M-1} $.

We now proceed to the second step of the proof. To obtain (\ref{repr}) from %
\eqref{sum1} we will eliminate the terms involving $e_{i}$, $i\neq u,w$, on
the right hand side of (\ref{sum1}). Now, we have assumed for each $s\neq
u,w $, that $w$ is $K$-accessible from $s$. Hence, applying the same
argument as in Step 1, but with $u$ replaced by $s$, we obtain the existence
of $M_{s}<\infty $, a sequence of jumps $\{(\mathbf{i}_{m}^{\left( s\right) }%
\mathbf{,j}_{m}^{\left( s\right) })\}_{m=1}^{M_{s}-1}$ and strictly positive
coefficients $\{b_{m}^{\left( s\right) }\}_{m=1}^{M_{s}-1}$ and $%
\{c_{i}^{\left( s\right) }\}_{i=1,i\neq u}^{d}$ with $c_{w}^{(s)}>0$ and 
\begin{equation}
\sum_{m=1}^{M_{s}-1}b_{m}^{\left( s\right) }\left( e_{\mathbf{j}_{m}^{\left(
s\right) }}-e_{\mathbf{i}_{m}^{\left( s\right) }}\right) =\sum_{i=1,i\neq
s}^{d}c_{i}^{\left( s\right) }e_{i}-\left( \sum_{i=1,i\neq
s}^{d}c_{i}^{\left( s\right) }\right) e_{s}.  \label{ssum}
\end{equation}%
Further, by a simple rescaling, it is clear that one can assume without loss
of generality that $\sum_{i=1,i\neq s}^{d}c_{i}^{(s)}=1$, which in turn
implies that 
\begin{equation}
\sum_{i=1,i\neq w,s}^{d}c_{i}^{(s)}<1,\qquad \mbox{ for every }s=1,\ldots
,d,s\neq w.  \label{cideficit}
\end{equation}%
It suffices to find nonnegative $\left\{ \theta _{s}\right\} _{s=1,s\neq
w}^{d}$ such that 
\begin{equation}
\sum_{s=1,s\neq w}^{d}\left( \theta _{s}\left( \sum_{i=1,i\neq
s}^{d}c_{i}^{\left( s\right) }e_{i}-e_{s}\right) \right) =e_{w}-e_{u},
\label{solv}
\end{equation}%
for then one can substitute (\ref{ssum}) into (\ref{solv}) to obtain (\ref%
{repr}) with $M\leq \sum_{s=1,s\neq w}^{d}(M_{s}-1)$ and coefficients $a_{m}$
of the form $\theta _{s}b_{m}^{(s)}$. For notational simplicity, below we
assume without loss of generality that $w=d$. Then, using \eqref{ssum}, it
is clear that both sides of \eqref{solv} are perpendicular to $%
\sum_{i=1}^{d}e_{i}$. Thus $\left\{ \theta _{s}\right\} _{s=1}^{d-1}$
satisfies \eqref{solv} if and only if the $i$th components of both sides of (%
\ref{solv}) are equal for $i\in \{1,\ldots ,d-1\}$, or in other words, if
the following system of linear equations is satisfied: 
\begin{equation*}
\left\{ 
\begin{array}{rcl}
\theta _{1} & = & c_{1}^{\left( 2\right) }\theta _{2}+\cdots +c_{1}^{\left(
d-1\right) }\theta _{d-1} \\ 
\theta _{2} & = & c_{2}^{\left( 1\right) }\theta _{1}+\cdots +c_{2}^{\left(
d-1\right) }\theta _{d-1} \\ 
\vdots &  &  \\ 
\theta _{u} & = & c_{u}^{\left( 1\right) }\theta _{1}+\cdots +c_{u}^{\left(
d\right) }\theta _{d-1}+1 \\ 
\vdots &  &  \\ 
\theta _{d-1} & = & c_{d-1}^{\left( 1\right) }\theta _{1}+\cdots
+c_{d-1}^{\left( d-2\right) }\theta _{d-2}.%
\end{array}%
\right.
\end{equation*}%
This system can be expressed more concisely as $[I_{d-1}-C]\theta =\tilde{e}%
_{u}$, where $I_{d-1}$ is the $(d-1)\times (d-1)$-dimensional identity
matrix, $\tilde{e}_{u}$ is the $(d-1)$-dimensional vector that has a one in
the $u$th component and $0$ elsewhere, and $C$ is the $(d-1)\times (d-1)$
matrix given by 
\begin{equation*}
C\doteq \left( 
\begin{array}{cccc}
0 & c_{1}^{\left( 2\right) } & ... & c_{1}^{\left( d\right) } \\ 
c_{2}^{(1)} & \ddots &  & \vdots \\ 
\vdots &  & \ddots & c_{d-2}^{(d-1)} \\ 
c_{d-1}^{\left( 1\right) } & ... & c_{d-1}^{\left( d-2\right) } & 0%
\end{array}%
\right) .
\end{equation*}%
Note that $C$ is a non-negative matrix with row sums strictly less than $1$
due to \eqref{cideficit}. Thus, applying Lemma \ref{linsys} below with $%
N=d-1 $ and $y=\tilde{e}_{u}$, we conclude that the system of linear
equations has a unique solution, which is also nonnegative. This completes
the proof of the lemma.
\end{proof}

\begin{lemma}
\label{linsys} Suppose that $A=I_{N}-C$, where $I_{N}$ is an $N\times N$
identity matrix and $C=\left( c_{ij}\right) _{i,j=1}^{N}$ for some $N\in 
\mathbb{N}$, such that $c_{ii}=0$, $c_{ij}\geq 0$ and $%
\sum_{j=1}^{N}c_{ij}<1 $. Also, let $y\in \mathbb{[}0,\mathbb{\infty )}^{N}$%
. Then the system of linear equations $Ax=y$ has a unique nonnegative
solution $\left\{ x_{i}\right\} _{i=1}^{N}$.
\end{lemma}

\begin{proof}
The spectral radius of $C$ is less than $1$ since its matrix norm is less
than $1$. Therefore $\det A>0$, and $A^{-1}$ exists. The fact that $A^{-1}$
is a positive matrix follows from a general result in inverse positivity 
\cite[Theorem 6.3.8]{BP}. The nonnegativity of $x$ then follows from the
nonnegativity of $y$.
\end{proof}

\begin{lemma}
\label{commint} Suppose Assumption \ref{ue} and Assumption \ref{ass-kerg}
hold. Then for any $a>0$, there exist $c>0,C^{\prime }<\infty $ and $\bar{F}%
\in \mathbb{N}$ such that for any $x,y\in \tilde{\mathcal{S}}^{a}$, there
exists a communicating path $\phi $ from $x$ to $y$ with constants $c$, $d$
and $\bar{F}$ such that $\phi $ lies in $\tilde{\mathcal{S}}^{a/2}$ and
satisfies $\mathrm{Len}(\phi )\leq C^{\prime }||x-y||$.
\end{lemma}

\begin{proof}
Fix $a>0$. Then we observe the following elementary fact: there exists $%
C_{1}<\infty $, such that for every $x,y\in {\mathcal{S}}$, there exists a
continuous piecewise linear path $\phi _{0}$ from $x$ to $y$ that lies in ${%
\mathcal{S}}$, uses only velocities in the directions $\left\{
e_{j}-e_{i}:i,j\in \mathcal{X}\right\} $, and for which $\mathrm{Len}(\phi
_{0})\leq C_{1}\left\Vert x-y\right\Vert $. Since $\tilde{\mathcal{S}}^{a}$
is similar to ${\mathcal{S}}$, a rescaling implies that for every $x,y\in 
\tilde{\mathcal{S}}^{a}$, there is a continuous piecewise linear path $\phi
_{0}$ from $x$ to $y$ that lies in $\tilde{\mathcal{S}}^{a}$ and only uses
velocities in the directions $\left\{ e_{j}-e_{i}:i,j\in \mathcal{X}\right\} 
$, and for which $\mathrm{Len}(\phi _{0})\leq C_{1}\left\Vert x-y\right\Vert 
$. Given such a path $\phi _{0}$ with $\phi _{0}\left( 0\right) =x$, let $%
M<\infty $, and $0=t_{1}<\cdots <t_{M}$ and $\left\{ u_{m}\right\}
_{m=1}^{M-1},\left\{ w_{m}\right\} _{m=1}^{M-1}\in \mathcal{X}$ be such that 
\begin{equation*}
\dot{\phi}_{0}\left( t\right) =e_{w_{m}}-e_{u_{m}},\quad t\in
(t_{m},t_{m+1}),m=1,\ldots ,M-1,
\end{equation*}%
where a uniform bound on $M$ can be assumed. We now use Lemma \ref{repre} to
replicate these velocities using jumps with positive rates. For each $%
m=1,\ldots ,M-1$, there exist $M_{m}\in \mathbb{N}$, $\{(\mathbf{i}%
_{k}^{\left( m\right) },\mathbf{j}_{k}^{\left( m\right) })\}_{m=1}^{M_{m}-1}$
and $\{a_{k}^{\left( m\right) }\}_{m=1}^{M_{m}-1}$ such that $%
e_{w_{m}}-e_{u_{m}}=\sum_{r=1}^{M_{m}-1}a_{r}^{\left( m\right) }(e_{\mathbf{j%
}_{r}^{\left( m\right) }}-e_{\mathbf{i}_{r}^{\left( m\right) }})$. With the
appropriate partition of $(t_{m},t_{m+1})$, we can construct a trajectory $%
\phi $ that uses only these velocities and satisfies $\phi (t_{m})=\phi
_{0}\left( t_{m}\right) $ and $\phi (t_{m+1})=\phi _{0}\left( t_{m+1}\right) 
$. If for any $s\in (t_{m},t_{m+1})$ we have $\left\Vert \phi (s)-\phi
_{0}\left( s\right) \right\Vert >a/2$, then we can partition $%
(t_{m},t_{m+1}) $ into an integral number $K$ of smaller segments on which
we replicate the velocity $e_{w_{m}}-e_{u_{m}}$, and guarantee $\phi
(t_{m}+k[t_{m+1}-t_{m}]/K)=\phi _{0}\left( t_{m}+k[t_{m+1}-t_{m}]/K\right) $
for $k=1,\ldots ,K$. For large enough $K$ this implies $\left\Vert \phi
(s)-\phi _{0}\left( s\right) \right\Vert \leq a/2$ for all $t\in \lbrack
0,T_{M}]$. Since ${\mathcal{X}}$ is finite there is a maximum velocity used
in this process, and hence we can assume a uniform bound on $K$ that depends
only on $a$, and also the existence of $C<\infty $ such that $\mathrm{Len}%
(\phi )\leq C\mathrm{Len}(\phi _{0})$, and so can take $C^{\prime }=CC_{1}$.

The path so constructed satisfies $\phi \in AC\left( \left[ 0,t_{M}\right] :%
\mathcal{S}^{a/2}\right) $. By (\ref{alpha1}), (\ref{lambda}) and the fact
that $\mathfrak{M}_{\mathbf{i}_{m}\mathbf{j}_{m}}^{k_{m}}>0$, with $c_{0}>0$
as in (\ref{def-c0}) we have that for all $s\in \left[ 0,t_{M}\right] $, $%
\lambda _{v}\left( \phi \left( s\right) \right) \geq \frac{1}{K!}\left( 
\frac{a}{2}\right) ^{d}c_{0}$. Since $\min_{i=1,\ldots ,d}y_{i}\geq a$, $%
\phi $ is a communicating path on $[0,t_{M}]$ from $x$ to $y$ with constants 
$c\doteq c_{0}/2^{d}K!$, $p=d$, and uniformly bounded $F$.
\end{proof}

We now complete the proof of Proposition \ref{ver1}.

\begin{proof}[Proof of Proposition \protect\ref{ver1}.]
When $K=1$, this follows from the stronger result proved in Lemma \ref{move}%
. Hence, suppose $K \geq 2$. Given $x\in {\mathcal{S}}$ and $y\in \mathrm{int%
}({\mathcal{S}})$, let $0<a<\min (1/((K+1)^{d-1}d),\min_{i=1,\ldots
,d}y_{i}) $. By Lemma \ref{in}, there exist $z\in \tilde{{\mathcal{S}}}^{a}$%
, a communicating path from $x$ to $z$ with constants independent of $x$ and 
$z$. We also have $y\in \tilde{{\mathcal{S}}}^{a}$, and so by Lemma \ref%
{commint} there exists a communicating path from $z$ to $y$, with constants
independent of $z$ and $y$. It is straightforward to see that the
concatenation of these two paths is a communicating path from $x$ to $y$
with constants independent of $x$ and $y$. Moreover, it follows from the
properties stated in Lemmas \ref{in} and \ref{commint} that there exists $%
C^{\prime }< \infty$ such that the family of communicating paths thus
constructed satisfies \eqref{length}. This proves that Property \ref%
{prop-communicate} is satisfied.
\end{proof}

\subsection{Estimates on the Jump Rates}

\label{subs-ratioest}

In this section we derive certain estimates on the jump rates that are
satisfied under our assumptions on the transition rates. These estimates are
used in the subsequent proofs. We recall that ${\mathcal{N}}_{v}\doteq
\{i:v_{i}<0\}$.

\begin{lemma}
\label{lem-estimates} Suppose the family $\{\Gamma_{\mathbf{ij}}^k (\cdot), (%
\mathbf{i,j}) \in {\mathcal{J}}^k, k = 1, \ldots, K\}$ satisfies Assumption %
\ref{intsys} and Assumption \ref{ue}, and let $\{\lambda_v, v \in {\mathcal{V%
}}\}$ be the associated jump rates. Then the following assertions hold.

\begin{enumerate}
\item There exists $\widehat{C} < \infty$ such that for every $v \in {%
\mathcal{V}}$, 
\begin{equation*}
\lambda_v (x) \leq \widehat{C} \left( \prod_{i\in {\mathcal{N}}_v} x_i
\right), \quad x \in {\mathcal{S}}.
\end{equation*}

\item There exists a continuous function $\bar{C}: [0,\infty) \mapsto
[0,\infty)$ such that $\bar{C} (r) \rightarrow 1$ as $r \rightarrow 0$ such
that for every $v \in {\mathcal{V}}$ and $x, y \in {\mathcal{S}}$, 
\begin{equation*}
\frac{\lambda_v(x)}{\lambda_v(y)} \leq \bar{C}(||x-y||) \prod_{i \in {%
\mathcal{X}}: y_i < x_i} \left( \frac{x_i}{y_i} \right)^K.
\end{equation*}

\item For every $v \in {\mathcal{V}}$, either $\lambda_v \equiv 0$ or 
\begin{equation*}
\lambda_v (x) \geq \frac{c_0}{K!} \left( \prod_{l=1}^k x_i \right) \geq 
\frac{c_0}{K!} \prod_{i=1}^d x_i^K, \quad x \in {\mathcal{S}},
\end{equation*}
and hence, for every $a > 0$, $\inf_{x \in \tilde{{\mathcal{S}}}^a}
\lambda_v(x) > 0$ and $\inf_{x \in {\mathcal{S}}^a} \lambda_v (x) > 0$,
where $\tilde{{\mathcal{S}}}^a$ and ${\mathcal{S}}^a$ are defined in %
\eqref{tS} and \eqref{S}, respectively.

\item If, in addition, Assumption \ref{ass-simjumps} holds, then there exist 
$\bar{c}>0$ such that for every $v\in {\mathcal{V}}$, either $\lambda
_{v}(\cdot )\equiv 0$ or there exist $r_{i}\geq 1$, $i\in {\mathcal{N}}_{v}$%
, with $\sum_{i\in {\mathcal{N}}_{v}}r_{i}\leq K$ such that for every $x\in {%
\mathcal{S}}$ and $y\in $\emph{int}$({\mathcal{S}})$, 
\begin{equation}
\frac{\lambda _{v}(x)}{\lambda _{v}(y)}\geq \bar{c}\prod_{i\in {\mathcal{N}}%
_{v}}\left( \frac{x_{i}}{y_{i}}\right) ^{r_{i}}.  \label{est-lbound}
\end{equation}
\end{enumerate}
\end{lemma}

\begin{proof}
We start with the proof of the first property. For any $v\in {\mathcal{V}}$,
for every $k=1,\ldots ,K$ and $(\mathbf{i},\mathbf{j})\in {\mathcal{J}}^{k}$
such that $e_{\mathbf{j}}-e_{\mathbf{i}}=v$, \eqref{nvinc} shows that each $%
i\in {\mathcal{N}}_{v}$ appears at least once in $\{i_{l},l=1,\ldots ,k\}$,
and so $\prod_{l=1}^{k}x_{i_{l}}\leq \prod_{i\in {\mathcal{N}}_{v}}x_{i}$.
Substituting this and the bound on the transition rates in \eqref{def-R0}
into the definition of $\lambda _{v}$ in \eqref{lambda}, it follows that the
first property holds with $\widehat{C}\doteq R_{0}|{\mathcal{J}}|$.

Next, fix $v\in {\mathcal{V}}$ and note that by Assumption \ref{ue}, we can
rewrite $\lambda _{v}$ from \eqref{lambda} as 
\begin{equation}
\lambda _{v}(x)=\sum_{k=1}^{K}\sum_{\overset{(\mathbf{i,j})\in {\mathcal{J}}%
_{+}^{k}:}{e_{\mathbf{j}}-e_{\mathbf{i}}=v}}\alpha _{\mathbf{i}\mathbf{j}%
}^{k}(x),\quad x\in {\mathcal{S}},  \label{alt-lambda}
\end{equation}%
where recall the definition of ${\mathcal{J}}_{+}^{k}$ from \eqref{nk}. Now,
let $k^{\ast }\doteq k^{\ast }(x,y)\in \{1,\ldots ,K\}$, $(\mathbf{i^{\ast }}%
,\mathbf{j}^{\ast })=(\mathbf{i}^{\ast }(x,y),\mathbf{j}^{\ast }(x,y))\in {%
\mathcal{J}}_{+}^{k^{\ast }}$ be such that 
\begin{equation*}
\frac{\alpha _{\mathbf{i}^{\ast }\mathbf{j}^{\ast }}^{k^{\ast }}(x)}{\alpha
_{\mathbf{i}^{\ast }\mathbf{j}^{\ast }}^{k^{\ast }}(y)}=\max_{k=1,\ldots
,K}\max_{\overset{(\mathbf{i,j})\in {\mathcal{J}}_{+}^{k}:}{e_{\mathbf{j}%
}-e_{\mathbf{i}}=v}}\frac{\alpha _{\mathbf{ij}}^{k}(x)}{\alpha _{\mathbf{ij}%
}^{k}(y)}.
\end{equation*}%
Then, since for any finite index set ${\mathcal{I}}$ and numbers $%
a_{i},b_{i}>0,i\in {\mathcal{I}}$, 
\begin{equation*}
\frac{\sum_{i\in {\mathcal{I}}}a_{i}}{\sum_{i\in {\mathcal{I}}}b_{i}}\leq
\max_{i\in {\mathcal{I}}}\frac{a_{i}}{b_{i}},
\end{equation*}%
from \eqref{alt-lambda} and \eqref{alpha1} it follows that 
\begin{equation}
\frac{\lambda _{v}(x)}{\lambda _{v}(y)}\leq \frac{\alpha _{\mathbf{i}^{\ast }%
\mathbf{j}^{\ast }}^{k^{\ast }}(x)}{\alpha _{\mathbf{i}^{\ast }\mathbf{j}%
^{\ast }}^{k^{\ast }}(y)}=\left( \prod_{l=1}^{k^{\ast }}\frac{x_{i_{l}^{\ast
}}}{y_{i_{l}^{\ast }}}\right) \frac{\Gamma _{\mathbf{i}^{\ast }\mathbf{j}%
^{\ast }}^{k^{\ast }}(x)}{\Gamma _{\mathbf{i}^{\ast }\mathbf{j}^{\ast
}}^{k^{\ast }}(y)}.  \label{alpha-ineq}
\end{equation}%
Combining the lower bounds on $\Gamma _{\mathbf{i}^{\ast }\mathbf{j}^{\ast
}}^{k^{\ast }}$ in \eqref{def-c0}, and letting $C_{1}$ denote the maximum of
the Lipschitz constants of $\Gamma _{\mathbf{ij}}^{k}$, $(\mathbf{i,j})\in {%
\mathcal{J}}^{k},k=1,\ldots ,K$ (which is finite by Assumption \ref{intsys}%
), we obtain the inequality 
\begin{equation}
\frac{\Gamma _{\mathbf{i^{\ast }j^{\ast }}}^{k^{\ast }}(x)}{\Gamma _{\mathbf{%
i^{\ast }j^{\ast }}}^{k^{\ast }}(y)}=1+\frac{\Gamma _{\mathbf{i^{\ast
}j^{\ast }}}^{k^{\ast }}(x)-\Gamma _{\mathbf{i^{\ast }j^{\ast }}}^{k^{\ast
}}(y)}{\Gamma _{\mathbf{i^{\ast }j^{\ast }}}^{k^{\ast }}(y)}\leq \left( 1+%
\frac{C_{1}}{c_{0}}||x-y||\right) .  \label{ineq-gammahere}
\end{equation}%
On the other hand, for any $k=1,\ldots ,K$ and $(\mathbf{i},\mathbf{j})\in {%
\mathcal{J}}^{k}$, 
\begin{equation*}
\prod_{l=1}^{k}\left( \frac{x_{i_{l}}}{y_{i_{l}}}\right) \leq
\prod_{l=1,\ldots ,k:x_{i_{l}}>y_{i_{l}}}\left( \frac{x_{i_{l}}}{y_{i_{l}}}%
\right) \leq \prod_{i\in {\mathcal{X}}:x_{i}>y_{i}}\left( \frac{x_{i}}{y_{i}}%
\right) ^{K}.
\end{equation*}%
Substituting this and \eqref{ineq-gammahere} into \eqref{alpha-ineq}, we see
that the second property is satisfied by the function $\bar{C}(r)\doteq
1+C_{1}r/c_{0}$, $r\geq 0$.

For the remaining two properties we can fix $v \in {\mathcal{V}}$ and assume
without loss of generality that $\lambda_v$ is not identically zero. Then,
by the continuity of $\lambda_v(\cdot)$, which follows from Assumption \ref%
{intsys}, there exists $x \in \mathrm{int}({\mathcal{S}})$ such that $%
\lambda_v (x) > 0$. In turn, from the form of $\lambda_v$ in %
\eqref{alt-lambda}, it follows that there exists $k = 1, \ldots, K$ and $(%
\mathbf{i}, \mathbf{j}) \in {\mathcal{J}}^k_+$ such that $v = e_{\mathbf{j}}
- e_{\mathbf{i}}$ and $\Gamma_{\mathbf{i} \mathbf{j}} (x) > 0$. The
relations \eqref{lambda} and \eqref{alpha1} and the definition of $c_0$ in %
\eqref{def-c0} then show that 
\begin{equation*}
\lambda_v (x) \geq \alpha_{\mathbf{i}\mathbf{j}}^k (x) \geq \frac{c_0}{K!}
\prod_{l=1}^k x_{i_l}, \quad x \in {\mathcal{S}},
\end{equation*}
which implies the inequality in property 3) because for each $i \in \{1,
\ldots, d\}$, $x_i \in [0,1]$ and $|\{l = 1, \ldots, k: i_l = i\}| \leq k
\leq K$. The bound $\inf_{x \in {\mathcal{S}}^a} \lambda_v (x) > 0$ and $%
\inf_{x \in {\mathcal{S}}^a} \lambda_v (x) > 0$ are an immediate consequence
of the inequality, the definition of $\tilde{{\mathcal{S}}}^a$ and the fact
that ${\mathcal{S}}^a \subset \tilde{{\mathcal{S}}}^{a/2}$.

For the last property, first note that for any $x,y\in {\mathcal{S}}$, using %
\eqref{alt-lambda}, \eqref{def-c0}, \eqref{def-R0} and \eqref{alpha1}, we
have 
\begin{equation}
\displaystyle\frac{\lambda _{v}(x)}{\lambda _{v}(y)}\geq \frac{c_{0}}{R_{0}K!%
}\displaystyle\frac{\sum_{k=1}^{K}\sum_{(\mathbf{i},\mathbf{j})\in {\mathcal{%
J}}_{+}^{k}:e_{\mathbf{j}}-e_{\mathbf{i}}=v}\left(
\prod_{l=1}^{k}x_{i_{l}}\right) }{\sum_{k=1}^{K}\sum_{(\mathbf{i},\mathbf{j}%
)\in {\mathcal{J}}_{+}^{k}:e_{\mathbf{j}}-e_{\mathbf{i}}=v}\left(
\prod_{l=1}^{k}y_{i_{l}}\right) }.  \label{est-intermed}
\end{equation}%
If Assumption \ref{ass-simjumps}(1) holds then let $(\mathbf{i}^{\ast },%
\mathbf{j}^{\ast })\in {\mathcal{J}}_{+}$ be as in the assumption, and
define $r_{j}=r_{j}(v)\doteq |\langle v,e_{j}\rangle |$ for $j\in {\mathcal{N%
}}_{v}$. Then clearly $(\mathbf{i}^{\ast },\mathbf{j}^{\ast })\in {\mathcal{J%
}}_{+}^{k^{\ast }}$, where $k^{\ast }\doteq \sum_{j\in {\mathcal{N}}%
_{v}}r_{j}$, and 
\begin{equation*}
\sum_{k=1}^{K}\sum_{(\mathbf{i},\mathbf{j})\in {\mathcal{J}}_{+}^{k}:e_{%
\mathbf{j}}-e_{\mathbf{i}}=v}\left( \prod_{l=1}^{k}x_{i_{l}}\right) \geq
\prod_{l=1}^{k^{\ast }}x_{i_{l}^{\ast }}=\prod_{j\in {\mathcal{N}}%
_{v}}\left( x_{j}\right) ^{r_{j}}.
\end{equation*}%
On the other hand, \eqref{nvinc} along with the fact that $y\in \lbrack
0,1]^{d}$ implies that 
\begin{equation*}
\sum_{k=1}^{K}\sum_{(\mathbf{i},\mathbf{j})\in {\mathcal{J}}_{+}^{k}:e_{%
\mathbf{j}}-e_{\mathbf{i}}=v}\left( \prod_{l=1}^{k}y_{i_{l}}\right) \leq |{%
\mathcal{J}}_{+}|\prod_{j\in {\mathcal{N}}_{v}}\left( y_{i_{j}}\right)
^{r_{j}}.
\end{equation*}%
Substituting the last two inequalities into \eqref{est-intermed}, we see
that \eqref{est-lbound} holds with $\bar{c}=c_{0}/R_{0}K!|{\mathcal{J}}_{+}|$%
, and $r_{j}$, $j\in {\mathcal{N}}_{v}$, as specified above. On the other
hand, if Assumption \ref{ass-simjumps}(2) is satisfied then let $r_{j},j\in {%
\mathcal{N}}_{v}$, and $k^{\ast }=\sum_{j\in {\mathcal{N}}_{v}}r_{j}$ be as
stated in the assumption. Notice that for all $(\mathbf{i},\mathbf{j})\in {%
\mathcal{J}}_{+}^{k^{\ast }}$ such that $e_{\mathbf{j}} - e_{\mathbf{i}} = v$%
, the equality $\prod_{l=1}^{k}x_{i_{l}}=\prod_{j\in {\mathcal{N}}%
_{v}}\left( x_{j}\right) ^{r_{j}}$ holds. Therefore, there is some constant $%
C^{\prime }<\infty $, that only depends on $r_{j},j\in {\mathcal{N}}_{v}$,
such that 
\begin{equation*}
\sum_{k=1}^{K}\sum_{(\mathbf{i},\mathbf{j})\in {\mathcal{J}}_{+}^{k}:e_{%
\mathbf{j}}-e_{\mathbf{i}}=v}\left( \prod_{l=1}^{k}x_{i_{l}}\right) =\sum_{(%
\mathbf{i},\mathbf{j})\in {\mathcal{J}}_{+}^{k^{\ast }}:e_{\mathbf{j}}-e_{%
\mathbf{i}}=v}\left( \prod_{l=1}^{k}x_{i_{l}}\right) =C^{\prime }\prod_{j\in 
{\mathcal{N}}_{v}}\left( x_{j}\right) ^{r_{j}},
\end{equation*}%
with the same equality also holding when $x$ is replaced by $y$. When
substituted back into \eqref{est-intermed}, this shows that %
\eqref{est-lbound} holds with $\bar{c}=c_{0}/R_{0}K!$ and the given $r_{j}$, 
$j\in {\mathcal{N}}_{v}$. This completes the proof of the lemma.
\end{proof}

\subsection{A property of the LLN trajectory}

\label{subs-lln-inward}

In this section, we show that the communication and growth conditions on the
limit rates $\{\lambda_v(\cdot), v \in {\mathcal{V}}\}$ imply that the
associated LLN trajectory $\mu$ has the following property, which is
crucially used in the proof of the large deviation lower bound.

\begin{property}
\label{cond-lln} There exist constants $b>0$ and $D \in [1,\infty)$ such
that for any $x\in \mathcal{S}$, the associated LLN path $\mu$ that solves
the ODE \eqref{1} and starts at $x$ is such that for every $i \in {\mathcal{X%
}}$, 
\begin{equation}  \label{ineq-lln}
\mu _{i}\left( t\right) \geq bt^{D}, \qquad t\in \left[ 0,1\right].
\end{equation}
\end{property}

We now state the main result of this section.

\begin{proposition}
\label{prop-lln} Suppose the family of jump rates $\{\lambda _{v}(\cdot
),v\in {\mathcal{V}}\}$ satisfies Property \ref{prop-lambda}, Property \ref%
{prop-communicate} and properties (1) and (4) of Lemma \ref{lem-estimates}.
Then Property \ref{cond-lln} is also satisfied.
\end{proposition}

To provide insight into the proof of Proposition \ref{prop-lln}, we first
show why the conclusion holds for the specific $K$-ergodic particle system
with $d=4, K=2$ introduced in Example \ref{eg3} when $c_i = 1$ for $%
i=1,\ldots, 6$.

\noindent \textbf{Example \ref{eg3} cont'd.} 
Let $x=\mu \left( 0\right) $, and assume without loss of generality that $%
x_{1}\geq x_{2}\geq x_{3}\geq x_{4}$, and therefore that $x_{1}\geq 1/4$. We
start by establishing a basic inequality for $\mu _{i}(t)$. Recall the form
of $\lambda _{v}(\cdot ),v\in {\mathcal{V}}$, given in Example \ref{eg3},
and note that the ODE (\ref{1}) implies 
\begin{equation*}
\dot{\mu}_{1}(t)=\sum_{v\in \mathcal{V}}\left\langle v,e_{1}\right\rangle
\lambda _{v}\left( \mu \left( t\right) \right) \geq -\mu _{1}(t)-\frac{1}{2}%
\mu _{1}(t)\mu _{2}(t)\geq -\frac{3}{2}\mu _{1}(t)\text{.}
\end{equation*}%
Thus, for $t\in \lbrack 0,1]$, $\mu _{1}\left( t\right) \geq x_{1}e^{-\frac{3%
}{2}t}\geq b_{1}$, where $b_{1}\doteq e^{-\frac{3}{2}}/4$. Then we have 
\begin{equation*}
\dot{\mu}_{2}(t)=\sum_{v\in \mathcal{V}}\left\langle v,e_{2}\right\rangle
\lambda _{v}\left( \mu \left( t\right) \right) \geq \mu _{1}(t)-\mu _{2}(t)-%
\frac{1}{2}\mu _{1}(t)\mu _{2}(t)\geq b_{1}-\frac{3}{2}\mu _{2}(t)\text{,}
\end{equation*}%
which implies $\mu _{2}\left( t\right) \geq \frac{2}{3}b_{1}(1-e^{-\frac{3}{2%
}t})\geq b_{2}t$ for some $b_{2}>0$ (for example, $b_{2}=b_{1}/2$).
Substituting the last two bounds into the equations for $i=3,4$, we obtain 
\begin{equation*}
\dot{\mu}_{i}(t)=\sum_{v\in \mathcal{V}}\left\langle v,e_{i}\right\rangle
\lambda _{v}\left( \mu \left( t\right) \right) \geq \frac{1}{2}\mu
_{1}(t)\mu _{2}(t)-\mu _{i}(t)-\frac{1}{2}\mu _{3}(t)\mu _{4}(t)\geq \frac{%
b_{1}b_{2}}{2}t-\frac{3}{2}\mu _{i}(t),
\end{equation*}%
and hence, for $t\in \lbrack 0,1]$, $\mu _{i}\left( t\right) \geq b_{3}t^{2}$
for some $b_{3}>0$ (for example, $b_{3}=b_{1}b_{2}e^{-3/2}/4$). Thus, we
have shown that Example \ref{eg3}, with the chosen parameters, satisfies
Property \ref{cond-lln}.

The proof for the general case is more technical and is given below.

\begin{proof}[Proof of Proposition \protect\ref{prop-lln}]
Define $\delta _{1}\doteq \max \left\{ \left\vert \left\langle
e_{j},v\right\rangle \right\vert :j\in \mathcal{X}\text{, }v\in \mathcal{V}%
\right\} <\infty $, and let $R_{1}\doteq \delta _{1}\widehat{C}\left\vert 
\mathcal{V}\right\vert $, where $\widehat{C}$ is the constant in property
(1) of Lemma \ref{lem-estimates}. Then the fact that $\mu $ solves the ODE (%
\ref{1}) implies that for $i\in {\mathcal{X}}$ and\ $t\in \left[ 0,1\right] $%
, 
\begin{eqnarray}
\dot{\mu}_{i}\left( t\right) &=&\sum_{v\in \mathcal{V}}\left\langle
v,e_{i}\right\rangle \lambda _{v}\left( \mu \left( t\right) \right)  \notag
\\
&=&\sum_{v\in \mathcal{V}:\langle v,e_{i}\rangle >0}\left\langle
v,e_{i}\right\rangle \lambda _{v}\left( \mu \left( t\right) \right)
+\sum_{v\in \mathcal{V}:\langle v,e_{i}\rangle <0}\left\langle
v,e_{i}\right\rangle \lambda _{v}\left( \mu \left( t\right) \right)  \notag
\\
&\geq &\sum_{v\in \mathcal{V}:\langle v,e_{i}\rangle >0}\left\langle
v,e_{i}\right\rangle \lambda _{v}\left( \mu \left( t\right) \right)
-R_{1}\mu _{i}(t).  \label{9.3.1}
\end{eqnarray}%
Since each $\lambda _{v}(\cdot )$ is nonnegative, \eqref{9.3.1} and the
comparison principle for ODEs imply that 
\begin{equation}
\mu _{i}\left( t\right) \geq \mu _{i}\left( 0\right) e^{-R_{1}t},\text{ \ }%
t\in \left[ 0,1\right] ,i\in {\mathcal{X}}.  \label{1lbd}
\end{equation}

In order to use \eqref{9.3.1} to show that \eqref{ineq-lln} holds (for
suitable $b>0$ and $D<\infty$), we first obtain a lower bound on $\lambda
_{v}(\mu (t))$ by comparing it to $\lambda _{v}(\phi (t))$ for a suitable
communicating path $\phi$, and then apply the estimate \eqref{est-lbound}.
Define $y$ $\dot{=}\left( \frac{1}{d},...,\frac{1}{d}\right) $, and note
that by Property \ref{prop-communicate} and Remark \ref{flex}, there exists
a communicating path $\phi $ from $\mu (0)=x$ to $y$ on $[0,1]$. Let $\phi $
admit a representation in terms of $F<\infty ,\left\{ t_{m}\right\}
_{m=0}^{F},\left\{ v_{m}\right\} _{m=1}^{F}$ and $\left\{ U_{m}\right\}
_{m=1}^{F}$ as in (\ref{CP}) and for $m=0,1,\ldots ,F$, denote $y^{\left(
m\right) }\doteq \phi \left( t_{m}\right) $. Then, applying the inequality
in property ii) of Definition \ref{def-compath} with $s=t_{m-1}$, $m = 1,
\ldots, F$, and $y = \frac{1}{d} \sum_{i=1}^d e_i$, we see that there exists 
$c^{\prime }\in (0,1)$ such that 
\begin{equation}
\lambda _{v_{m}}\left( y^{\left( m-1\right) }\right) \geq c^{\prime },\quad
m=1,\ldots ,F.  \label{lvm2}
\end{equation}
Now define ${\mathcal{X}}_0 \doteq \emptyset$ and for $l = 1, \ldots. F$,
define ${\mathcal{X}}_l \doteq \{j \in {\mathcal{X}}: \mu_j(0) < y_j^{(l)}\}$%
. Then, for $m = 1, \ldots, F$, by the continuity of $\lambda_{v_m}$
(Property \ref{prop-lambda}), there exists $\tilde{y}^{(m-1)} \in \mathrm{int%
} ({\mathcal{S}})$ sufficiently close to $y^{(m-1)}$ such that 
\begin{equation}
\lambda _{v_{m}}\left( \tilde{y}^{\left( m-1\right) }\right) \geq \frac{%
c^{\prime }}{2},\quad m=1,\ldots ,F,  \label{lvm3}
\end{equation}
and 
\begin{equation}  \label{def-xl}
\{j \in {\mathcal{X}}: \mu_j(0) < \tilde{y}_j^{(m-1)}\} = {\mathcal{X}}%
_{m-1}.
\end{equation}

Now, fix $m\in \{1,\ldots ,F\}$. Let $\bar{c}>0$ and $r_{j}=r_{j}(v_{m})$, $%
j\in {\mathcal{N}}_{v_{m}}$, be the constants in property (4) of Lemma \ref%
{lem-estimates}. 
For $t\in \lbrack 0,1]$, we can first apply the estimate \eqref{est-lbound}
with $x=\mu (t)$ and $y=\tilde{y}^{(m-1)}$ and use \eqref{lvm3}, and then
use \eqref{1lbd} and $\sum_{j\in {\mathcal{X}}}r_{j}\leq K$ to obtain 
\begin{equation*}
\lambda _{v_{m}}\left( \mu (t)\right) \geq \frac{c^{\prime }}{2}\bar{c}%
\displaystyle\prod\limits_{j\in {\mathcal{N}}_{v_{m}}}\left( \frac{\mu
_{j}(t)}{\tilde{y}_{j}^{\left( m-1\right) }}\right) ^{r_{j}}\geq \bar{c}%
_{1}\prod\limits_{j\in {\mathcal{N}}_{v_{m}}}\left( \frac{\mu _{j}(0)}{%
\tilde{y}_{j}^{\left( m-1\right) }}\right) ^{r_{j}},
\end{equation*}%
where $\bar{c}_{1}\doteq c^{\prime }\bar{c}e^{-R_{1}K}/2>0$. Since %
\eqref{def-xl} implies $\mu _{j}(0)\geq \tilde{y}_{j}^{(m-1)}$ for $j\in {%
\mathcal{X}}\setminus {\mathcal{X}}_{m-1}$ and $\tilde{y}_{j}^{(m-1)}\leq 1$
for all $j$ (and in particular $j\in {\mathcal{X}}_{m-1}$), we can further
simplify the last inequality to obtain 
\begin{equation}
\lambda _{v_{m}}\left( \mu (t)\right) \geq \bar{c}_{1}\displaystyle%
\prod\limits_{j\in {\mathcal{X}}_{m-1}\cap {\mathcal{N}}_{v_{m}}}\left( \mu
_{j}(t)\right) ^{r_{j}}\text{,}  \label{ratio}
\end{equation}%
where the product over an empty set is to be interpreted as $1$.

We claim, and show below, that for every $m=1,\ldots ,F$, there exists $%
b^{\left( m\right) }\in (0,1)$ such that for every $i\in \mathcal{X}_{m}$, %
\eqref{ineq-lln} holds with $b=b^{(m)}$ and $D=D\left( m\right) \doteq
\sum_{i=0}^{m-1}K^{i}<\infty $. Setting $m=F$, this then proves %
\eqref{ineq-lln} for all $i\in {\mathcal{X}}_{F}$ with $b=b^{(F)}$ and $%
D=D(F)$. This suffices to complete the proof because for $i\in {\mathcal{X}}%
\setminus {\mathcal{X}}_{F}$, $\mu _{i}(0)\geq y_{i}^{(F)}=y_{i}=1/d$ and so %
\eqref{1lbd} implies that (\ref{ineq-lln}) holds with $b\doteq \frac{1}{d}%
e^{-R_{1}}>0$ and $D=0$.

We now use an inductive argument to prove the claim. Define 
\begin{equation*}
\delta _{2}\doteq \min \left\{ \left\langle v,e_{j}\right\rangle :j\in 
\mathcal{X}\text{, }v\in \mathcal{V}\text{, s.t. }\left\langle
v,e_{j}\right\rangle >0\right\} >0.
\end{equation*}%
We first consider the case $m=1$. For every $i\in {\mathcal{X}}_{1}$, $\phi
_{i}(0)=\mu _{i}(0)<y_{i}^{(1)}=\phi _{i}(t_{1})$, which implies $\langle
v_{1},e_{i}\rangle >0$ since $\dot{\phi}_{i}(t)=U_{1}\langle
v_{1},e_{i}\rangle $ for a.e.\ $t\in (0,t_{1})$ due to the assumed
representation \eqref{CP} of $\phi $. Moreover, \eqref{ratio} and the fact
that $\mathcal{X}_{0}$ is the empty set, together imply $\lambda
_{v_{1}}(\mu (t))\geq \bar{c}_{1}$ for $t\in \lbrack 0,1]$. Substituting
this into \eqref{9.3.1}, one sees that for $i\in {\mathcal{X}}_{1}$, 
\begin{equation}
\dot{\mu}_{i}(t)\geq \langle v_{1},e_{i}\rangle \lambda _{v_{1}}(\mu
(t))-R_{1}\mu _{i}(t)>\delta _{2}\bar{c}_{1}-R_{1}\mu _{i}(t),\quad t\in
\lbrack 0,1].  \label{mu-insert}
\end{equation}%
By the comparison principle for ODEs, and the fact that $\mu _{i}(0)\geq 0$,
we see that there exists some $b^{\left( 1\right) }\in (0,1)$ such that for $%
t\in \lbrack 0,1]$, 
\begin{equation}
\mu _{i}\left( t\right) \geq \frac{\delta _{2}\bar{c}_{1}}{R_{1}}\left(
1-e^{-R_{1}t}\right) \geq b^{\left( 1\right) }t,  \label{mu-insert2}
\end{equation}%
Since $D(1)=1$, this shows that the claim holds for $m=1$.

Next, assume that for some $m_{0}\in \{1,\ldots ,F-1\}$, the claim holds for
all $m\in \{1,\ldots ,m_{0}\}$, and let $\bar{m}=m_{0}+1$. Fix $i\in 
\mathcal{X}_{\bar{m}}$. Then, since $\phi_i( 0) = \mu _{i}\left( 0\right)
<y_{i}^{\left( \bar{m}\right) }=\phi _{i}\left( t_{\bar{m}}\right) $, it is
clear from the representation (\ref{CP}) for $\phi $ that there exists $%
m^{\ast }\in \{1,\ldots ,\bar{m}\}$ such that $\langle v_{m^{\ast
}},e_{i}\rangle >0$. If $m^{\ast }=1$, this shows that \eqref{mu-insert},
and hence \eqref{mu-insert2}, holds. Since $D(\bar{m}) \geq 1$, the claim
holds for $m=\bar{m}$ with $b^{(\bar{m})} = b^{(1)}$. On the other hand, if $%
m^{\ast }\in \{2,\ldots ,\bar{m}\}$, then by \eqref{ratio} and the induction
hypothesis we have for $t\in \left[ 0,1\right] $,%
\begin{equation*}
\lambda _{v_{m^{\ast }}}\left( \mu \left( t\right) \right) \geq \bar{c}_{1} %
\displaystyle\prod\limits_{j\in \mathcal{X}_{m^{\ast }-1} \cap {\mathcal{N}}%
_{v_{m^{\ast}-1}}}\left( \mu _{j}\left( t\right) \right)^{r_{j}} \geq \bar{c}%
_{2}t^{KD(m^{\ast}-1)},
\end{equation*}
with $\bar{c}_{2}\doteq \bar{c}_{1}(b^{(m^{\ast }-1)})^{K}>0$, where we have
used the fact that $\sum_{j \in {\mathcal{N}}_{v_{m^{\ast}-1}}}r_j \leq K.$
Since $m\mapsto D(m)$ is increasing, this gives a lower bound of $\bar{c}%
_{2}t^{KD( \bar{m}-1)}$ on $\lambda _{v_{m^{\ast }}}\left( \mu \left(
t\right) \right) $. Substituting this into \eqref{9.3.1} we see that for\ $%
t\in \lbrack 0,1]$, 
\begin{eqnarray*}
\dot{\mu}_{i}\left( t\right) &\geq &\left\langle
v_{m^{\ast}},e_{i}\right\rangle \lambda _{v_{m^{\ast }}}\left( \mu \left(
t\right) \right) -R_{1}\mu _{i}\left( t\right) . \\
&\geq &\bar{c}_{3}t^{\bar{r}D(\bar{m}-1)}-R_{1}\mu _{i}\left( t\right) ,
\end{eqnarray*}%
where $\bar{c}_{3}\doteq \delta _{2}\bar{c}_{2}$, which we can assume
without loss of generality to lie in $(0,1)$. Note that the solution to the
ODE $\dot{u}\left( t\right) =\bar{c}_{3}t^{KD(\bar{m}-1)}-R_{1}u\left(
t\right) $ with $u(0)\geq 0$ satisfies $u(t)\geq b^{(\bar{m})}t^{KD(\bar{m}%
-1)+1}$ with $b^{(\bar{m})}\doteq e^{-R_{1}}\bar{c}_{3}/(KD(\bar{m}-1)+1)\in
(0,1)$. Applying the comparison principle for ODEs and noting that $D(\bar{m}%
)=KD(\bar{m}-1)+1$, it follows that $\mu _{i}(t)\geq b^{(\bar{m})}t^{D(\bar{m%
})}$, which proves the claim for $\bar{m}=m_{0}+1$. This completes the proof
of the proposition.
\end{proof}

\subsection{A discrete communication condition}

\label{subs-propverii}

We now show that under the slightly stronger assumption, Assumption \ref{g1}%
, on the transition rates of the interacting particle system, the empirical
measure jump Markov processes possesses a stronger controllability property
(Property \ref{prop-dcommunicate} below), which is used in the proof of the
locally uniform LDP in Section \ref{sec-locunif}. We first describe a
discrete version of the strong communication condition that was introduced
in Definition \ref{def1''}.

Given $u,v\in \mathbb{N}$, we denote $[u,v)=\left\{ u,u+1,...,v-1\right\} $.

\begin{definition}
\label{def1} For any $x,y\in \mathcal{S}_{n}$ and $T\in \mathbb{N}$, a 
\textbf{discrete} \textbf{strongly} \textbf{communicating path} of length $T$
from $x$ to $y$ with constants $c_{1}>0$, 
$p_{1}<\infty $, and $F\in \mathbb{N}$ is a set of\ points $\left\{ \phi
_{0},\phi _{1},...,\phi _{F}\right\} $ $\subset \mathcal{S}_{n}$ such that $%
\phi _{0}=x$, $\phi _{F}=y$ and the following properties 
are satisfied.

i). There exist $\left\{ v_{m}\right\} _{m=1}^{F}\subset \mathcal{V}$, and $%
0=t_{0}<t_{1}<\cdots <t_{F}=T$, such that 
\begin{equation*}
\phi _{s+1}-\phi _{s}=\frac{1}{n}v_{m}\text{, }s\in \{ t_{m-1}, t_{m-1}+1,
\ldots, t_{m}-1\}.
\end{equation*}

ii). For all $n$ sufficiently large and $m=1,...,F$,%
\begin{equation}
\lambda _{v_{m}}^{n}\left( \phi _{s}\right) \geq c_{1}\left( \displaystyle%
\prod\limits_{j\in \mathcal{N}_{v_{m}}}\left( \phi _{s}\right) _{j}\right)
^{p_{1}},\text{ \ } s\in \{ t_{m-1}, t_{m-1}+1, \ldots, t_{m}-1\},
\label{dlvm}
\end{equation}%
where for $v \in {\mathcal{V}}$, ${\mathcal{N}}_v$ is as defined in %
\eqref{def-nv}. 
\end{definition}

\begin{property}
\label{prop-dcommunicate} There exist constants $c,c_{1}>0$, $C^{\prime
},C_{1}^{\prime },p,p_{1}<\infty $ and $\bar{F}\in \mathbb{N}$ with the
following properties.

i). For any $x,y\in \mathcal{S}$, there exist $t\in (0,1]$, and a strongly
communicating path $\phi $ that connects $x$ to $y$ on $[0,t]$ with
constants $c,p,\bar{F},c_{1},p_{1}$ such that the scaling property %
\eqref{length} is satisfied.

ii). For any $n\in \mathbb{N}$ and any $x,y\in \mathcal{S}_{n}$, there exist 
$\tilde{F}$, and a discrete strongly communicating path $\phi _{n}$ of
length $T_{x,y}$ that connects $x$ to $y$ with constants $c_{1},p_{1},\tilde{%
F}$ such that the path lengths satisfy the scaling property%
\begin{equation}
T_{x,y}\leq C_{1}^{\prime }n||x-y||.  \label{disc-length}
\end{equation}
\end{property}

Clearly, Property \ref{prop-dcommunicate} is a strengthening of the
communication property stated in Property \ref{prop-communicate}.

\begin{proposition}
\label{ver2} If $\{\Gamma _{\mathbf{i}\mathbf{j}}(\cdot ),(\mathbf{i,}%
\mathbf{j})\in {\mathcal{J}}^{k},k=1,\ldots ,K\}$ satisfies Assumption \ref%
{ue}, Assumption \ref{intsys} and Assumption \ref{g1}, then the associated
jump rates $\{\lambda _{v}(\cdot ),v\in {\mathcal{V}}\}$ satisfy Property %
\ref{prop-dcommunicate} and hence, also Property \ref{prop-communicate}.
\end{proposition}

\begin{proof}
The first part of Property \ref{prop-dcommunicate}, namely the existence of
a strongly communicating path, follows immediately from the first assertion
of Lemma \ref{move}. For the second part, we construct a discrete strongly
communicating by discretizing the strongly communicating path constructed in
Lemma \ref{move}. Specifically, for $n\in \mathbb{N}$, given $x,y\in {%
\mathcal{S}}_{n}$, given the strongly communicating path $\phi ^{P,n}$ on $%
[0,t_{F}]$ for some $t_{F}\in \mathbb{N}$, which satisfies the additional
properties stated in the last assertion of Lemma \ref{move}, then it is easy
to verify that if $\phi _{m}\doteq \phi ^{P,n}(m)$ for $m\in \lbrack
0,t_{F}]\cap \mathbb{N}$, then $\{\phi _{m}\}$ is a discrete strongly
communicating path from $x$ to $y$ and the property \eqref{disc-length} can
be deduced from the corresponding property \eqref{length} for communicating
paths that was established in Lemma \ref{move}.
\end{proof}

\section{The Variational Representation Formula}

\label{sec-varrep}

\subsection{Variational Representation for a Poisson Random Measure}

\label{subs-varrep}

We first review the variational representation formula for a Poisson random
measure stated in \cite[Theorem 2.1]{BDM}. For a locally compact Polish
space $S$, let $\mathcal{B}\left( S\right) $ denote the Borel sigma algebra
and let $M_{F}\left( S\right) $ denote the space of all measures $\nu $ on $%
\left( S,\mathcal{B}\left( S\right) \right) $ satisfying $\nu (K)<\infty $
for every compact $K\subset S$. Letting $C_{c}\left( S\right) $ denote the
space of continuous functions with compact support, we equip $M_{F}\left(
S\right) $ with the weakest topology such that for every $f\in C_{c}\left(
S\right) $, the function $\nu \mapsto \int_{S}fd\nu ,\nu \in M_{F}\left(
S\right)$, is continuous. Let $\mathcal{Y}$ be a locally compact Polish
space, $\mathcal{Y}_{T}=[0,T]\times \mathcal{Y}$, both equipped with the
usual Euclidean topology, and let $\mathcal{M}=M_{F}(\mathcal{Y}_{T})$. For
some fixed measure $\nu \in M_{F}\left( \mathcal{Y}\right) $, let $\nu
_{T}=m_{T}\otimes \nu $, where $m_{T}$ is Lebesgue measure on $\left[ 0,T%
\right] $. [In our use below we take $\mathcal{Y}=[0,\infty )$ and $\nu $ to
be Lebesgue measure.] For $\theta \in \lbrack 0,\infty )$, let $\mathbb{P}%
_{\theta }$ denote the unique probability measure on $\left( \mathcal{M},%
\mathcal{B}\left( \mathcal{M}\right) \right) $ under which the canonical map 
$N:\mathcal{M}\rightarrow \mathcal{M}$, $N\left( \omega \right) =\omega $,
is a Poisson random measure with intensity measure $\theta \nu _{T}$. Let $%
\mathbb{E}_{\theta }$ denote expectation with respect to $\mathbb{P}_{\theta
}$. For notational convenience, we omit the dependence of $\mathbb{P}%
_{\theta }$ and $\mathbb{E}_{\theta }$ on the fixed measure $\nu _{T}$.

We define a \textit{controlled} Poisson random measure as follows. Let $%
\mathcal{W}=\mathcal{Y}\times \lbrack 0,\infty )$ and $\mathcal{W}%
_{T}=[0,T]\times \mathcal{W}=\mathcal{Y}_{T}\times \lbrack 0,\infty )$, both
equipped with the Euclidean product topology. Let $\mathcal{\bar{M}}=M_{F}(%
\mathcal{W}_{T})$ and let $\mathbb{\bar{P}}$ be the unique probability
measure on $\left( \mathcal{\bar{M}},\mathcal{B}\left( \mathcal{\bar{M}}%
\right) \right) $ under which the canonical map $\bar{N}:\mathcal{\bar{M}}%
\rightarrow \mathcal{\bar{M}}$, $\bar{N}\left( \omega \right) =\omega ,$ is
a Poisson random measure with intensity measure $\overline{\nu }_{T}=\nu
_{T}\otimes m$, where $m$ is Lebesgue measure on $[0,\infty )$. Let $\mathbb{%
\bar{E}}$ denote expectation with respect to $\mathbb{\bar{P}}$. Also,
define 
\begin{equation*}
\mathcal{G}_{t}\doteq \sigma \left\{ \bar{N}\left( (0,s]\times A\right)
:0\leq s\leq t,A\in \mathcal{B}\left( \mathcal{W}\right) \right\} ,
\end{equation*}%
and let $\mathcal{F}_{t}$ denote its completion under $\mathbb{\bar{P}}$. We
equip $\left( \mathcal{\bar{M}},\mathcal{B}\left( \mathcal{\bar{M}}\right)
\right) $ with the filtration $\left\{ \mathcal{F}_{t}\right\} _{0\leq t\leq
T}$ and denote by $\mathcal{\bar{P}}$ the corresponding predictable $\sigma $%
-field on $\left[ 0,T\right] \times \mathcal{\bar{M}}$.

\begin{definition}
\label{a-}Let $\mathcal{\bar{A}}$ be the class of $\left( \mathcal{\bar{P}}%
\otimes \mathcal{B}\left( \mathcal{Y}\right) \right) \backslash \mathcal{B}%
[0,\infty )$ measurable maps $\varphi :\left[ 0,T\right] \times \mathcal{%
\bar{M}\times Y}\rightarrow \lbrack 0,\infty )$.
\end{definition}

The role of $\varphi $ is to control the intensity of jumps at $(s,\omega
,y) $ by thinning in the additional $r$-variable in \eqref{nphi} below. For $%
\varphi \in \mathcal{\bar{A}}$, define $N^{\varphi }:\mathcal{\bar{M}}%
\rightarrow \mathcal{M}$ by 
\begin{equation}
N_{\omega }^{\varphi }\left( (0,t]\times U\right) \doteq \int_{(0,t]\times
U}\int_{0}^{\infty }\mathbb{I}_{\left[ 0,\varphi \left( s,\omega ,y\right) %
\right] }\left( r\right) \bar{N}_{\omega }\left( dsdydr\right) \text{, \ }%
t\in \left[ 0,T\right] ,U\in \mathcal{B}\left( \mathcal{Y}\right) ,\text{ }%
\omega \in \mathcal{\bar{M}}.  \label{nphi}
\end{equation}%
In what follows, we often suppress the dependence of $\varphi \left(
t,\omega ,y\right) $, $\bar{N}_{\omega }$ and $N_{\omega }^{\varphi }$ on $%
\omega $. Under $\mathbb{\bar{P}}$, $N^{\varphi }$ is a controlled random
measure on $\mathcal{Y}_{T}$ with $\varphi \left( s,y\right) $ determining
the intensity for points at location $y$ and time $s$. With some abuse of
notation, for $\theta \in \lbrack 0,\infty )$ we will let $N^{\theta }$ be
defined as in (\ref{nphi}) with $\varphi \left( s,y\right) \equiv \theta $.
Note that the law (on $\mathcal{M}$) of $N^{\theta }$ under $\mathbb{\bar{P}}
$ coincides with the law of $N$ under $\mathbb{P}_{\theta }$.

Recall $\ell \left( \cdot \right) $ as defined in (\ref{l}). For $\varphi
\in \mathcal{\bar{A}}$ define the random variable $L_{T}\left( \varphi
\right) $ by%
\begin{align}
L_{T}\left( \varphi \right) \left( \omega \right) & \doteq \int_{\mathcal{Y}%
_{T}}\ell \left( \varphi \left( t,\omega ,y\right) \right) \nu _{T}\left(
dtdy\right)  \notag \\
& =\int_{0}^{T}\left( \int_{[0,\infty )}\ell \left( \varphi \left( t,\omega
,y\right) \right) \nu \left( dy\right) \right) dt,\text{ }\omega \in 
\mathcal{\bar{M}}.  \label{L_T}
\end{align}

\begin{definition}
\label{ab-}Define $\mathcal{\bar{A}}_{b}$ to be the class of $\left( 
\mathcal{\bar{P}}\otimes \mathcal{B}\left( \mathcal{Y}\right) \right)
\backslash \mathcal{B}[0,\infty )$ measurable maps $\varphi $ such that for
some $B<\infty $, $\varphi \left( t,\omega ,y\right) \leq B$ for all $\left(
t,\omega ,y\right) \in \left[ 0,T\right] \times \mathcal{\bar{M}\times Y}$.
\end{definition}

In later sections we will set $T=1$, and hence the dependence of $\mathcal{%
\bar{A}}$ and $\mathcal{\bar{A}}_{b}$ on $T$ can be omitted. Let $%
M_{b}\left( \mathcal{M}\right) $ denote the space of bounded Borel
measurable functions on $\mathcal{M}$. We then have the following
representation formula for Poisson random measures.

\begin{theorem}
\label{3.1} Let $F\in M_{b}\left( \mathcal{M}\right) $. Then for any $\theta
>0$,%
\begin{equation}
-\log \mathbb{E}_{\theta }\left[ \exp \left( -F\left( N\right) \right) %
\right] =\inf_{\varphi \in \mathcal{\bar{A}}_{b}}\mathbb{\bar{E}}\left[
\theta L_{T}\left( \varphi \right) +F(N^{\theta \varphi })\right] .
\label{old rp}
\end{equation}
\end{theorem}

\begin{proof}
For $F\in M_{b}\left( \mathcal{M}\right) $ and $\theta >0$, it follows from
Theorem 2.1 of \cite{BDM} that 
\begin{eqnarray*}
-\log \mathbb{E}_{\theta }\left[ \exp \left( -F\left( N\right) \right) %
\right] &=&-\log \mathbb{\bar{E}}\left[ \exp \left( -F(N^{\theta })\right) %
\right] \\
&=&\inf_{\varphi \in \mathcal{\bar{A}}}\mathbb{\bar{E}}\left[ \theta
L_{T}\left( \varphi \right) +F(N^{\theta \varphi })\right] .
\end{eqnarray*}%
Moreover, Theorem 2.4 of \cite{BCD} states that the above infimization can
in fact be taken over the smaller class of controls, where for each control $%
\varphi $ there is $B\in (0,\infty )$ and compact $K\subset \mathcal{Y}$
such that $1/B\leq \varphi \left( t,\omega ,y\right) \leq B$ for all $\left(
t,\omega ,y\right) \in \left[ 0,T\right] \times \mathcal{\bar{M}}\times K$
and $\varphi \left( t,\omega ,x\right) =1$ for all $\left( t,\omega
,y\right) \in \left[ 0,T\right] \times \mathcal{\bar{M}}\times K^{c}$. Since 
$\mathcal{\bar{A}}_{b}\subset \mathcal{\bar{A}}$ contains this class of
controls, we obtain (\ref{old rp}).
\end{proof}

\subsection{Variational representation for the empirical measure process}

In this section we derive a variational representation formula for the
empirical measure process $\mu ^{n}$. We represent $\mu ^{n}$ as a solution
to a stochastic differential equation that is driven by finitely many iid
Poisson random measures, and use thinning functions to obtain the desired
jump rates. We then derive a variational representation formula for $\mu
^{n} $, by viewing it as the image of a measurable mapping that acts on a
collection of rescaled Poisson random measures.

Take $\nu =m$, so that $\nu _{T}=m_{T}\otimes m$. For $n\in \mathbb{N}$, let 
$\{N_{v}^{n},v\in \mathcal{V}\}$ be a collection of iid Poisson random
measures (on $\mathcal{Y}_{T}$) with intensity measure $n\nu _{T}$. Then we
have the following SDE representation for the empirical measure process: for 
$t\in \left[ 0,T\right] $,%
\begin{equation}
\mu ^{n}(t)=\mu ^{n}(0)+\sum_{v\in \mathcal{V}}v\int_{[0,t]}\int_{\mathcal{Y}%
}\mathbb{I}_{\left[ 0,\lambda _{v}^{n}\left( \mu ^{n}(s-)\right) \right] }(y)%
\frac{1}{n}N_{v}^{n}(dsdy).  \label{sde1}
\end{equation}%
The existence of a solution to (\ref{sde1}) is justified by the following
argument.

We set%
\begin{equation}
\bar{R}\doteq \sup_{v\in \mathcal{V},x\in \mathcal{S}_{n}\text{,}n\in 
\mathbb{N}}\lambda _{v}^{n}\left( x\right) <\infty ,  \label{m'}
\end{equation}%
where the finiteness of $\bar{R}$ follows because Property \ref{prop-lambda}
is satisfied and (\ref{m}) holds. Let $\mathcal{M}_{atom}$ denote the set of
all $\mathrm{m}=\left\{ \mathrm{m}_{v},v\in \mathcal{V}\right\} $, where for
each $v\in \mathcal{V}$, $\mathrm{m}_{v}$ is an atomic measure on $\mathcal{Y%
}_{T}$, with the property that $\mathrm{m}_{v}(\left\{ t\right\} \times
\lbrack 0,\bar{R}])>0$ for only finitely many $t$. Recall from Section \ref%
{subs-model} that ${\mathcal{S}}$ denotes the unit $(d-1)$-dimensional
simplex and ${\mathcal{S}}_{n}$ is the corresponding sublattice. Define $%
\tilde{\Delta}^{d-1}=\{x\in \mathbb{R}^{d}:\sum_{i=1}^{d}x_{i}=1\}$ to be
the $(d-1)$-dimensional hyperplane that contains ${\mathcal{S}}.$ Denote by $%
D\left( \left[ 0,T\right] :\mathcal{S}\right) $ and $D\left( \left[ 0,T%
\right] :\tilde{\Delta}^{d-1}\right) $ respectively the space of c\`{a}dl%
\`{a}g functions on $[0,T]$ that take values in $\mathcal{S}$ and $\tilde{%
\Delta}^{d-1}$. Let $\lambda ^{n}=\{\lambda _{v}^{n},v\in \mathcal{V}\}$. We
also extend the definition of $\lambda ^{n}$ to $\tilde{\Delta}^{d-1}$ by
define it to be zero in $\tilde{\Delta}^{d-1}\setminus \mathcal{S}_{n}$.
Define $h_{n}:\mathcal{M}_{atom}\times \mathcal{S}\times \left( \mathbb{[}%
0,\infty )^{\tilde{\Delta}^{d-1}}\right) ^{\otimes \left\vert \mathcal{V}%
\right\vert }\rightarrow D\left( \left[ 0,T\right] :\tilde{\Delta}%
^{d-1}\right) $ as the mapping that takes $\left( \mathrm{m},\rho ,\lambda
^{n}\right) \in \mathcal{M}_{atom}\times \mathcal{S}\times \left( \lbrack
0,\infty )^{\tilde{\Delta}^{d-1}}\right) ^{\otimes \left\vert \mathcal{V}%
\right\vert }$ to the process $\eta \in D\left( \left[ 0,T\right] :\tilde{%
\Delta}^{d-1}\right) $ defined by 
\begin{equation}
\eta \left( t\right) \doteq \rho +\sum_{v\in \mathcal{V}}v\int_{[0,t]}\int_{%
\mathcal{Y}}\mathbb{I}_{[0,\lambda _{v}^{n}\left( \eta (s-)\right) ]}(y)%
\mathrm{m}_{v}(dsdy),\quad t\in \lbrack 0,T].  \label{hn}
\end{equation}%
In particular, when $\mathrm{m}=\frac{1}{n}N^{n}$ the process $\eta $ lie in 
${\mathcal{S}}_{n}.$

The existence of a solution $\eta (\cdot )$ to (\ref{hn}) is easily verified
by the following recursive construction. Set $t_{0}=0$, and define $\eta
^{0}\left( t\right) \doteq \rho $ for $t\geq 0$. Assume as part of the
recursive construction that for some $k\in \mathbb{N}_{0}$, a solution $\eta
^{k}\left( \cdot \right) $ to (\ref{hn}) has been constructed on the
interval $[0,t_{k}]$, and that $\eta ^{k}\left( t\right) =\eta ^{k}\left(
t_{k}\right) $ for $t\geq t_{k}$. For any $t\in \left[ t_{k},T\right] $ and $%
v\in \mathcal{V}$, let%
\begin{equation*}
A_{v}\left( t\right) \doteq \left\{ \left( s,y\right) :s\in \lbrack
t_{k},t],y\in \left[ 0,\lambda _{v}^{n}\left( \eta ^{k}(s)\right) \right]
\right\} ,
\end{equation*}%
and 
\begin{equation*}
t_{k+1}\doteq \inf \left\{ t>t_{k}\text{ such that for some }v\in \mathcal{V}%
,\mathrm{m}_{v}\left( A_{v}\left( t\right) \right) >0\right\} \wedge T,
\end{equation*}%
where for $a,b\in \mathbb{R}$, $a\wedge b$ denotes the minimum of $a$ and $b$%
. We then define $\eta ^{k+1}:[0,T]\rightarrow \mathbb{R}^{d}$ by setting $%
\eta ^{k+1}\left( t\right) \doteq \eta ^{k}\left( t\right) $ for $t\in
\lbrack 0,t_{k+1})$,%
\begin{equation*}
\eta ^{k+1}\left( t_{k+1}\right) \doteq \eta ^{k}\left( t_{k}\right)
+\sum_{v\in \mathcal{V}}v\int_{[t_{k},t_{k+1}]}\int_{\mathcal{Y}}\mathbb{I}%
_{[0,\lambda _{v}^{n}\left( \eta ^{k}(s-)\right) ]}(y)\mathrm{m}_{v}(dsdy),
\end{equation*}%
and $\eta ^{k+1}\left( t\right) \doteq \eta ^{k+1}\left( t_{k+1}\right) $
for $t\in \lbrack t_{k+1},T]$.

Since $\mathrm{m}_{v}$ has finitely many atoms on $\left[ 0,T\right] \times
\lbrack 0,\bar{R}]$, the construction will produce a function defined on all
of $[0,T]$ in $M<\infty $ steps, at which time we set $\eta \left( t\right)
=\eta ^{M}\left( t\right) $. Since $N\in \mathcal{M}_{atom}$ for $\mathbb{P}%
_{n}-$a.e. $\omega \in \mathcal{M}$, we can write 
\begin{equation}
\mu ^{n}\left( t,\omega \right) =h_n\left( \frac{1}{n}N^{n},\mu
^{n}(0,\omega ),\lambda ^{n}\right) \left( t\right) .  \label{mun}
\end{equation}

We now describe two classes of controls that will be used below. Recall that 
$\overline{\nu }_{T}=\nu _{T}\otimes m$. Let $\left\{ \bar{N}_{v}^{n},v\in 
\mathcal{V}\right\} $ be a collection of iid Poisson random measures on $%
\mathcal{W}_{T}$ with intensity measure $n\overline{\nu }_{T}$. We will
apply the representation that is appropriate for these $\left\vert \mathcal{V%
}\right\vert $ independent Poisson random measures. The underlying
probability space is now the product space $\left( \mathcal{\bar{M}},%
\mathcal{B}\left( \mathcal{\bar{M}}\right) \right) ^{\otimes \left\vert 
\mathcal{V}\right\vert }$ (with an abuse of notation we retain $\mathbb{\bar{%
P}}$ to denote the probability measure on this space). Let \newline
\begin{equation*}
\mathcal{G}_{t}^{\left\vert \mathcal{V}\right\vert }\doteq \sigma \left\{ 
\bar{N}_{v}^{n}\left( (0,s]\times A\right) :0\leq s\leq t,A\in \mathcal{B}%
\left( \mathcal{Y}\right) ,v\in \mathcal{V}\right\}
\end{equation*}%
and let $\mathcal{F}_{t}^{\left\vert \mathcal{V}\right\vert }$ denote its
completion under $\mathbb{\bar{P}}$. Denote by $\mathcal{\bar{P}}%
^{\left\vert \mathcal{V}\right\vert }$ the predictable $\sigma $-field on $%
\left[ 0,T\right] \times \mathcal{\bar{M}}^{\otimes \left\vert \mathcal{V}%
\right\vert }\mathcal{\ }$with the filtration $\{\mathcal{F}_{t}^{\left\vert 
\mathcal{V}\right\vert }:0\leq t\leq T\}$ on $\left( \mathcal{\bar{M}},%
\mathcal{B}\left( \mathcal{\bar{M}}\right) \right) ^{\otimes \left\vert 
\mathcal{V}\right\vert }$, and let $\mathcal{\bar{A}}^{\otimes \left\vert 
\mathcal{V}\right\vert }$ and $\mathcal{\bar{A}}_{b}^{\otimes \left\vert 
\mathcal{V}\right\vert }$ be defined analogously as was done for the case of
a single Poisson random measure. Given $\varphi \in \mathcal{\bar{A}}%
_{b}^{\otimes \left\vert \mathcal{V}\right\vert }$ and $\mu ^{n}(0)\in {%
\mathcal{S}},$ we define the controlled jump Markov process $\hat{\mu}%
^{n}\in D\left( [0,T]:\mathcal{S}\right) $ to be the solution to the
following SDE: for $t\in \left[ 0,T\right] $,%
\begin{equation}
\hat{\mu}^{n}\left( t\right) =\mu ^{n}(0)+\sum_{v\in \mathcal{V}%
}v\int_{[0,t]}\int_{\mathcal{Y}}\mathbb{I}_{[0,\lambda _{v}^{n}\left( \hat{%
\mu}^{n}(s-)\right) ]}(y)\int_{[0,\infty )}\mathbb{I}_{[0,\varphi
_{v}(s-,y)]}(r)\frac{1}{n}\bar{N}_{v}^{n}(dsdydr).  \label{sde2}
\end{equation}%
As described previously, $\varphi _{v}(s,y)$ will control the jump rate as a
function of $(s,\omega ,y)$. In particular, the overall jump rate for a jump
of type $v$ is the product $\lambda _{v}^{n}\left( \hat{\mu}^{n}(s)\right)
\varphi _{v}(s,y)$, so that $\varphi _{v}(s,y)$ perturbs the jump rate away
from that of the original model, and the cumulative impact of the
perturbation is found by integrating $y$ over $[0,\lambda _{v}^{n}\left( 
\hat{\mu}^{n}(s)\right) ]$. For $v\in \mathcal{V}$, let $N_{v}^{n\varphi }$
be defined as in (\ref{nphi}), with $\varphi $ replaced by $\varphi _{v}$
and $\bar{N}$ replaced by $\bar{N}_{v}^{n}$, and let $N^{n\varphi
}=\{N_{v}^{n\varphi },v\in \mathcal{V}\}$. For fixed $\varphi \in \mathcal{%
\bar{A}}_{b}^{\otimes \left\vert \mathcal{V}\right\vert }$, $N^{n\varphi
}\in \mathcal{M}_{atom}$ a.s. From the definition of $h_{n}(\cdot )$ and $%
N^{n\varphi }$, it is clear that (\ref{sde2}) is equivalent to the relation 
\begin{equation*}
\hat{\mu}^{n}=h_{n}\left( \frac{1}{n}N^{n\varphi },\mu ^{n}(0),\lambda
^{n}\right) .
\end{equation*}

Applying Theorem \ref{3.1} and (\ref{mun}) with $F=G\circ h_{n}$, we obtain
the following representation formula for $\mu ^{n}$.

\begin{lemma}
\label{rep} For $G\in M_{b}\left( D\left( \left[ 0,T\right] :\mathcal{S}%
\right) \right) $,%
\begin{equation*}
-\frac{1}{n}\log \mathbb{E}\left[ \exp (-nG(\mu ^{n}))\right] =\inf_{\varphi
\in \mathcal{\bar{A}}_{b}^{\otimes \left\vert \mathcal{V}\right\vert }}%
\mathbb{\bar{E}}\left[ \sum_{v\in \mathcal{V}}L_{T}\left( \varphi
_{v}\right) +G(\hat{\mu}^{n}):\hat{\mu}^{n}=h_{n}\left( \frac{1}{n}%
N^{n\varphi },\mu ^{n}(0),\lambda ^{n}\right) \right] .
\end{equation*}
\end{lemma}

We now derive a simpler form of the variational representation formula than
the one given in Lemma \ref{rep}. The starting point for Lemma \ref{rep} is
the representation given in \cite[Theorem 2.1]{BDM}, which is general enough
to cover situations where different points in $y\in \mathcal{Y}$ correspond
to different \textquotedblleft types\textquotedblright\ of jumps. For our
purposes this is in fact more general than we need, since all points in $%
\mathcal{Y}$ correspond to exactly the same type of jump, and all that is
needed from the space $\mathcal{Y}$ is that it be big enough that arbitrary
jump rates [such as $\lambda _{v}^{n}\left( \hat{\mu}^{n}(s)\right) $] can
be obtained by thinning. For example, we could have used $\mathcal{Y}=[0,%
\bar{R}]$ rather than $[0,\infty )$. The $y$'s in an interval such as $%
[0,\lambda _{v}^{n}\left( \hat{\mu}^{n}(s)\right) ]$ all play the same role,
which is to indicate that a jump of type $v$ should occur. Hence one
expects, and we will verify using Jensen's inequality, that one can
reformulate the representation in terms of controls with no explicit $y$%
-dependence. Thus we will replace the $t$, $y$ and $v$ dependent controls $%
\mathcal{\bar{A}}_{b}^{\otimes \left\vert \mathcal{V}\right\vert }$ by
controls $\mathcal{A}_{b}^{\otimes \left\vert \mathcal{V}\right\vert }$ that
only have $t$ and $v$ dependence, and rewrite the running cost as a function
of the new controlled jump rates.

\begin{definition}
\label{ab}Define $\mathcal{A}_{b}^{\otimes \left\vert \mathcal{V}\right\vert
}$ to be the class of $\mathcal{\bar{P}}\backslash \mathcal{B}([0,\infty
)^{\left\vert \mathcal{V}\right\vert })$ measurable maps $\varphi :$ $\left[
0,T\right] \times \mathcal{\bar{M}}^{\otimes \left\vert \mathcal{V}%
\right\vert }\rightarrow \lbrack 0,\infty )^{\left\vert \mathcal{V}%
\right\vert }$ such that $\sup_{t\in \left[ 0,T\right] ,\omega \in \mathcal{%
\bar{M}}^{\otimes \left\vert \mathcal{V}\right\vert },v}\varphi
_{v}(t,\omega )\leq B$ for some $B<\infty $.
\end{definition}

Define $\Lambda ^{n}:\mathcal{A}_{b}^{\otimes \left\vert \mathcal{V}%
\right\vert }\times \mathcal{S}\rightarrow D\left( \left[ 0,T\right] :\tilde{%
\Delta}^{d-1}\right) $ by%
\begin{equation}
\Lambda ^{n}\left( \bar{\alpha},\rho \right) \left( t\right) =\rho
+\sum_{v\in \mathcal{V}}v\int_{[0,t]}\int_{\mathcal{Y}}\mathbb{I}_{\left[ 0,%
\bar{\alpha}_{v}(s-)\right] }(y)\frac{1}{n}N_{v}^{n}(dsdy).  \label{lam}
\end{equation}%
$\Lambda ^{n}\left( \cdot ,\rho \right) $ is well-defined for $\bar{\alpha}%
\in \mathcal{A}_{b}^{\otimes \left\vert \mathcal{V}\right\vert }$.

We are now in a position to state the main variational representation
formula, the proof of which is deferred to the Appendix. This representation
appears to be the most appropriate one for finite state Markov chains, and
expresses the variational functional as the sum of the expected cost for
perturbing the jump rates, plus the expected value of the test function
evaluated at the process whose dynamics follow the perturbed rates.

\begin{theorem}
\label{3.2}Let $F\in M_{b}\left( D\left( \left[ 0,T\right] :\tilde{\Delta}%
^{d-1}\right) \right) $. Then%
\begin{align*}
& -\frac{1}{n}\log \mathbb{E}\left[ \exp (-nF(\mu ^{n}))\right] \\
& \quad =\inf_{\bar{\alpha}\in \mathcal{A}_{b}^{\otimes \left\vert \mathcal{V%
}\right\vert }}\mathbb{\bar{E}}\left[ \sum_{v\in \mathcal{V}%
}\int_{0}^{T}\lambda _{v}^{n}\left( \bar{\mu}^{n}(t)\right) \ell \left( 
\frac{\bar{\alpha}_{v}(t)}{\lambda _{v}^{n}\left( \bar{\mu}^{n}(t)\right) }%
\right) dt+F(\bar{\mu}^{n}):\bar{\mu}^{n}=\Lambda ^{n}\left( \bar{\alpha}%
,\mu ^{n}\left( 0\right) \right) \right] .
\end{align*}
\end{theorem}

\begin{remark}
Since the integrand in the right hand side of the equation above is singular
when $\bar{\mu}^{n}(t)\notin \mathcal{S}$ for some $t\in \left[ 0,T\right] $%
, it is equivalent to infimize over a smaller class of control, namely, $%
\bar{\alpha}\in \mathcal{A}_{b}^{\otimes \left\vert \mathcal{V}\right\vert }$
such that $\bar{\alpha}_{v}\left( t\right) =0$ when $\bar{\mu}^{n}(t)=x\in
\partial \mathcal{S}$ and $x+\frac{1}{n}v$ is taken outside $\mathcal{S}$.
By (\ref{lam}) the controlled process $\bar{\mu}^{n}$ will then lie in $%
\mathcal{S}$. Therefore it suffices to prove Theorem \ref{3.2} for $F\in
M_{b}\left( D\left( \left[ 0,T\right] :\mathcal{S}\right) \right) $.
\end{remark}

\subsection{The Law of Large Numbers limit}

\label{subs-pflln}

We next prove the law of large numbers limit stated in Theorem \ref{2.1}.
First recall the law of large numbers result for scaled Poisson random
measures: for any $A\in \mathcal{B}\left( \left[ 0,T\right] \times \lbrack
0,\infty )\right) $ such that $m_{T}\otimes m\left( A\right) <\infty $, $%
\frac{1}{n}N_{v}^{n}\left( A\right) \rightarrow m_{T}\otimes m\left(
A\right) $ in probability, for any $v\in \mathcal{V}$. This implies that for
any $f\in C_{c}\left( \left[ 0,T\right] \times \lbrack 0,\infty )\right) $
we have $\int_{[0,T]\times \lbrack 0,\infty )}f\left( s,y\right) \frac{1}{n}%
N_{v}^{n}\left( dsdy\right) $ $\rightarrow \int_{\lbrack 0,T]\times \lbrack
0,\infty )}f\left( s,y\right) dsdy$. Rewrite (\ref{sde1}) as%
\begin{equation*}
\mu ^{n}(t)=\rho _{0}+\sum_{v\in \mathcal{V}}v\int_{[0,t]}\lambda
_{v}^{n}\left( \mu ^{n}(s-)\right) ds+M^{n}(t),
\end{equation*}%
where $M^{n}(t)\doteq \sum_{v\in \mathcal{V}}\int_{[0,t]\times \lbrack
0,\infty )}\mathbb{I}_{\left[ 0,\lambda _{v}^{n}\left( \mu ^{n}(s-)\right) %
\right] }(y)(\frac{1}{n}N_{v}^{n}(dsdy)-dsdy)$ is an $\left\{ \mathcal{F}%
_{t}^{n}\right\} $-martingale. For any $\varepsilon >0$, with $\bar{R}$
defined as in (\ref{m'}), Doob's maximal inequality gives 
\begin{eqnarray*}
\mathbb{P}\left( \sup_{t\in \left[ 0,T\right] }\left\Vert
M^{n}(t)\right\Vert >\varepsilon \right) &\leq &\frac{1}{\varepsilon ^{2}}%
\mathbb{E}\left[ \left\Vert M^{n}(T)\right\Vert \right] ^{2} \\
&=&\frac{1}{\varepsilon ^{2}}\mathbb{E}\left[ \left\Vert \sum_{v\in \mathcal{%
V}}\int_{0}^{T}\int_{0}^{\infty }\mathbb{I}_{\left[ 0,\lambda _{v}^{n}\left(
\mu ^{n}(s-)\right) \right] }(y)\frac{1}{n}(N_{v}^{n}(dsdy)-ndsdy)\right%
\Vert ^{2}\right] \\
&\leq &\frac{\left\vert \mathcal{V}\right\vert }{n\varepsilon ^{2}}\mathbb{E}%
\left[ \int_{0}^{T}\int_{0}^{\bar{R}}dsdx\right] ,
\end{eqnarray*}%
which tends to zero as $n\rightarrow \infty $. Let $\mu $ the unique
solution $\mu $ of 
\begin{equation*}
\mu (t)=\rho _{0}+\sum_{v\in \mathcal{V}}v\int_{[0,t]}\lambda _{v}\left( \mu
(s-)\right) ds,
\end{equation*}%
which is the integral version of (\ref{1}). Combining $\mathbb{P}(\sup_{t\in %
\left[ 0,T\right] }\left\Vert M^{n}(t)\right\Vert >\varepsilon )\rightarrow
0 $ with the fact that $\lambda _{v}^{n}$ converges uniformly to $\lambda
_{v}$ and that $\lambda _{v}$ is Lipschitz continuous (by Property \ref%
{prop-lambda}), it follows from Gronwall's inequality that $\mu
^{n}\rightarrow \mu $ in probability (uniformly on $t\in \lbrack 0,T]$).
This completes the proof.

\bigskip

\section{Proof of the LDP Upper Bound}

\label{sec-pfupper}

A large deviation upper bound for a general class of sequences of Markov
processes was obtained in \cite{DEW}. We will apply the result of \cite{DEW}
to establish a large deviation upper bound for the sequence $\{\mu
^{n}(\cdot )\}_{n\in \mathbb{N}}$, in which for each $n\in \mathbb{N}$, $\mu
^{n}(\cdot )$ is a jump Markov process on ${\mathcal{S}}_{n}$ with generator 
${\mathcal{L}}_{n}$ in \eqref{Lnl} such that the associated sequence of
rates $\{\lambda _{v}^{n}(\cdot ),v\in {\mathcal{V}}\},n\in \mathbb{N},$
satisfy Property \ref{prop-lambda} for suitable Lipschitz continuous
functions $\{\lambda _{v}(\cdot ),v\in {\mathcal{V}}\}$. Theorem 1.1 of \cite%
{DEW} applies to Markov processes whose infinitesimal generator uses the
limit jump rates: 
\begin{equation}
\mathcal{L}_{n}^{0}\left( f\right) \left( x\right) =n\sum_{v\in \mathcal{V}%
}\lambda _{v}\left( x\right) \left[ f\left( x+\frac{1}{n}v\right) -f\left(
x\right) \right] .  \label{Ln0}
\end{equation}%
However, as discussed below the uniform convergence of $\lambda _{v}^{n}(x)$
to $\lambda _{v}(x)$ implies that the large deviation properties the
sequence of Markov processes with generators $\mathcal{L}_{n}^{0}$ and those
generators $\eqref{Lnl}$ coincide.

To state the result from \cite{DEW}, for $x, \theta \in \mathbb{R}^{d}$,
define 
\begin{equation}
H\left( x, \theta \right) \doteq \sum_{v\in \mathcal{V}}\lambda _{v}\left(
x\right) \left( \exp \left\langle \theta ,v\right\rangle -1\right) .
\label{eqn:defofH}
\end{equation}%
Note that $H$ is continuous. Let $L^{0}$ be its Legendre-Fenchel transform
defined by 
\begin{equation}
L^{0}\left( x,\beta \right) \doteq \sup_{\theta \in \mathbb{R}^{d}}\left[
\left\langle \theta,\beta \right\rangle -H\left( x,\theta \right) \right] .
\label{rate fn L0}
\end{equation}%
Also, for $t\in \left[ 0,1\right] $ define $I_{t}^{0}$ as in (\ref{I}), but
with $L$ replaced by $L^{0}$.

\begin{proposition}
\label{1.1}For any compact set ${\mathcal{K}}\subset \mathcal{S}$ and $%
M<\infty ,$ the set%
\begin{equation*}
\left\{ \gamma :I^{0}\left( \gamma \right) \leq M,\gamma \left( 0\right) \in 
{\mathcal{K}}\right\}
\end{equation*}%
is compact. Assume the family of jump rates $\left\{
\lambda_{v}\left(\cdot\right), v\in \mathcal{V}\right\} $ satisfies Property %
\ref{prop-lambda}. Also, assume that the initial conditions $\left\{ \mu
^{n}\left( 0\right) \right\} _{n\in \mathbb{N}}$ are deterministic, and $\mu
^{n}\left( 0\right) \rightarrow \mu _{0}\in {\mathcal{P}}\left( \mathcal{X}%
\right) $ as $n$ tends to infinity. Let $\left\{ Y^{n}\right\} _{n\in 
\mathbb{N}}$ be a sequence of Markov processes with generator $\mathcal{L}%
_{n}^{0}$, and $Y^{n}=\mu ^{n}\left( 0\right) $. Then $\left\{ Y^{n}\right\} 
$ satisfies the large deviation upper bound with rate function $I^{0}$.
\end{proposition}

\begin{proof}
This result follows from Theorem 1.1 of \cite{DEW} with (in the notation of 
\cite{DEW}) $\varepsilon =1/n$, $a\left( \cdot \right) =b\left( \cdot
\right) =0$, and $\mu _{x}\left( \cdot \right) =\mathbb{I}_{\{x\in \mathcal{S%
}\}}\sum_{v\in \mathcal{V}}\lambda _{v}\left( x\right) \delta _{v}\left(
\cdot \right) $.
\end{proof}

We have introduced the function $L^{0}$ in (\ref{rate fn L0}) and the
``local rate function'' $L$ in (\ref{L}), defined respectively in terms of a
Legendre transform and the Poisson local rate function $\ell$. We now show
that these functions are equal (see also Lemma 3.1 of \cite{SW}).

\begin{proposition}
\label{4.2} Assume the family of jump rates $\left\{ \lambda _{v}\left(
\cdot \right) ,v\in \mathcal{V}\right\} $ satisfies Property \ref%
{prop-lambda}. For all $x\in \mathcal{S},\beta \in \Delta ^{d-1}$, 
\begin{equation}
L^{0}\left( x,\beta \right) =\inf_{q\in \lbrack 0,\infty )^{|{\mathcal{V}}%
|}:\sum_{v\in \mathcal{V}}vq_{v}=\beta }\sum_{v\in \mathcal{V}}\lambda
_{v}\left( x\right) \ell \left( \frac{q_{v}}{\lambda _{v}\left( x\right) }%
\right) =L\left( x,\beta \right) .  \label{eqn:inf-conv}
\end{equation}%
Moreover, $I=I^{0}$.
\end{proposition}

\begin{proof}
Defining $h_{v,a}:\mathbb{R}^{d}\rightarrow \mathbb{R}$ by $h_{v,a}\left(
\theta \right) =a\left( \exp \left( \left\langle \theta,v\right\rangle
\right) -1\right) $ for $v\in \mathbb{R}^{d}$ and $a\in \lbrack 0,\infty )$,
we can write $H\left( x,\theta \right) =\sum_{v\in \mathcal{V}}h_{v,\lambda
_{v}\left( x\right) }\left( \theta\right) $. The Legendre-Fenchel transform
of $h_{v,a}$ can be computed explicitly as 
\begin{equation*}
h_{v,a}^{\ast }\left( \beta \right) =\left\{ 
\begin{array}{cc}
a\ell \left( y\right) & \text{if }\beta =avy, \\ 
\infty & \text{otherwise.}%
\end{array}%
\right.
\end{equation*}%
Since $H$ is a finite sum of convex functions, we can apply a standard
result in convex analysis to calculate its Legendre-Fenchel transform (see,
e.g., Theorem D.4.2 of \cite{DE}): 
\begin{equation*}
\left( \sum_{v\in \mathcal{V}}h_{v,\lambda _{v}\left( x\right) }\right)
^{\ast }\left( \beta \right) =\inf \left\{ \sum_{v\in \mathcal{V}%
}h_{v,\lambda _{v}\left( x\right) }^{\ast }\left( \beta _{v}\right)
:\sum_{v\in \mathcal{V}}\beta _{v}=\beta \right\} .
\end{equation*}%
Hence, \eqref{eqn:inf-conv} holds, which immediately implies $I = I^0$.
\end{proof}

\begin{proof}[Proof of the upper bound \eqref{uppbd} and \eqref{ini} in
Theorem \protect\ref{th-ldips}]
The difference between the generators (\ref{Lnl}) and (\ref{Ln0}) is the $n$%
-dependence of the jump rates. However, for every $v \in \mathcal{V}$, the
rate $\lambda _{v}^{n}(\cdot)$ convergences uniformly to $\lambda
_{v}(\cdot) $, it can be shown that the sequences of processes governed by
these two generators have the same large deviation rate function. This can
be proved by adapting the argument in \cite{DEW} or by using a standard
coupling argument to show that the two chains are exponentially equivalent
(see \cite[Section 3.4 and Appendix C]{WW} for complete details). Thus, the
upper bound follows from Proposition \ref{1.1} and Proposition \ref{4.2},
which also imply the compactness of the level sets of $I$ stated in %
\eqref{ini}.
\end{proof}

\bigskip

\section{Properties of the Local Rate Function}

\label{sec-locrf}

In this section we establish useful properties of the proposed local rate
function 
\begin{equation}
L\left( x,\beta \right) = \inf_{q \in [0,\infty)^{|{\mathcal{V}}%
|}:\sum_{v\in \mathcal{V}}vq_{v}=\beta }\sum_{v\in \mathcal{V}}\lambda
_{v}\left( x\right) \ell \left( \frac{q_{v}}{\lambda _{v}\left( x\right) }%
\right) \text{, \ }x\in \mathcal{S},\beta \in \Delta ^{d-1},
\label{eqn:local_rate}
\end{equation}%
first introduced in (\ref{L}). The following observation will be useful in
establishing properties of the function $L$. Given a set of vectors $\left\{
w_{j},j=1,\ldots ,F\right\} \subset \mathbb{R}^{d}$, let the positive cone
spanned by $\left\{ w_{j},j=1,\ldots ,F\right\} $ be denoted by 
\begin{equation}
\mathcal{C}\left( \left\{ w_{j}\right\} \right) \dot{=}\left\{ w\in \mathbb{R%
}^{d}:\text{there exist }a_{j}\geq 0,j=1,\ldots ,F,\text{ with }%
w=\sum_{j=1}^{F}a_{j}w_{j}\right\} .  \label{Cu}
\end{equation}

\begin{remark}
\label{rem-cone} \emph{Define } 
\begin{equation}
\mathcal{V}_{+}\doteq \left\{ v\in \mathcal{V}:\text{for any }a>0\text{, }%
\inf_{x\in \mathcal{S}^{a}}\lambda _{v}\left( x\right) >0\right\} .
\label{cone}
\end{equation}%
\emph{to be the set of directions for which the associated jump rates are
bounded below away from zero on every compact subset of $\mathrm{int}({%
\mathcal{S}})$. We claim that if $\{\lambda _{v}(\cdot ),v\in {\mathcal{V}}%
\} $ satisfies Property \ref{prop-lambda}, Property \ref{prop-communicate}
and property (3) of Lemma \ref{lem-estimates}, then for every $x\in \mathrm{%
int}({\mathcal{S}})$, } 
\begin{equation*}
{\mathcal{C}}(\{v\in {\mathcal{V}}:\lambda _{v}(x)>0\})={\mathcal{C}}({%
\mathcal{V}}_{+})=\Delta ^{d-1}.
\end{equation*}%
\emph{The first equality is a direct consequence of property (3) of Lemma %
\ref{lem-estimates}. To show the second equality, it is clear that ${%
\mathcal{C}}({\mathcal{V}}_{+})\subset \Delta ^{d-1}$. To see why the
reverse containment is true, given any $w\in \Delta ^{d-1}$, choose $x,y\in 
\mathrm{int}({\mathcal{S}})$ such that $y=x+rw$ for some $r>0$. Then
Property \ref{prop-communicate} implies that there exists $t>0$ and a
communicating path $\phi $ on $[0,t]$ from $x$ to $y$. By Definition \ref%
{def-compath}(i), this means that there exists $F\in \mathbb{N}$ and $%
v_{m}\in {\mathcal{V}},m=1,\ldots ,F$, such that $w=\phi (t)-\phi (0)$ is a
positive linear combination of the vectors $v_{m},m=1,\ldots ,F$. On the
other hand, since $y\in \mathrm{int}({\mathcal{S}})$, property ii) of
Definition \ref{def-compath} implies that for each $m=1,\ldots ,F$, $\lambda
_{v_{m}}$ is not identically zero on the simplex. By property (3) of Lemma %
\ref{lem-estimates}, this implies that $v_{m}\in {\mathcal{V}}_{+}$ for
every $m$, which in turn implies $w\in {\mathcal{C}}({\mathcal{V}}_{+})$.
Since $w$ is an arbitrary vector in $\Delta ^{d-1}$, this proves the claim. }
\end{remark}

\begin{lemma}
\label{4.1}Assume that $\{\lambda_v(\cdot), v \in {\mathcal{V}}\}$ satisfies
Property \ref{prop-lambda}, Property \ref{prop-communicate} and property (3)
of Lemma \ref{lem-estimates}. Then $L$ is nonnegative and uniformly
continuous on compact subsets of \emph{int}$\left( \mathcal{S}\right) \times
\Delta ^{d-1}$, and for each $x\in \mathcal{S}$, $L\left( x,\cdot \right) $
is strictly convex on its domain of finiteness.
\end{lemma}

\begin{proof}
$L$ is nonnegative by definition \eqref{eqn:local_rate} and for each $x\in {%
\mathcal{S}}$, relation \eqref{rate fn L0} exhibits $L\left( x,\cdot \right) 
$ as the Legendre-Fenchel transform of the smooth convex function $H(x,\cdot
)$ defined in (\ref{eqn:defofH}). It follows from \cite[Theorem 12.2]{R}
that $L\left( x,\cdot \right) $ is strictly convex on its domain of
finiteness. Since Property \ref{prop-lambda} holds, by Proposition \ref{4.2}
we have $L=L^{0}$. Due to property (3) of Lemma \ref{lem-estimates}, we can
replace the sum over $v\in \mathcal{V}$ in the expression (\ref{eqn:inf-conv}%
) for $L_{0}$ by the sum over $v\in \mathcal{V}_{+}$. According to Remark %
\ref{rem-cone}, under the assumptions of the lemma, the convex cone
generated by $\{v\in \mathcal{V}_{+}\}$ is all of $\Delta ^{d-1}$. Since $%
x\in $ \textrm{int}$\left( \mathcal{S}\right) $ implies that all elements of 
$\{\lambda _{v}(x),v\in \mathcal{V}_{+}\}$ are strictly positive, (\ref%
{eqn:inf-conv}) implies $L\left( x,\beta \right) <\infty $ for $\beta \in
\Delta ^{d-1}$. Since Property \ref{prop-lambda} implies each $\lambda
_{v}(x),v\in \mathcal{V}$, is continuous, the joint continuity of $L$ on 
\textrm{int}$\left( \mathcal{S}\right) \times \Delta ^{d-1}$ follows also
from (\ref{eqn:inf-conv}) and positivity of $\lambda _{v}(x),v\in \mathcal{V}%
_{+}$ for $x\in $\textrm{int}$\left( \mathcal{S}\right) $. This implies
uniform continuity on compact subsets of \textrm{int}$\left( \mathcal{S}%
\right) \times \Delta ^{d-1}$.
\end{proof}

The following elementary inequality can be proved using Legendre transforms.

\begin{lemma}
\label{4.4}For $r, q\in \lbrack 0,\infty )$ we have $r\ell \left( \frac{q}{r}%
\right) +r\left( e-1\right) \geq q$.
\end{lemma}

We now study the asymptotic behavior of $L$ in the second variable.

\begin{proposition}
\label{4.5} Suppose $\{\lambda _{v}(\cdot ),v\in {\mathcal{V}}\}$ satisfies
the assumptions stated in Lemma \ref{4.1}. Given $a>0$, there exist
constants $B=B\left( a\right) <\infty $ and $C_{2}=C_{2}\left( a,B\right)
,C_{3}=C_{3}\left( a,B\right) <\infty $, such that 
\begin{equation*}
L\left( x,\beta \right) \leq \left\{ 
\begin{array}{ll}
C_{2}\left\Vert \beta \right\Vert \log \left\Vert \beta \right\Vert & \text{%
if }x\in \mathcal{S}^{a}\text{ and }\beta \in \Delta ^{d-1},\left\Vert \beta
\right\Vert >B \\ 
C_{3} & \text{if }x\in \mathcal{S}^{a}\text{ and }\beta \in \Delta
^{d-1},\left\Vert \beta \right\Vert \leq B.\text{ }%
\end{array}%
\right.
\end{equation*}%
Moreover, for $B<\infty $ sufficiently large, there exists $%
c_{1}=c_{1}\left( B\right) >0$ such that if $x\in \mathcal{S}$, then 
\begin{equation}
L\left( x,\beta \right) \geq c_{1}\left\Vert \beta \right\Vert \log
\left\Vert \beta \right\Vert \quad \mbox{ for all }\beta \in \Delta
^{d-1},||\beta ||>B.  \label{lxbeta-lower}
\end{equation}%
In particular, $\beta \mapsto L\left( x,\beta \right) $ is superlinear,
uniformly in $x$.
\end{proposition}

\begin{proof}
Fix $a>0$. For any $B<\infty $, since $\left\{ \left( x,\beta \right) \in 
\mathcal{S}^{a}\mathcal{\times }\Delta^{d-1}:\left\Vert \beta \right\Vert
\leq B\right\} $ is a compact subset of int$\left( \mathcal{S}\right) \times
\Delta^{d-1}$, the uniform boundedness of $L$ on this set follows directly
from the uniform continuity of $L$ established in Lemma \ref{4.1}.

For the upper bound when $\left\Vert \beta \right\Vert >B$, we first assume $%
\left\Vert \beta \right\Vert =1$. By Remark \ref{rem-cone}, there exists a
vector $q=q\left( \beta \right) \in \lbrack 0,\infty )^{\left\vert \mathcal{V%
}\right\vert }$, such that $\sum_{v\in \mathcal{V}_{+}}vq_{v}=\beta $, $%
q_{v}>0$ for $v\in \mathcal{V}_{+}$ and $q_{v}=0$ for $v\in \mathcal{V}%
\backslash \mathcal{V}_{+}$. Since ${\mathcal{C}}({\mathcal{V}}_{+})=\Delta
^{d-1}$, we can assume $\max_{v,\left\Vert \beta \right\Vert =1}\left\vert
q_{v}\left( \beta \right) \right\vert $ is finite. By scaling, it follows
that there exists some constant $c_{0}<\infty $, such that for any $\beta
\in \Delta ^{d-1}$, there exists a vector $q\in \mathbb{[}0,\infty
)^{\left\vert \mathcal{V}\right\vert }$ such that $\sum_{v\in \mathcal{V}%
_{+}}vq_{v}=\beta $, $\max_{v}\left\vert q_{v}\right\vert \leq
c_{0}\left\Vert \beta \right\Vert $, and $q_{v}=0$ for $v\in \mathcal{V}%
\backslash \mathcal{V}_{+}$. It follows that for some $c_{4}<\infty $, 
\begin{equation*}
L\left( x,\beta \right) \leq c_{4}\sum_{v\in \mathcal{V}}q_{v}\log \frac{%
q_{v}}{\lambda _{v}\left( x\right) }\leq C_{2}\left\Vert \beta \right\Vert
\log \left\Vert \beta \right\Vert
\end{equation*}%
if $\left\Vert \beta \right\Vert \geq B$, for some $B=B(a)$ sufficiently
large and all $x\in \mathcal{S}^{a}$. This finishes the proof of the upper
bound.

Now, consider the lower bound for $L$ on $\left\{ \left( x,\beta \right) \in 
{\mathcal{S}}\times \Delta ^{d-1}:\left\Vert \beta \right\Vert >B\right\} $.
Since $L=L^{0}$ due to Proposition \ref{4.2}, by the definition 
\eqref{rate
fn L0} of $L^{0}$, we have for $t>0,$ $\theta =t\frac{\beta }{\left\Vert
\beta \right\Vert },$ and $R<\infty $ defined as in (\ref{m}), 
\begin{align*}
L\left( x,\beta \right) & \geq \left\langle \theta ,\beta \right\rangle
-H\left( x,\theta \right) \\
& \geq t\left\Vert \beta \right\Vert -\sum_{v\in \mathcal{V}}\lambda
_{v}\left( x\right) \exp (\left\langle \theta ,v\right\rangle ) \\
& \geq t\left\Vert \beta \right\Vert -R\left\vert \mathcal{V}\right\vert
\exp \left( \max_{v\in \mathcal{V}}\left\Vert v\right\Vert t\right) .
\end{align*}%
Substituting $t=\frac{1}{\max_{v\in \mathcal{V}}\left\Vert v\right\Vert }%
\log \left\Vert \beta \right\Vert $, this implies 
\begin{equation*}
L\left( x,\beta \right) \geq \frac{1}{\max_{v\in \mathcal{V}}\left\Vert
v\right\Vert }\left\Vert \beta \right\Vert \log \left\Vert \beta \right\Vert
-R\left\vert \mathcal{V}\right\vert \left\Vert \beta \right\Vert \geq
c_{1}\left\Vert \beta \right\Vert \log \left\Vert \beta \right\Vert
\end{equation*}%
for some constant $c_{1}>0,$ provided $\left\Vert \beta \right\Vert $ is
sufficiently large.
\end{proof}

Recall that, given $\xi > 0$, $D\left( \left[ a,b\right] :{\mathcal{S}}%
^\xi\right)$ denotes the space of c\`{a}dl\`{a}g functions on $[a,b]$ taking
values in ${\mathcal{S}}^\xi$.

\begin{proposition}
\label{4.7} Suppose $\{\lambda _{v}(\cdot ),v\in {\mathcal{V}}\}$ satisfies
the assumptions stated in Lemma \ref{4.1}. Given $0\leq a<b\leq 1$ and $\xi
>0$, suppose that $\gamma \in AC\left( \left[ a,b\right] :\mathcal{S}^{\xi
}\right) $ satisfies $\int_{a}^{b}L\left( \gamma \left( s\right) ,\dot{\gamma%
}\left( s\right) \right) ds<\infty $. Let $\left\{ \gamma ^{\delta }\right\}
_{\delta \in (0,1)}\subset D\left( \left[ a,b\right] :\mathcal{S}^{\xi
}\right) $ be such that $\sup_{t\in \left[ a,b\right] }\left\Vert \gamma
^{\delta }\left( t\right) -\gamma \left( t\right) \right\Vert \rightarrow 0$
as $\delta \rightarrow 0$. Then for any $\varepsilon >0$, there exists $%
\delta _{0}=\delta _{0}\left( \xi ,\varepsilon \right) >0$ such that for $%
\delta <\delta _{0}$, 
\begin{equation*}
\left\vert \int_{a}^{b}L\left( \gamma \left( s\right) ,\dot{\gamma}\left(
s\right) \right) ds-\int_{a}^{b}L\left( \gamma ^{\delta }\left( s\right) ,%
\dot{\gamma}\left( s\right) \right) ds\right\vert <\varepsilon .
\end{equation*}
\end{proposition}

\begin{proof}
Fix $\xi >0$ and $0\leq a<b\leq 1$. Let $A$ be the measurable set of points $%
s\in \left[ a,b\right] $ for which $\dot{\gamma}\left( s\right) $ is well
defined and lies in $\Delta ^{d-1}$, so that $\left[ a,b\right] \backslash A$
has zero Lebesgue measure. Let $C_{2}<\infty ,c_{1}>0$ and $B<\infty $ be
chosen according to Proposition \ref{4.5} so that $L\left( x,\beta \right)
\leq C_{2}\left\Vert \beta \right\Vert \log \left\Vert \beta \right\Vert $
if $x\in \mathcal{S}^{\xi /2}$ and $\left\Vert \beta \right\Vert >B$, and $%
L\left( x,\beta \right) \geq c_{1}\left\Vert \beta \right\Vert \log
\left\Vert \beta \right\Vert $ if $x\in \mathcal{S}$ and $\beta \in \Delta
^{d-1}$, $\left\Vert \beta \right\Vert >B$. For $\bar{B}\in (B,\infty )$
define $\bar{A}\doteq \left\{ s\in A:\left\Vert \dot{\gamma}\left( s\right)
\right\Vert \leq \bar{B}\right\} $. Assume $\delta _{0}>0$ is small enough
that $\gamma ^{\delta }\left( s\right) \in \mathcal{S}^{\xi /2}$ for all $%
\delta <\delta _{0}$ and $s\in \lbrack a,b]$. Then 
\begin{equation*}
\int_{\left[ a,b\right] \backslash \bar{A}}L\left( \gamma ^{\delta }\left(
s\right) ,\dot{\gamma}\left( s\right) \right) ds\leq C_{2}\int_{A\backslash 
\bar{A}}\left\Vert \dot{\gamma}\left( s\right) \right\Vert \log \left\Vert 
\dot{\gamma}\left( s\right) \right\Vert ds\leq \frac{C_{2}}{c_{1}}\int_{%
\left[ a,b\right] \backslash \bar{A}}L\left( \gamma \left( s\right) ,\dot{%
\gamma}\left( s\right) \right) ds,
\end{equation*}%
and since by assumption $s\mapsto L(\gamma (s),\dot{\gamma}(s))$ is
integrable on $[a,b]$, for large enough $\bar{B}<\infty $, 
\begin{equation}
\int_{\left[ a,b\right] \backslash \bar{A}}L\left( \gamma ^{\delta }\left(
s\right) ,\dot{\gamma}\left( s\right) \right) ds+\int_{\left[ a,b\right]
\backslash \bar{A}}L\left( \gamma \left( s\right) ,\dot{\gamma}\left(
s\right) \right) ds\leq \varepsilon /2.  \label{4.7.1}
\end{equation}%
On the other hand, since $\int_{a}^{b}L(\gamma \left( s\right) ,\dot{\gamma}%
\left( s\right) )ds<\infty $, by dominated convergence and the continuity of 
$L\left( \cdot ,\beta \right) $ for fixed $\beta \in \Delta ^{d-1}$
established in Lemma \ref{4.1}, we have 
\begin{equation}
\int_{\bar{A}}L\left( \gamma ^{\delta }\left( s\right) ,\dot{\gamma}\left(
s\right) \right) ds\rightarrow \int_{\bar{A}}L\left( \gamma \left( s\right) ,%
\dot{\gamma}\left( s\right) \right) ds.  \label{4.7.2}
\end{equation}%
Hence by choosing $\delta _{0}>0$ smaller if need be, (\ref{4.7.1}) and (\ref%
{4.7.2}) imply that for $\delta \in (0,\delta _{0})$ 
\begin{equation*}
\left\vert \int_{a}^{b}L\left( \gamma \left( s\right) ,\dot{\gamma}\left(
s\right) \right) ds-\int_{a}^{b}L\left( \gamma ^{\delta }\left( s\right) ,%
\dot{\gamma}\left( s\right) \right) ds\right\vert \leq \varepsilon .
\end{equation*}
\end{proof}

\begin{lemma}
\label{path} Suppose $\{\lambda _{v}(\cdot ),v\in {\mathcal{V}}\}$ satisfies
the assumptions stated in Lemma \ref{4.1}. Suppose that $\gamma \in AC\left( %
\left[ 0,1\right] :\mathcal{S}\right) $ satisfies $\int_{a}^{b}L\left(
\gamma \left( s\right) ,\dot{\gamma}\left( s\right) \right) ds<\infty $ for
some $0\leq a<b\leq 1$. Then 
\begin{equation*}
\left\Vert \gamma \left( t\right) -\gamma \left( a\right) \right\Vert \log 
\frac{1}{t-a}\rightarrow 0\text{ }\ \text{as }t\downarrow a.
\end{equation*}
\end{lemma}

\begin{proof}
Fix $t\in \left( a,b\right) $ and let $A$ be the measurable set of points $%
s\in \left[ a,b\right] $ be for which $\dot{\gamma}\left( s\right) $ is well
defined and lies in $\Delta ^{d-1}$. We claim, and show below, that $%
\left\Vert \dot{\gamma}\left( \cdot \right) \right\Vert \log \left\Vert \dot{%
\gamma}\left( \cdot \right) \right\Vert $ is integrable on $[a,b]$. By
Proposition \ref{4.5}, there exists $B$ sufficiently large and $c_{1}(B)>0 $
such that \eqref{lxbeta-lower} holds. Therefore, defining $A_{1}\doteq
\left\{ s\in A:\left\Vert \dot{\gamma}\left( s\right) \right\Vert \leq
B\right\} $, we have 
\begin{align*}
\int_{a}^{b}\left\Vert \dot{\gamma}\left( s\right) \right\Vert \log
\left\Vert \dot{\gamma}\left( s\right) \right\Vert ds=\int_{A}\left\Vert 
\dot{\gamma}\left( s\right) \right\Vert \log \left\Vert \dot{\gamma}\left(
s\right) \right\Vert ds& \leq \frac{1}{c_{1}}\int_{A\setminus A_{1}}L\left(
\gamma \left( s\right) ,\dot{\gamma}\left( s\right) \right)
ds+\int_{A_{1}}B\log Bds \\
& \leq \frac{1}{c_{1}}\int_{a}^{b}L\left( \gamma \left( s\right) ,\dot{\gamma%
}\left( s\right) \right) ds+\left( B\log B\right) \left( b-a\right) ,
\end{align*}%
where the last inequality uses the nonnegativity of $L$. On the other hand,
by Jensen's inequality, for $t \in (a,1)$, 
\begin{align*}
\int_{a}^{t}\left\Vert \dot{\gamma}\left( s\right) \right\Vert \log
\left\Vert \dot{\gamma}\left( s\right) \right\Vert ds& \geq \left(
t-a\right) \left\Vert \frac{\gamma \left( t\right) -\gamma \left( a\right) }{%
t-a}\right\Vert \log \left\Vert \frac{\gamma \left( t\right) -\gamma \left(
a\right) }{t-a}\right\Vert \\
& =\left\Vert \gamma \left( t\right) -\gamma \left( a\right) \right\Vert
\log \frac{\left\Vert \gamma \left( t\right) -\gamma \left( a\right)
\right\Vert }{t-a}.
\end{align*}%
Now, since $||\dot{\gamma}(\cdot)||\log ||\dot{\gamma}(\cdot)||$ is
integrable, the left-hand side of the last display goes to zero as $t
\downarrow a$. The lemma follows by observing that $\left\Vert \gamma \left(
t\right) -\gamma \left( a\right) \right\Vert \log \left\Vert \gamma \left(
t\right) -\gamma \left( a\right) \right\Vert $ also goes to zero as $%
t\downarrow a$.
\end{proof}

Recall the definition of ${\mathcal{V}}_{+}$ given in Remark \ref{rem-cone}.
The following result is used in Lemma \ref{6.3}, which contains a
perturbation argument used in proving the LDP lower bound.

\begin{lemma}
\label{abs4.7''} Suppose $\{\lambda_v(\cdot), v \in {\mathcal{V}}\}$
satisfies the assumptions stated in Lemma \ref{4.1}. Let $c_{0}\left( \rho
\right) $ be given such that $c_{0}\left( \rho \right) \rightarrow 1$ as $%
\rho \rightarrow 0$. Suppose that $x\in \mathcal{S}$, $\left\{ x^{\rho
}\right\} _{\rho >0}\subset \mathrm{int}\left( \mathcal{S}\right) $, are
such that $\left\Vert x-x^{\rho }\right\Vert \rightarrow 0$ as $\rho
\rightarrow 0$, and for any $\rho >0$ and $v\in \mathcal{V}_{+}$, $\lambda
_{v}\left( x\right) /\lambda _{v}\left( x^{\rho }\right) \leq c_{0}\left(
\rho \right) $. Then there exists $c=c\left( \rho \right) $ that only
depends on $c_{0}\left( \rho \right) $ and $\left\Vert x-x^{\rho
}\right\Vert $, that satisfies $c(\rho )\rightarrow 0$ as $\rho \rightarrow
0 $, and has the property that 
\begin{equation}
L\left( x^{\rho },\beta \right) \leq \left( 1+c\left( \rho \right) \right)
L\left( x,\beta \right) +c\left( \rho \right) ,\quad \beta \in \Delta ^{d-1}.
\label{x-xdelta}
\end{equation}
\end{lemma}

\begin{proof}
Fix $\beta \in \Delta ^{d-1}$. We can assume without loss of generality that
there exists $q\in \mathbb{[}0,\infty )^{\left\vert \mathcal{V}%
_{+}\right\vert }$ such that $\sum_{v\in {\mathcal{V}}_{+}}vq_{v}=\beta $
because, if not, then $L(x,\beta )$ is infinite and \eqref{x-xdelta} holds
trivially. Now, we claim (and justify below) that to prove the lemma, it
suffices to show that for every $\rho >0$, there exists a function $c\left(
\rho \right) $ (depending only on $||x^{\rho }-x||$ and $c_{0}(\rho )$) such
that for every $q\in \lbrack 0,\infty )^{\left\vert \mathcal{V}%
_{+}\right\vert }$ such that $\sum_{v\in {\mathcal{V}}_{+}}vq_{v}=\beta $, 
\begin{equation}
\sum_{v\in {\mathcal{V}}_{+}}\lambda _{v}\left( x^{\rho }\right) \ell \left( 
\frac{q_{v}}{\lambda _{v}\left( x^{\rho }\right) }\right) \leq \left(
1+c\left( \rho \right) \right) \sum_{v\in {\mathcal{V}}_{+}}\lambda
_{v}\left( x\right) \ell \left( \frac{q_{v}}{\lambda _{v}\left( x\right) }%
\right) +c\left( \rho \right) ,  \label{i}
\end{equation}%
and $c(\rho )\rightarrow 0$ as $\rho \rightarrow 0$. To see that the claim
holds, recall the expression (\ref{eqn:inf-conv}) for $L$ and note that the
left-hand side of \eqref{x-xdelta} is dominated by the left-hand side of %
\eqref{i}. The right-hand side of \eqref{x-xdelta} is the infimum of the
right-hand side of \eqref{i} over all $q\in \mathbb{[}0,\infty )^{\left\vert 
\mathcal{V}_{+}\right\vert }$ such that $\sum_{v\in {\mathcal{V}}%
_{+}}vq_{v}=\beta $, where we have used the fact that $\lambda _{v}(x)>0$
then $v\in {\mathcal{V}}_{+}$, which follows from property (3) of Lemma \ref%
{lem-estimates}.

We have the following relations, each line of which is explained below. 
\begin{align*}
& \sum_{v\in \mathcal{V}_{+}}\lambda _{v}\left( x^{\rho }\right) \ell \left( 
\frac{q_{v}}{\lambda _{v}\left( x^{\rho }\right) }\right) -\sum_{v\in 
\mathcal{V}_{+}}\lambda _{v}\left( x\right) \ell \left( \frac{q_{v}}{\lambda
_{v}\left( x\right) }\right) \\
& \quad =\sum_{v\in \mathcal{V}_{+}}q_{v}\log \frac{\lambda _{v}\left(
x\right) }{\lambda _{v}\left( x^{\rho }\right) }+\sum_{v\in \mathcal{V}%
_{+}}\left( \lambda _{v}\left( x^{\rho }\right) -\lambda _{v}\left( x\right)
\right) \\
& \quad \leq \log c_{0}\left( \rho \right) \sum_{v\in \mathcal{V}%
_{+}}q_{v}+C_{1}\left\Vert x^{\rho }-x\right\Vert \\
& \quad \leq \log c_{0}\left( \rho \right) \sum_{v\in \mathcal{V}_{+}}\left(
\lambda _{v}\left( x\right) \ell \left( \frac{q_{v}}{\lambda _{v}\left(
x\right) }\right) +\lambda _{v}\left( x\right) \left( e-1\right) \right)
+C_{1}\left\Vert x^{\rho }-x\right\Vert \\
& \quad \leq \log c_{0}\left( \rho \right) \sum_{v\in \mathcal{V}%
_{+}}\lambda _{v}\left( x\right) \ell \left( \frac{q_{v}}{\lambda _{v}\left(
x\right) }\right) +\bar{C}\left( \rho \right) .
\end{align*}%
The equality follows from the expression \eqref{l} for $\ell $; the first
inequality (with $C_{1}<\infty $) due to the assumption of the lemma and the
Lipschitz continuity of $\lambda _{v}$; the second inequality follows from
Lemma \ref{4.4} with $r=\lambda _{v}(x)$ and $q=q_{v}$; and the final
inequality just uses the definition $\bar{C}\left( \rho \right) \doteq \log
c_{0}\left( \rho \right) R\left\vert \mathcal{V}\right\vert \left(
e-1\right) +C_{1}\left\Vert x-x^{\rho }\right\Vert $. Then \eqref{i} holds
with $c\left( \rho \right) =\max \left\{ \log c_{0}\left( \rho \right) ,\bar{%
C}\left( \rho \right) \right\} $. Since $c(\rho )$ depends only on $%
c_{0}(\rho )$ and $\left\Vert x-x^{\rho }\right\Vert $, and the assumptions
of the lemma imply that $c(\rho )\rightarrow 0$ as $\rho \rightarrow 0$,
this completes the proof.
\end{proof}

For $t\in \left[ 0,1\right] $ and $c>0$, define%
\begin{equation*}
\gamma _{c}\left( s\right) \doteq \gamma \left( cs\right) ,\text{ \ }s\in %
\left[ 0,t\right] ,
\end{equation*}%
which is a time reparametrization of $\gamma $. The next result is used in
the proof of the locally uniform LDP in Section 9. It states that given a
path $\gamma $ with finite cost, the cost of the path depends continuously
on the reparameterization of time.

\begin{proposition}
\label{4.8} Suppose $\{\lambda _{v}(\cdot ),v\in {\mathcal{V}}\}$ satisfies
the assumptions stated in Lemma \ref{4.1}. For $t\in \lbrack 0,1)$, suppose $%
\gamma \in AC\left( \left[ 0,1\right] :\mathcal{S}\right) $ is such that $%
I_{t}\left( \gamma \right) <\infty $. Then the function $c\mapsto
I_{t/c}\left( \gamma _{c}\right) $ is continuous at $1$.
\end{proposition}

\begin{proof}
First note that for $c$ close to $1$, $\gamma _{c}\in AC\left( \left[ 0,t/c%
\right] :\mathcal{S}\right) $ and 
\begin{equation*}
I_{t/c}\left( \gamma _{c}\right) =\int_{0}^{t/c}L\left( \gamma \left(
cs\right) ,c\dot{\gamma}\left( cs\right) \right) ds=\frac{1}{c}%
\int_{0}^{t}L\left( \gamma \left( r\right) ,c\dot{\gamma}\left( r\right)
\right) dr.
\end{equation*}%
We now bound the integral of $\frac{1}{c}L\left( \gamma ,c\dot{\gamma}%
\right) -L\left( \gamma ,\dot{\gamma}\right) $ over $[0,t]$. Recall the
definition of $L$ in (\ref{L}). Since $\gamma $ is absolutely continuous and 
$I_{t}(\gamma )<\infty $, $\dot{\gamma}\left( u\right) $ is well defined and 
$L(\gamma (u),\dot{\gamma}(u))<\infty $ for almost every $u\in \lbrack 0,t]$%
. Thus, for any such $u\in \left[ 0,t\right] $ and $\varepsilon >0$, there
exists $q\in \mathbb{[}0,\infty )^{\left\vert \mathcal{V}\right\vert }$ such
that $\sum_{v\in \mathcal{V}}vq_{v}=c\dot{\gamma}\left( u\right) $ and 
\begin{equation*}
\sum_{v\in \mathcal{V}}\lambda _{v}\left( \gamma \left( u\right) \right)
\ell \left( \frac{q_{v}/c}{\lambda _{v}\left( \gamma \left( u\right) \right) 
}\right) \leq L\left( \gamma \left( u\right) ,\dot{\gamma}\left( u\right)
\right) +\varepsilon .
\end{equation*}%
On the other hand, using the expression \eqref{l} for $\ell $ we also have 
\begin{align*}
& \sum_{v\in \mathcal{V}}\lambda _{v}\left( \gamma \left( u\right) \right)
\ell \left( \frac{q_{v}/c}{\lambda _{v}\left( \gamma \left( u\right) \right) 
}\right) \\
& \quad =\sum_{v\in \mathcal{V}}\left( \frac{q_{v}}{c}\log \frac{q_{v}/c}{%
\lambda _{v}\left( \gamma \left( u\right) \right) }-\frac{q_{v}}{c}+\lambda
_{v}\left( \gamma \left( u\right) \right) \right) \\
& \quad =\frac{1}{c}\sum_{v\in \mathcal{V}}\left( q_{v}\log \frac{q_{v}}{%
\lambda _{v}\left( \gamma \left( u\right) \right) }-q_{v}+\lambda _{v}\left(
\gamma \left( u\right) \right) \right) +\left( \frac{1}{c}\log \frac{1}{c}%
\right) \sum_{v\in \mathcal{V}}q_{v} \\
& \quad \quad +\left( 1-\frac{1}{c}\right) \sum_{v\in \mathcal{V}}\lambda
_{v}\left( \gamma \left( u\right) \right) \\
& \quad \geq \frac{1}{c}L\left( \gamma \left( u\right) ,c\dot{\gamma}\left(
u\right) \right) +\left( \frac{1}{c}\log \frac{1}{c}\right) \sum_{v\in 
\mathcal{V}}q_{v}+\left( 1-\frac{1}{c}\right) \sum_{v\in \mathcal{V}}\lambda
_{v}\left( \gamma \left( u\right) \right) .
\end{align*}%
The last two relations imply that 
\begin{equation}
\frac{1}{c}L\left( \gamma \left( u\right) ,c\dot{\gamma}\left( u\right)
\right) -L\left( \gamma \left( u\right) ,\dot{\gamma}\left( u\right) \right)
\leq \varepsilon -\left( \frac{1}{c}\log \frac{1}{c}\right) \sum_{v\in 
\mathcal{V}}q_{v}-\left( 1-\frac{1}{c}\right) \sum_{v\in \mathcal{V}}\lambda
_{v}\left( \gamma \left( u\right) \right) .  \label{first}
\end{equation}

Similarly, by taking $q\in \mathbb{[}0,\infty )^{\left\vert \mathcal{V}%
\right\vert }$, such that $\sum_{v\in \mathcal{V}}vq_{v}=\dot{\gamma}\left(
u\right) $ and 
\begin{equation*}
\sum_{v\in \mathcal{V}}\lambda _{v}\left( \gamma \left( u\right) \right)
\ell \left( \frac{cq_{v}}{\lambda _{v}\left( \gamma \left( u\right) \right) }%
\right) \leq L\left( \gamma \left( u\right) ,c\dot{\gamma}\left( u\right)
\right) +c\varepsilon ,
\end{equation*}%
an analogous computation yields%
\begin{equation}
\frac{1}{c}L\left( \gamma \left( u\right) ,c\dot{\gamma}\left( u\right)
\right) -L\left( \gamma \left( u\right) ,\dot{\gamma}\left( u\right) \right)
\geq \left( \log c\right) \sum_{v\in \mathcal{V}}q_{v}-\left( 1-\frac{1}{c}%
\right) \sum_{v\in \mathcal{V}}\lambda _{v}\left( \gamma \left( u\right)
\right) -\varepsilon .  \label{second}
\end{equation}
We now apply Lemma \ref{4.4} with $r=\lambda _{v}\left( \gamma \left(
u\right) \right) $ and $q=q_{v}/c$, and the bound \eqref{m} for $\lambda
_{v} $ to obtain 
\begin{eqnarray*}
\frac{1}{c}\sum_{v\in \mathcal{V}}q_{v} &\leq &\sum_{v\in \mathcal{V}}\left(
\lambda _{v}\left( \gamma \left( u\right) \right) \ell \left( \frac{q_{v}/c}{%
\lambda _{v}\left( \gamma \left( u\right) \right) }\right) +\lambda
_{v}\left( \gamma \left( u\right) \right) \left( e-1\right) \right) \\
&\leq &L\left( \gamma \left( u\right) ,\dot{\gamma}\left( u\right) \right)
+\varepsilon +R_{1}
\end{eqnarray*}%
where $R_{1} \doteq |{\mathcal{V}}| R (e-1) <\infty $, with $R$ being the
bound in \eqref{m}. Combining this with (\ref{first}) and (\ref{second}), we
see that for $c\ $sufficiently close to $1,$ 
\begin{align*}
& \left\vert \frac{1}{c}L\left( \gamma \left( u\right) ,c\dot{\gamma}\left(
u\right) \right) -L\left( \gamma \left( u\right) ,\dot{\gamma}\left(
u\right) \right) \right\vert \\
& \quad \leq M\left\vert \mathcal{V}\right\vert \left\vert 1-\frac{1}{c}%
\right\vert +\max \left\{ \log \frac{1}{c},c\log c\right\} \left( L\left(
\gamma \left( u\right) ,\dot{\gamma}\left( u\right) \right) +\varepsilon
+M_{1}\right) +\varepsilon .
\end{align*}
Since this holds for almost every $u \in [0,t]$, one can first integrate
over $\left[ 0,t\right] $, then take $c\rightarrow 1$ and use the finiteness
of $I (\gamma)$ and finally send $\varepsilon \rightarrow 0$ to complete the
proof.
\end{proof}

\section{\protect\bigskip Proof of the LDP lower bound}

\label{sec-pflower}

We now turn to the proof of the LDP lower bound, which we will establish for
a somewhat larger class of jump Markov processes than the empirical measure
processes. Again, we assume for each $n\in \mathbb{N}$, $\mu ^{n}(\cdot )$
is a jump Markov process on ${\mathcal{S}}_{n}$ with generator ${\mathcal{L}}%
_{n}$ in \eqref{Lnl} such that the associated sequence of rates $\{\lambda
_{v}^{n}(\cdot )\}_{v\in {\mathcal{V}}},n\in \mathbb{N},$ satisfy Property %
\ref{prop-lambda}, that is, converge uniformly to suitable Lipschitz
continuous functions $\{\lambda _{v}(\cdot )\}_{v\in {\mathcal{V}}}$.
Additional conditions imposed on $\{\lambda _{v}(\cdot )\}_{v\in {\mathcal{V}%
}}$ will be stated in the lemmas below. Recall that for notational
convenience we assume the time interval is of the form $[0,1]$. To prove the
lower bound it suffices to show that for any fixed trajectory $\gamma \in
D\left( \left[ 0,1\right] :\mathcal{S}\right) $, given any $\varepsilon >0$
and $\delta >0$ there exists $\eta >0$ such that if $\left\Vert \mu
^{n}\left( 0\right) -\gamma \left( 0\right) \right\Vert <\eta $ for all $n$
large enough, 
\begin{equation}
\liminf_{n\rightarrow \infty }\frac{1}{n}\log \mathbb{P}\left( \left\Vert
\mu ^{n}-\gamma \right\Vert _{\infty }<\delta \right) \geq -I\left( \gamma
\right) -\varepsilon .  \label{red-lbound}
\end{equation}%
Without loss of generality we assume $I\left( \gamma \right) <\infty $,
which in particular implies that $\gamma \in AC([0,1],{\mathcal{S}})$.

One source of difficulty here is that the transition rates of $\mu ^{n}$ may
tend to zero as $\mu ^{n}$ approaches the boundary of $\mathcal{S}$, which
could lead to singularity of the local rate function. Our approach here
adapts an idea from the study of a discrete time model in \cite{DNW}. We
first show that the singularity can be avoided except at $t=0$, by slightly
perturbing the original path, with arbitrarily small additional cost.

\subsection{Perturbation argument}

\label{subs-pert}

The idea of the perturbation argument is as follows. Recall the definition
of $\mathcal{S}^{a}$ in (\ref{S}). For any $a>0$ fixed, by property (3) of
Lemma \ref{lem-estimates}, the rates $\lambda _{v}\left( \cdot \right) $ are
either identically zero or uniformly bounded below away from zero within $%
\mathcal{S}^{a}$. Therefore, a standard approximation argument can be used
to establish the LDP in $\mathcal{S}^{a}$, uniformly with respect to the
initial condition. When $\gamma \left( 0\right) =x\in \mathcal{S}/\mathcal{S}%
^{a}$, by using Proposition \ref{prop-lln}, one can construct a perturbed
trajectory of $\gamma $ that hits $\mathcal{S}^{a}$ in an arbitrarily short
time as $a\rightarrow 0$, and in such a way that the difference in cost
between $\gamma $ and the perturbed trajectory can be made sufficiently
small.

\begin{lemma}
\label{6.3}Assume the family of jump rates $\left\{ \lambda _{v}\left(
\cdot\right),v\in \mathcal{V}\right\} $ satisfies Property \ref{prop-lambda}%
, Property \ref{prop-communicate}, properties (2) and (3) of Lemma \ref%
{lem-estimates}, and the associated LLN trajectory satisfies Property \ref%
{cond-lln}. Consider $\gamma \in AC\left( \left[ 0,1\right] :\mathcal{S}%
\right) $ such that $I\left( \gamma \right) <\infty $. Then given any $%
\varepsilon >0$, there exists $\tilde{b}>0$, $D<\infty $ and a trajectory $%
\psi \in AC\left( \left[ 0,1\right] :\mathcal{S}\right) $ such that \newline
i) $\psi \left( 0\right) =\gamma \left( 0\right) $ and $\left\Vert \psi
-\gamma \right\Vert _{\infty }<\varepsilon $,

ii) $\psi _{i}\left( t\right) \geq \tilde{b}t^{D}$ for $i=1,...,d$ and any $%
t\in \left[ 0,1\right] $,

iii) $I\left( \psi \right) \leq I\left( \gamma \right) +\varepsilon $.
\end{lemma}

\begin{proof}
For $0<\rho <1$ define $\psi ^{\rho }\doteq \rho \mu +\left( 1-\rho \right)
\gamma $, where $\mu $ is the law of large numbers trajectory defined in (%
\ref{1}) with $\mu \left( 0\right) =\gamma \left( 0\right) $. Let $%
C_{d}\doteq \max_{x,y\in {\mathcal{S}}}||x-y||$ be the diameter of $\mathcal{%
S}$. Then we have $\psi (0)=\gamma (0)$ and $\left\Vert \psi ^{\rho }-\gamma
\right\Vert _{\infty }=\rho \left\Vert \mu -\gamma \right\Vert _{\infty
}\leq C_{d}\rho $. By Property \ref{cond-lln} of the LLN trajectory, there
exist $b>0$ and $D<\infty $ such that $\mu _{i}(t)\geq bt^{D}$ for $%
i=1,\ldots ,d$, which in turn implies the lower bound $\psi _{i}^{\rho
}\left( t\right) \geq \rho \mu _{i}\left( t\right) \geq \rho bt^{D}$, $%
i=1,\ldots ,d$. Thus, for all $\rho <\varepsilon /C_{d}$, $\psi =\psi ^{\rho
}$ satisfies property (i) and property (ii) holds with $\tilde{b}\doteq \rho
b>0$. It only remains to show that there exists some $\rho \in
(0,\varepsilon /C_{d})$ such that $\psi =\psi ^{\rho }$ satisfies property
(iii). We first show that there exists $c\left( \rho \right) <\infty $ which
goes to zero as $\rho \rightarrow 0$, such that for almost every $t\in \left[
0,1\right] $, 
\begin{equation}
L\left( \psi ^{\rho }\left( t\right) ,\dot{\gamma}\left( t\right) \right)
\leq \left( 1+c\left( \rho \right) \right) L\left( \gamma \left( t\right) ,%
\dot{\gamma}\left( t\right) \right) +c\left( \rho \right) .  \label{as}
\end{equation}

For $t\in (0,1]$, property (ii) shows that $\psi ^{\rho }(t)\in \mathrm{int}(%
{\mathcal{S}})$, and we also have $\gamma _{i}\left( t\right) /\psi
_{i}^{\rho }\left( t\right) \leq 1/(1-\rho )$ for every $i=1,\ldots ,d$. By
property (2) of Lemma \ref{lem-estimates}, there exists a function $\bar{C}%
:[0,\infty )\mapsto \lbrack 0,\infty )$ with $\bar{C}(r)\rightarrow 1$ as $%
r\rightarrow 0$ such that 
\begin{equation*}
\frac{\lambda _{v}(\gamma (t))}{\lambda _{v}(\psi ^{\rho }(t))}\leq \bar{C}%
(||\psi ^{\rho }(t)-\gamma (t)||_{\infty })\prod_{i=1,\ldots ,d:\gamma
_{i}(t)>\psi _{i}^{\rho }(t)}\left( \frac{\gamma _{i}(t)}{\psi _{i}^{\rho
}(t)}\right) ^{K}\leq c_{0}(\rho ),
\end{equation*}%
where 
\begin{equation*}
c_{0}(\rho )\doteq \bar{C}(||\psi ^{\rho }(t)-\gamma (t)||_{\infty })\left( 
\frac{1}{1-\rho }\right) ^{Kd}.
\end{equation*}%
As $\rho \rightarrow 0$, $c_{0}(\rho )\rightarrow 1$ because $||\psi ^{\rho
}(t)-\gamma (t)||_{\infty }\rightarrow 0$. Thus, an application of Lemma \ref%
{abs4.7''} with $x=\gamma \left( t\right) $ and $x^{\rho }=\psi ^{\rho
}\left( t\right) $ shows that (\ref{as}) holds for suitable $c(\rho )$.
Likewise, since $\mu _{i}\left( t\right) /\psi _{i}^{\rho }\left( t\right)
\leq 1/\rho $ for $i=1,\ldots ,d$, 
property (2) of Lemma \ref{lem-estimates} implies 
\begin{equation}
\frac{\lambda _{v}\left( \mu \left( t\right) \right) }{\lambda _{v}\left(
\psi ^{\rho }\left( t\right) \right) }\leq \bar{C}_{\ast }\left( \frac{1}{%
\rho }\right) ^{Kd},  \label{ratiov}
\end{equation}%
where $\bar{C}_{\ast }\doteq \max_{R\in \lbrack 0,C_{d}]}\bar{C}(r)$ is
finite because $\bar{C}$ is continuous. Therefore, by the definition of $L$
in (\ref{L}) and the fact that $\dot{\mu}\left( t\right) =\sum_{v\in 
\mathcal{V}}v\lambda _{v}\left( \mu \left( t\right) \right) $, we have 
\begin{eqnarray}
L\left( \psi ^{\rho }\left( t\right) ,\dot{\mu}\left( t\right) \right) &\leq
&\sum_{v\in \mathcal{V}}\lambda _{v}\left( \psi ^{\rho }\left( t\right)
\right) \ell \left( \frac{\lambda _{v}\left( \mu \left( t\right) \right) }{%
\lambda _{v}\left( \psi ^{\rho }\left( t\right) \right) }\right)  \notag
\label{diff} \\
&=&\sum_{v\in \mathcal{V}}\lambda _{v}(\mu (t))\log \left( \frac{\lambda
_{v}(\mu (t))}{\lambda _{v}(\psi ^{\rho }(t))}\right) +\lambda _{v}(\psi
^{\rho }(t))-\lambda _{v}(\mu (t))  \notag \\
&\leq &C_{2}\left( \log \frac{1}{\rho }+1\right)
\end{eqnarray}%
for some $C_{2}<\infty $, where to obtain the last inequality we apply %
\eqref{ratiov}, use the Lipschitz continuity of $\lambda _{v}$ (Property \ref%
{prop-lambda}) and the estimate $||\psi ^{\rho }(t)-\mu (t)||\leq C_{d}$.

Using the convexity and nonnegativity of $L\left( x,\cdot \right) $ stated
in Proposition \ref{4.1}, along with relations (\ref{as}) and \eqref{diff},
one has for almost every $t\in \lbrack 0,1]$, 
\begin{eqnarray*}
L\left( \psi ^{\rho }\left( t\right) ,\dot{\psi}^{\rho }\left( t\right)
\right) &\leq &L\left( \psi ^{\rho }\left( t\right) ,\dot{\gamma}\left(
t\right) \right) +\rho L\left( \psi ^{\rho }\left( t\right) ,\dot{\mu}\left(
t\right) \right) \\
&\leq &\left( 1+c\left( \rho \right) \right) L\left( \gamma \left( t\right) ,%
\dot{\gamma}\left( t\right) \right) +c_{5}\left( \rho \right) ,
\end{eqnarray*}%
with $c_{5}\left( \rho \right) \doteq C_{2}\rho (\log 1/\rho +1)+c\left(
\rho \right) $. Integrating both sides of the last inequality over $\left[
0,1\right] $, we get 
\begin{equation*}
I\left( \psi ^{\rho }\right) \leq \left( 1+c\left( \rho \right) \right)
I\left( \gamma \right) +c_{5}\left( \rho \right) .
\end{equation*}%
Since $c(\rho )\rightarrow 0$ and $c_{5}(\rho )\rightarrow 0$ as $\rho
\rightarrow 0$, property (iii) holds with $\psi =\psi ^{\rho }$ for all $%
\rho >0$ sufficiently small.
\end{proof}

In view of Lemma \ref{6.3}, it suffices to establish the lower bound %
\eqref{red-lbound} for paths $\gamma \in AC\left( \left[ 0,1\right] :%
\mathcal{S}\right) $ with $I\left( \gamma \right) <\infty $ that satisfy the
additional condition that 
\begin{equation}
\text{there are }b_{0}>0\text{, }D<\infty \text{ such that }\gamma
_{i}(t)\geq b_{0}t^{D\text{ }}\text{ for all }i=1,\ldots ,d,t\in \lbrack
0,1].  \label{gm}
\end{equation}

\subsection{Analysis for short times}

We first state the main result of this subsection.

\begin{lemma}
\label{6.1.9} Suppose $\{\lambda _{v}(\cdot ),v\in {\mathcal{V}}\}$
satisfies Property \ref{prop-lambda} and Property \ref{prop-communicate},
and the sequence of deterministic initial conditions $\left\{ \mu ^{n}\left(
0\right) \right\} _{n\in \mathbb{N}}$ converges to $\mu _{0}\in \mathcal{S}$
as $n$ tends to infinity, and let $\varepsilon >0$ and $\delta >0$ be given.
Then there exists $\tau >0$ such that for any $\sigma >0$, there is $\eta
=\eta \left( \sigma \right) >0$ such that $\left\Vert \mu _{0}-\gamma \left(
0\right) \right\Vert \leq \eta $ implies%
\begin{equation*}
\liminf_{n\rightarrow \infty }\frac{1}{n}\log \mathbb{P}\left( \left\Vert
\mu ^{n}\left( \tau \right) -\gamma \left( \tau \right) \right\Vert \leq
\sigma ,\sup_{t\in \lbrack 0,\tau ]}\left\Vert \mu ^{n}(t)-\gamma \left(
t\right) \right\Vert \leq \delta \right) \geq -\frac{\varepsilon }{2}.
\end{equation*}
\end{lemma}

We first present the idea behind the proof. Given $\delta >0$, for $\tau >0$
sufficiently small we use excursion bounds for jump Markov processes (Lemma %
\ref{6.5} below) to establish a lower bound for the quantity 
\begin{equation*}
\mathbb{P}\left( \sup_{t\in \left[ 0,\tau \right] }\left\Vert \mu ^{n}\left(
t\right) -\gamma \left( t\right) \right\Vert <\delta \right) .
\end{equation*}%
The more difficult part is to obtain, for any $0<\sigma <\delta $, a lower
bound for $\mathbb{P}\left( \left\Vert \mu ^{n}(\tau )-\gamma (\tau
)\right\Vert <\sigma \right) $ that is uniform in $\mu ^{n}(0)$ as long as $%
\left\Vert \mu ^{n}(0)-\gamma (0)\right\Vert $ is sufficiently small. For
the latter, given $\varepsilon >0$, $\tau \in (0,1]$, for any $\sigma >0$,
consider the penalty function $g:\mathcal{S}\rightarrow \mathbb{R}$ defined
by 
\begin{equation}
g(x)=\left\{ 
\begin{array}{ll}
0 & \mbox{ if }\left\vert \left\vert x-\gamma (\tau )\right\vert \right\vert
<\sigma , \\ 
2\varepsilon  & \text{otherwise}.%
\end{array}%
\right.   \label{h}
\end{equation}%
We then have $g\in {\mathcal{M}}_{b}({\mathcal{S}})$ and 
\begin{equation}
\mathbb{P}(\left\Vert \mu ^{n}(\tau )-\gamma (\tau )\right\Vert <\sigma
)+e^{-2n\varepsilon }\geq \mathbb{E}\left[ \exp (-ng(\mu ^{n}(\tau )))\right]
.  \label{6.2.1}
\end{equation}%
To lower bound the right-hand side of \eqref{6.2.1}, we will use the
variational representation formula from Theorem \ref{3.2}: 
\begin{align}
& -\frac{1}{n}\log \mathbb{E}\left[ \exp (-ng(\mu ^{n}\left( \tau \right) ))%
\right]   \notag  \label{repform} \\
& \quad =\inf_{\bar{\alpha}\in \mathcal{A}_{b}^{\otimes \left\vert \mathcal{V%
}\right\vert }}\mathbb{\bar{E}}\left[ \sum_{v\in \mathcal{V}}\int_{0}^{\tau
}\lambda _{v}^{n}\left( \bar{\mu}^{n}(t)\right) \ell \left( \frac{\bar{\alpha%
}_{v}(t)}{\lambda _{v}^{n}\left( \bar{\mu}^{n}(t)\right) }\right) dt+g(\bar{%
\mu}^{n}\left( \tau \right) ):\bar{\mu}^{n}=\Lambda ^{n}\left( \bar{\alpha}%
,\mu ^{n}(0)\right) \right] ,
\end{align}%
with $\mathcal{A}_{b}^{\otimes \left\vert \mathcal{V}\right\vert }$ and $%
\Lambda ^{n}$ defined as in Definition \ref{ab} and \eqref{lam},
respectively. Thus, to prove Lemma \ref{6.1.9} we need to construct a
suitable controlled process $\bar{\mu}^{n}$ that has \textquotedblleft low
cost\textquotedblright\ and is sufficiently close to $\gamma (\tau )$ at
time $\tau $. We now provide the details of the proof.

\begin{proof}[Proof of Lemma \protect\ref{6.1.9}]
The idea is to argue that for large $n$ and small $\tau $, if $\mu ^{n}$
starts close to $\gamma (0)$, then it stays close to a communicating path
(see Definition \ref{def-compath}) that connects $\gamma \left( 0\right) $
to $\gamma \left( \tau \right) $, which lies in $\mathrm{int}({\mathcal{S}})$
due to \eqref{gm}. Since the jump rates (along the directions used to get
from $\gamma (0)$ to $\gamma (\tau )$) are bounded below away from zero
along such a path, one obtains a nice upper bound for the cost.
Specifically, by Property \ref{prop-communicate}, Definition \ref%
{def-compath} and Remark \ref{flex}, there exists a communicating path $\phi
\in C\left( \left[ 0,\tau \right] :\mathcal{S}\right) $, with $\phi \left(
0\right) =\gamma \left( 0\right) $ and $\phi \left( \tau \right) =\gamma
\left( \tau \right) $, and $F,U<\infty $, $\left\{ v_{m}\right\}
_{m=1}^{F}\subset \mathcal{V}$ and $0=t_{0}<t_{1}<\cdots <t_{F}=\tau $, such
that 
\begin{equation*}
\frac{d}{dt}\phi \left( t\right) =\sum_{v\in \mathcal{V}}\bar{\alpha}_{v}(t)v%
\text{, \ a.e. }t\in \left[ 0,\tau \right] ,
\end{equation*}%
where 
\begin{equation}
\bar{\alpha}_{v}(t)=\left\{ 
\begin{array}{ll}
U\mathbb{I}_{[t_{m-1},t_{m})}\left( t\right) & \text{ if }v=v_{m},m=1,...,F,
\\ 
0 & \text{ if }v\notin \left\{ v_{m}\right\} _{m=1}^{F}\text{.}%
\end{array}%
\right.  \label{baralpha}
\end{equation}%
Also, by Definition \ref{def-compath}, there exist $p,D<\infty $ and $c>0$,
such that 
\begin{equation}
\lambda _{v_{m}}\left( \phi \left( t\right) \right) \geq c\left(
\min_{i=1,\ldots ,d}\gamma _{i}\left( \tau \right) \right) ^{p}\geq
c_{0}\tau ^{Dp}\text{,\ if }t\in \left[ t_{m-1},t_{m}\right] ,\text{ }%
m=1,...,F\text{,}  \label{D**}
\end{equation}%
where the second inequality uses (\ref{gm}) and $c_{0}=cb_{0}>0$.

Now define $\bar{\mu}^{n}=\Lambda ^{n}\left( \bar{\alpha},\mu ^{n}\left(
0\right) \right) $, where $\Lambda ^{n}$ is as defined in (\ref{lam}).
Property \ref{prop-lambda} and the LLN for Poisson random measures (see
Section \ref{subs-pflln}) imply that $\left\{ \bar{\mu}^{n}\right\} _{n\in 
\mathbb{N}}$ converges uniformly on $[0,\tau ]$ in probability to $\bar{\mu}$%
, where $\bar{\mu}\left( 0\right) =\mu _{0}$, and 
\begin{equation}
\frac{d}{dt}\bar{\mu}\left( t\right) =\sum_{v\in \mathcal{V}}\bar{\alpha}%
_{v}(t)v,\text{ a.e. }t\in \left[ 0,\tau \right] .  \label{mu-}
\end{equation}%
Since the trajectories $\phi $ and $\bar{\mu}$ satisfy the same
(state-independent) ODE, we have $\left\Vert \bar{\mu}\left( t\right) -\phi
\left( t\right) \right\Vert =\left\Vert \mu _{0}-\gamma \left( 0\right)
\right\Vert $. Thus, by the Lipschitz continuity of $\lambda_v$ (Property %
\ref{prop-lambda}) and (\ref{D**}), for any fixed $\tau $, there exists some 
$\eta _{0}\left( \tau \right) >0$, such that for any $\eta \leq \eta
_{0}\left( \tau \right) $, if $\left\Vert \mu _{0}-\gamma \left( 0\right)
\right\Vert \leq \eta $ then 
\begin{equation}
\lambda _{v_{m}}\left( \bar{\mu}(t)\right) \geq \frac{c_{0}}{2}\tau ^{Dp},%
\text{ for }t\in \left[ t_{m-1},t_{m}\right] ,\text{ }m=1,...,F\text{.}
\label{D***}
\end{equation}

We now bound the cost for the sequence of jump processes $\{\bar{\mu}%
^{n}\}_{n\in \mathbb{N}}$ by making use of the bound \eqref{D***} on its law
of large numbers limit. Given the form of $\bar{\alpha}_v$ and $\ell $ in %
\eqref{baralpha} and \eqref{l}, respectively, we have 
\begin{align}
& \mathbb{\bar{E}}\left[ \sum_{v\in \mathcal{V}}\int_{0}^{\tau }\lambda
_{v}^{n}\left( \bar{\mu}^{n}(t)\right) \ell \left( \frac{\bar{\alpha}_{v}(t)%
}{\lambda _{v}^{n}\left( \bar{\mu}^{n}(t)\right) }\right) dt+g(\bar{\mu}%
^{n}\left( \tau \right) )\right]  \notag  \label{control-bd} \\
& =\mathbb{\bar{E}}\left[ \sum_{m=1}^{F}\int_{t_{m-1}}^{t_{m}}\left( U\log
\left( \frac{U}{\lambda _{v_{m}}^{n}\left( \bar{\mu}^{n}(t)\right) }\right)
-U+\sum_{v\in {\mathcal{V}}}\lambda _{v}^{n}\left( \bar{\mu}^{n}(t)\right)
\right) dt+g(\bar{\mu}^{n}\left( \tau \right) )\right] .
\end{align}%
Now fix $\tau >0$ and $\eta <\min \left\{ \eta _{0}\left( \tau \right)
,\sigma /2\right\} $. Then by (\ref{D***}), if $\left\Vert \mu _{0}-\gamma
\left( 0\right) \right\Vert \leq \eta $ then for each $m=1,\ldots ,F$, on
the interval $[t_{m-1},t_{m}]$, $\lambda _{v_{m}}(\bar{\mu}\left( t\right) )$
is uniformly bounded below away from zero. Since $\bar{\mu}^{n}$ converges
in probability to $\bar{\mu}$, uniformly on $[0,\tau ]$, and Property \ref%
{prop-lambda} holds, this implies that for each $m=1,\ldots ,F$ and $t\in
\lbrack t_{m-1},t_{m}]$, $\log \left( \lambda _{v_{m}}^{n}\left( \bar{\mu}%
^{n}\left( t\right) \right) \right) $ converges in probability to $\log
\left( \lambda _{v_{m}}\left( \bar{\mu}\left( t\right) \right) \right) $
uniformly for $t\in \lbrack t_{m-1},t_{m}]$. Thus, taking the limit superior
as $n\rightarrow \infty $ in \eqref{control-bd}, by the dominated
convergence theorem and the upper semicontinuity of $g$ defined in (\ref{h}%
), we obtain 
\begin{align*}
& \limsup_{n\rightarrow \infty }\mathbb{\bar{E}}\left[ \sum_{v\in \mathcal{V}%
}\int_{0}^{\tau }\lambda _{v}^{n}\left( \bar{\mu}^{n}(t)\right) \ell \left( 
\frac{\bar{\alpha}_{v}(t)}{\lambda _{v}^{n}\left( \bar{\mu}^{n}(t)\right) }%
\right) dt+g(\bar{\mu}^{n}\left( \tau \right) )\right] \\
& \quad \leq \mathbb{\bar{E}}\left[ \sum_{m=1}^{F}\int_{t_{m-1}}^{t_{m}}%
\left( U\log \left( \frac{U}{\lambda _{v_{m}}\left( \bar{\mu}(t)\right) }%
\right) -U+\sum_{v\in {\mathcal{V}}}\lambda _{v}\left( \bar{\mu}(t)\right)
\right) dt+g(\bar{\mu}\left( \tau \right) )\right] \\
& \quad \leq \tau \left( U\log U+U\log \left( \frac{c_{0}}{2}\tau
^{Dp}\right) +|{\mathcal{V}}|R\right) ,
\end{align*}%
where the last inequality uses the lower bound in \eqref{D***}, the upper
bound in \eqref{m}, the identity $t_F = \tau$ and the fact that $g(\bar{\mu}%
(\tau ))=0$ because $||\bar{\mu}(\tau )-\gamma (\tau )||=||\bar{\mu}%
_{0}-\gamma (0)||\leq \eta <\sigma $. Choose $\tau >0$ sufficiently small
such that the last expression is less than $\varepsilon /4$. Observing that
the control $\bar{\alpha}$ in \eqref{baralpha} is a deterministic process
that is uniformly bounded, and hence, lies in ${\mathcal{A}}_{b}^{\otimes {%
\mathcal{V}}}$, 
we can combine the last display with the representation formula %
\eqref{repform}: for all sufficiently large $n$ and sufficiently small $\eta 
$, $\left\Vert \mu _{0}-\gamma (0)\right\Vert <\eta $ implies 
\begin{equation*}
-\frac{1}{n}\log \mathbb{E}\left[ \exp (-ng(\mu ^{n}(\tau )))\right] \leq 
\frac{\varepsilon}{2}.
\end{equation*}%
When combined with (\ref{6.2.1}), this gives the lower bound 
\begin{equation}
\mathbb{P}(\left\Vert \mu ^{n}(\tau )-\gamma (\tau )\right\Vert <\sigma
)\geq e^{-n\varepsilon /2}-e^{-2n\varepsilon }.  \label{tauest}
\end{equation}

We will conclude the argument by establishing an upper bound on the
probability of $\mu ^{n}$ having a large excursion during the interval $%
\left[ 0,\tau \right]$. Given $\varepsilon >0$, applying a standard
martingale inequality (stated as Lemma \ref{6.5} below), for sufficiently
small $\tau$ we have%
\begin{equation*}
\mathbb{P}\left( \sup_{t \in [0,\tau] }\left\Vert \mu ^{n}(t)-\mu
^{n}(0)\right\Vert >\frac{\delta }{3}\right) \leq 2d\exp \left(
-n\varepsilon \right) .
\end{equation*}%
On the other hand, since $\gamma$ is continuous, by taking $\tau $ smaller
if necessary we can guarantee that $\sup_{t \in \left[ 0,\tau \right]
}\left\Vert \gamma \left( t\right) -\gamma \left( 0\right) \right\Vert \leq
\delta /3$. It follows that for $\eta \in \left[ 0,\frac{\delta }{3}\right] $%
, 
\begin{equation*}
\mathbb{P}\left( \sup_{t \in [0, \tau] }\left\Vert \mu ^{n}(t)-\gamma \left(
t\right) \right\Vert >\delta \right) \leq \mathbb{P}\left( \sup_{ t \in [0,
\tau]}\left\Vert \mu ^{n}(t)-\mu ^{n}(0)\right\Vert >\frac{\delta }{3}%
\right) \leq 2d\exp \left( -n\varepsilon \right).
\end{equation*}%
Combining this with the estimate (\ref{tauest}) we arrive at the desired
conclusion.
\end{proof}

The following lemma is an adaptation of Lemma 2.3 in \cite{DEW}. The lemma
follows from bounds for certain exponential martingales.

\begin{lemma}
\label{6.5}Let $\bar{C}_{1}=\max_{v\in \mathcal{V}}\left\Vert v\right\Vert ,%
\bar{C}_{2}=R\left\vert \mathcal{V}\right\vert \bar{C}_{1}$, and for $%
\varrho >\bar{C}_{2}$ define $\bar{\ell}\left( \varrho \right) \dot{=}%
\varrho \left( \log \left( \varrho /\bar{C}_{2}\right) -1\right) /\bar{C}%
_{1} $. Then $\bar{\ell}\left( \varrho \right) /\varrho \rightarrow \infty $
as $\varrho \rightarrow \infty $, and given any $\delta >0$, for all $\tau
\leq \delta /2\sqrt{d}\bar{C}_{2}$ 
\begin{equation*}
\mathbb{P}\left( \sup_{t\in \lbrack 0,\tau ]}\left\Vert \mu ^{n}\left(
t\right) -\mu ^{n}\left( 0\right) \right\Vert \geq \delta \right) \leq
2d\exp \left( -\tau n\bar{\ell}\left( \frac{\delta }{2\sqrt{d}\tau }\right)
\right) .
\end{equation*}
\end{lemma}

\subsection{Analysis for $t \in [\protect\tau, 1]$}

\label{subs-lb2}

As shown in Section \ref{subs-pert}, to establish the large deviation lower
bound, it suffices to establish the estimate \eqref{red-lbound} for $\gamma
\in AC\left( \left[ 0,1\right] :\mathbb{\mathcal{S}}\right) $ that satisfies 
$I(\gamma )<\infty $ and the bound (\ref{gm}). So for any $\tau >0$ there
exists $\xi >0$ such that $\gamma \left( t\right) $ lies in $\mathbb{%
\mathcal{S}}^{\xi }$ for all $t\in \lbrack \tau ,1]$. Therefore, we now fix $%
\tau >0$ and $\xi >0$ and consider large deviations of $\mu ^{n}$ in $[\tau
,1]$ from a path $\gamma \in AC\left( \left[ \tau ,1\right] :\mathbb{%
\mathcal{S}^{\xi }}\right) $.

For $y\in \mathbb{\mathcal{S}}$ and $r>0$, let $B\left( y,r\right) $ denote
the open Euclidean ball centered at $y$ with radius $r$. For $\psi \in
AC\left( \left[ \tau ,1\right] :\mathbb{\mathcal{S}}\right) $ with $\psi
(0)=y$, we denote 
\begin{equation*}
I^{y}\left( \psi \right) \doteq \int_{\tau }^{1}L\left( \psi \left( s\right)
,\dot{\psi}\left( s\right) \right) ds,
\end{equation*}%
to emphasize the dependence on $y$ (though we omit the dependence on $\tau $%
). Given $y_{n}\in \mathbb{\mathcal{S}}$, let $\mathbb{P}_{y_{n}}$ and $%
\mathbb{E}_{y_{n}}$ denote the probability and expectation, respectively,
conditioned on $\mu ^{n}\left( \tau \right) =y_{n}$. Define the mapping $%
\Lambda _{\tau }^{n}:\mathcal{A}_{b}^{\otimes \left\vert \mathcal{V}%
\right\vert }\times \mathbb{\mathcal{S}}\rightarrow D\left( \left[ \tau ,1%
\right] :\tilde{\Delta}^{d-1}\right) $ by%
\begin{equation*}
\Lambda _{\tau }^{n}\left( \bar{\alpha},\rho \right) \left( t\right) =\rho
+\sum_{v\in \mathcal{V}}v\int_{[\tau ,t]}\int_{\mathcal{Y}}\mathbb{I}_{\left[
0,\bar{\alpha}_{v}(s-)\right] }(x)\frac{1}{n}N_{v}^{n}(dsdx),
\end{equation*}%
for $\bar{\alpha}\in \mathcal{A}_{b}^{\otimes \left\vert \mathcal{V}%
\right\vert }$ and $\rho \in \mathbb{\mathcal{S}}$. We will prove the
following uniform Laplace principle lower bound for $\left\{ \mu ^{n}\left(
\cdot \right) \right\} _{n\in \mathbb{N}}$ on $\left[ \tau ,1\right] $,
where we restrict to Lipschitz continuous test functions. By \cite[Corollary
1.2.5]{DE}, this implies the corresponding large deviation lower bound.

\begin{proposition}
\label{6.6} Suppose the assumptions of Lemma \ref{6.1.9} hold. Fix $\tau \in
(0,1)$. Let $\xi >0$ and $\gamma \in AC([\tau ,1]:\mathbb{\mathcal{S}}^{\xi
})$ be such that $\gamma (\tau )=y$ and $I^{y}(\gamma )<\infty $. Then there
exists $\sigma >0$ such that for any bounded and Lipschitz continuous
functional $F$ on $D\left( \left[ \tau ,1\right] :\mathbb{\mathcal{S}}%
\right) $, 
\begin{equation}
\liminf_{n\rightarrow \infty }\inf_{y_{n}\in B\left( y,\sigma \right)
}\left( \frac{1}{n}\log \mathbb{E}_{y_{n}}\left[ \exp (-nF(\mu ^{n}))\right]
-G\left( y_{n},F\right) \right) \geq 0,  \label{de}
\end{equation}%
where 
\begin{equation}
G\left( y,F\right) \doteq -\inf_{\psi \in AC\left( \left[ \tau ,1\right] :%
\mathbb{\mathcal{S}}^{\xi }\right) }\left[ I^{y}(\psi )+F(\psi )\right] .
\label{G}
\end{equation}%
In particular, this implies the following uniform (with respect to initial
conditions) large deviation lower bound: for any $\varepsilon >0$ and $%
\delta >0$, there exists $\sigma >0$ such that for any sequence $%
\{y_{n}\}_{n\in \mathbb{N}}\subset B\left( y,\sigma \right) $, 
\begin{equation}
\liminf_{n\rightarrow \infty }\frac{1}{n}\log \mathbb{P}_{y_{n}}\left(
\sup_{t\in \left[ \tau ,1\right] }\left\Vert \mu ^{n}\left( t\right) -\gamma
\left( t\right) \right\Vert <\delta \right) \geq -I^{y}\left( \gamma \right)
-\frac{\varepsilon }{2}.  \label{lbound-imp}
\end{equation}
\end{proposition}

The proof of Proposition \ref{6.6} relies on the following approximation
argument. Fix $y\in \mathbb{\mathcal{S}}^{\xi }$ and a bounded and Lipschitz
continuous functional $F$ on $D\left( \left[ \tau ,1\right] :\mathbb{%
\mathcal{S}}\right) $. By Proposition 1.2.7 of \cite{DE}, to prove (\ref{de}%
), it suffices to show that for any sequence $\left\{ y_{n}\right\} _{n\in 
\mathbb{N}}$ such that $\left\Vert y_{n}-y\right\Vert \rightarrow 0$ as $%
n\rightarrow \infty $,%
\begin{equation}
\liminf_{n\rightarrow \infty }\frac{1}{n}\log \mathbb{E}_{y_{n}}\left[ \exp
(-nF(\mu ^{n}))\right] \geq G\left( y,F\right) .  \label{6.6.1}
\end{equation}

It suffices to show that for any $\varepsilon >0$ and $\gamma _{\varepsilon
}\in AC\left( \left[ \tau ,1\right] :\mathbb{\mathcal{S}}^{\xi }\right) $
such that $-\left( I^{y}(\gamma _{\varepsilon })+F(\gamma _{\varepsilon
})\right) \geq G\left( y,F\right) -\varepsilon $, we have 
\begin{equation*}
\liminf_{n\rightarrow \infty }\frac{1}{n}\log \mathbb{E}_{y_{n}}\left[ \exp
(-nF(\mu ^{n}))\right] \geq -\left( I^{y}(\gamma _{\varepsilon })+F(\gamma
_{\varepsilon })\right) ,
\end{equation*}%
or, equivalently, 
\begin{equation}
\limsup_{n\rightarrow \infty }-\frac{1}{n}\log \mathbb{E}_{y_{n}}\left[ \exp
(-nF(\mu ^{n}))\right] \leq I^{y}(\gamma _{\varepsilon })+F(\gamma
_{\varepsilon }).  \label{Eqn:withge}
\end{equation}%
Fix $\varepsilon >0$ and denote $\gamma _{\varepsilon }$ simply by $\gamma $%
. We now approximate $\gamma $ by a piecewise linear path. Let $\Delta =%
\frac{1-\tau }{\mathbb{J}}$ for some $\mathbb{J}\in \mathbb{N}$. For $%
j=0,1,...,\mathbb{J}-1$ let $a_{j}^{\Delta }=\frac{1}{\Delta }\int_{\tau
+j\Delta }^{\tau +\left( j+1\right) \Delta }\dot{\gamma}\left( s\right) ds$.
Define 
\begin{equation}
\dot{\gamma}^{\Delta }\left( t\right) =a_{j}^{\Delta }\text{ \ \ \ if }t\in
(\tau +j\Delta ,\tau +\left( j+1\right) \Delta ),\text{ }j=0,...,\mathbb{J}%
-1,  \notag
\end{equation}%
and%
\begin{equation}
\gamma ^{\Delta }\left( t\right) =y+\int_{\tau }^{t}\dot{\gamma}^{\Delta
}\left( s\right) ds\text{ \ for }t\in \left[ \tau ,1\right] .
\label{gammadelta}
\end{equation}%
Then $\gamma ^{\Delta }$ is the piecewise linear interpolation of the
continuous process $\gamma $ with mesh size $\Delta $. Note that for any $%
v\in \mathcal{V}$, $\lambda _{v}(\gamma ^{\Delta }\left( \cdot \right) )$ is
continuous and uniformly bounded away from zero on $t\in \left[ \tau ,1%
\right] $. The proof of (\ref{de}) thus relies on the following standard
approximation result (we refer to Lemma 65 in Section 3.6.3 of \cite{WW} for
a complete proof).

\begin{lemma}
\label{6.7} Suppose $\{\lambda _{v}(\cdot ),v\in {\mathcal{V}}\}$ satisfies
Property \ref{prop-lambda} and Property \ref{prop-communicate}. Let $\tau
,\xi ,y$ and $\gamma $ be as in Proposition \ref{6.6} and define $\gamma
^{\Delta }$ as in (\ref{gammadelta}). Then for any $\varepsilon >0$, there
exists $\Delta \left( \varepsilon \right) >0$, such that for any $\Delta
<\Delta \left( \varepsilon \right) $, and a.e. $t\in \left[ \tau ,1\right] $%
, there exists a piecewise constant vector $q^{\Delta }\left( t\right) \in
\lbrack 0,\infty )^{\left\vert {\mathcal{V}}\right\vert }$ such that $%
\sum_{v\in \mathcal{V}}vq_{v}^{\Delta }\left( t\right) =\dot{\gamma}^{\Delta
}\left( t\right) $, and 
\begin{equation}
\int_{\tau }^{1}\sum_{v\in \mathcal{V}}\lambda _{v}(\gamma ^{\Delta }\left(
t\right) )\ell \left( \frac{q_{v}^{\Delta }\left( t\right) }{\lambda
_{v}(\gamma ^{\Delta }\left( t\right) )}\right) dt\leq I^{y}\left( \gamma
\right) +\varepsilon .  \label{*}
\end{equation}
\end{lemma}

We now complete the proof of Proposition \ref{6.6}. By Lemma \ref{6.7}, for
any $\varepsilon >0$, there exists $\Delta $ sufficiently small and a
collection of piecewise constant functions $\left\{ q_{v}^{\Delta }\left(
\cdot \right) \right\} _{v\in \mathcal{V}}$ on $\left[ \tau ,1\right] $ that
satisfy (\ref{*}). It follows directly from the LLN for Poisson random
measures that as $n\rightarrow \infty $, $\bar{\mu}^{n}=\Lambda _{\tau
}^{n}\left( q^{\Delta },y_{n}\right) $ converges uniformly on $\left[ \tau ,1%
\right] $ in probability to $\gamma ^{\Delta }$. Therefore, by the uniform
continuity of $\lambda _{v}\left( \cdot \right) \ell \left( q_{v}^{\Delta
}/\lambda _{v}\left( \cdot \right) \right) $ on $\mathbb{\mathcal{S}}^{\xi }$
and the uniform convergence of $\lambda _{v}^{n}\left( \cdot \right) $ to $%
\lambda _{v}\left( \cdot \right) $ on ${\mathcal{S}}$ by Property \ref%
{prop-lambda}, $\lambda _{v}^{n}\left( \bar{\mu}^{n}(\cdot )\right) \ell
\left( q_{v}^{\Delta }\left( \cdot \right) /\lambda _{v}^{n}\left( \bar{\mu}%
^{n}(\cdot )\right) \right) $ converges uniformly on $\left[ \tau ,1\right] $
in probability to $\lambda _{v}(\gamma ^{\Delta }\left( \cdot \right) )\ell
\left( q_{v}^{\Delta }\left( \cdot \right) /\lambda _{v}(\gamma ^{\Delta
}\left( \cdot \right) )\right) $. Combining the variational representation
formula (Theorem \ref{3.2}), (\ref{*}), and the dominated convergence
theorem, for any Lipschitz continuous functional $F$ on $D\left( \left[ \tau
,1\right] :\mathbb{\mathcal{S}}\right) $, we have

\begin{align*}
& \limsup_{n\rightarrow \infty }-\frac{1}{n}\log \mathbb{E}_{y_{n}}\left[
\exp (-nF(\mu ^{n}))\right] \\
& \quad =\limsup_{n\rightarrow \infty }\inf_{\bar{\alpha}\in \mathcal{A}%
_{b}^{\otimes \left\vert \mathcal{V}\right\vert }}\mathbb{\bar{E}}_{y_{n}}%
\left[ \sum_{v\in \mathcal{V}}\int_{\tau }^{1}\lambda _{v}^{n}\left( \bar{\mu%
}^{n}(t)\right) \ell \left( \frac{\bar{\alpha}_{v}(t)}{\lambda
_{v}^{n}\left( \bar{\mu}^{n}(t)\right) }\right) dt+F(\bar{\mu}^{n}):\bar{\mu}%
^{n}=\Lambda _{\tau }^{n}\left( \bar{\alpha},y_{n}\right) \right] \\
& \quad \leq \limsup_{n\rightarrow \infty }\mathbb{\bar{E}}_{y_{n}}\left[
\int_{\tau }^{1}\sum_{v\in \mathcal{V}}\lambda _{v}^{n}\left( \bar{\mu}%
^{n}(t)\right) \ell \left( \frac{q_{v}^{\Delta }(t)}{\lambda _{v}^{n}\left( 
\bar{\mu}^{n}(t)\right) }\right) dt+F(\bar{\mu}^{n}):\bar{\mu}^{n}=\Lambda
_{\tau }^{n}\left( q^{\Delta },y_{n}\right) \right] \\
& \quad =\int_{\tau }^{1}\sum_{v\in \mathcal{V}}\lambda _{v}(\gamma ^{\Delta
}\left( t\right) )\ell \left( \frac{q_{v}^{\Delta }\left( t\right) }{\lambda
_{v}(\gamma ^{\Delta }\left( t\right) )}\right) dt+F(\gamma ^{\Delta }) \\
& \quad \leq I^{y}\left( \gamma \right) +\varepsilon +F(\gamma ^{\Delta }).
\end{align*}%
Letting $\Delta \rightarrow 0$ gives the upper bound $I^{y}\left( \gamma
\right) +F\left( \gamma \right) +\varepsilon $, and since $\varepsilon >0$
is arbitrary this gives (\ref{Eqn:withge}).

We now have all the ingredients to complete the proof of the LDP lower bound.

\begin{proof}[Proof of the lower bound \eqref{lowbd} of Theorem \protect\ref%
{th-ldips}]
We start by showing that the assumptions on the transition rates $\{\Gamma_{%
\mathbf{i} \mathbf{j}}^k (\cdot), (\mathbf{i} \mathbf{j}) \in {\mathcal{J}}%
^k, k = 1, \ldots, K\}$ imply all the required conditions on the jump rates $%
\{\lambda_v(\cdot), v \in {\mathcal{V}}\}$ that are necessary to apply the
results in Section \ref{sec-pflower}. Indeed, Property \ref{prop-lambda}
follows from Assumption \ref{intsys}, Lemma \ref{lem-estimates} shows that
all four properties of the lemma follow from Assumption \ref{ue} and
Assumption \ref{ass-simjumps} and finally, since Assumption \ref{ass-kerg}
also holds, Proposition \ref{ver1} shows that the jump rates also satisfy
Property \ref{prop-communicate}. From the discussion at the beginning of
Section \ref{sec-pflower} and Lemma \ref{6.3} of Section \ref{subs-pert}, it
follows that to prove the LDP lower bound \eqref{lowbd} it suffices to
establish \eqref{red-lbound} for $\gamma \in AC([0,1]:\mathbb{\mathcal{S}})$
that satisfies the lower bound \eqref{gm}. The latter lower bound guarantees
that, even if $\gamma $ starts on the boundary of $\mathbb{\mathcal{S}}$,
for any $\tau >0$ it lies a strictly positive distance from that boundary,
and thus after $\tau$, Proposition \ref{6.6} can be applied to get a uniform
lower bound for initial conditions close to $\gamma (\tau )$. Due to the
Markov property, the proof is then completed by observing that Lemma \ref%
{6.1.9} shows that, with an error that is vanishingly small as $\tau
\rightarrow 0$, $\mu(\tau)$ can be brought into the required sufficiently
small neighborhood of $\gamma(\tau)$, while staying close to $\gamma$ on $%
[0,\tau]$.
\end{proof}

\begin{remark}
\label{rem-ldp} \emph{From the proof of the upper bound in Section \ref%
{sec-pfupper} and the proof of the lower bound above, it is clear that the
conclusions of Theorem \ref{th-ldips} in fact holds for a more general class
of jump Markov processes. Specifically, it holds for any sequence $\{\mu
^{n}\}_{n\in \mathbb{N}},$ of jump Markov processes on ${\mathcal{S}}$ with
generators of the form \eqref{Lnl}, for which the associated sequence of
jump rates $\{\lambda _{v}^{n}(\cdot ),v\in {\mathcal{V}}\}_{n\in \mathbb{N}%
} $ satisfies Property \ref{prop-lambda}, Property \ref{prop-communicate},
and the properties stated in Lemma \ref{lem-estimates}. Moreover, the only
place where Assumption \ref{ass-simjumps} is used is in the proof of
property (4) of Lemma \ref{lem-estimates}, which in turn is only used in the
proof of Property \ref{cond-lln} of the LLN trajectory. Thus, to extend the
results to situations where Assumption \ref{ass-simjumps} fails to hold, it
suffices to directly verify Property \ref{cond-lln}. }
\end{remark}

\section{The Locally Uniform LDP}

\label{sec-locunif}

We now turn to the proof of Theorem \ref{th-localldips}. We assume
throughout this section that the conditions (and conclusions) of Theorem \ref%
{th-ldips} are satisfied, and below, only specify additional conditions that
are imposed. Fix $t\in \left[ 0,1\right] $. As shown in Corollary \ref{2.3},
one can express the rate function $J_{t}$ of $\left\{ \mu ^{n}\left(
t\right) \right\} _{n\in \mathbb{N}}$ in terms of a variational problem. In
what follows, fix $x\in \mathcal{S}$ and $\left\{ x_{n}\right\} _{n\in 
\mathbb{N}}$ such that $x_{n}\in \mathcal{S}_{n}$ and $\left\Vert
x_{n}-x\right\Vert \rightarrow 0$ as $n\rightarrow \infty $.

\subsection{Proof of the locally uniform LDP upper bound}

\label{subs-luldpupper}

Given any $\varepsilon >0$, recall that $B\left( x,\varepsilon \right) $
denotes the open Euclidean ball centered at $x$ with radius $\varepsilon $,
and that $\bar{B}(x, \varepsilon)$ denotes its closure. For $n $
sufficiently large such that $x_{n}\in \bar{B}\left( x,\varepsilon \right) $%
, by the LDP upper bound stated in Corollary \ref{2.3}, 
\begin{eqnarray*}
\limsup_{n\rightarrow \infty }\frac{1}{n}\log \mathbb{P}\left( \mu
^{n}\left( t\right) =x_{n}\right) &\leq &\limsup_{n\rightarrow \infty }\frac{%
1}{n}\log \mathbb{P}\left( \mu ^{n}\left( t\right) \in \bar{B}\left(
x,\varepsilon \right) \right) \\
&\leq &-\bar{J}_{t}^{\varepsilon }\left( \mu _{0},x\right) ,
\end{eqnarray*}%
where we define%
\begin{equation}
\bar{J}_{t}^{\varepsilon }\left( \mu _{0},x\right) \doteq \inf \left\{
I_{t}\left( \gamma \right) :\gamma \in D\left( \left[ 0,1\right] :\mathcal{S}%
\right) ,\gamma \left( 0\right) =\mu _{0},\gamma \left( t\right) \in \bar{B}%
\left( x,\varepsilon \right) \right\} .  \label{Jeps}
\end{equation}

To prove the locally uniform LDP upper bound, it suffices to show that 
\begin{equation}
\liminf_{\varepsilon \rightarrow 0}\bar{J}_{t}^{\varepsilon }\left( \mu
_{0},x\right) \geq J_{t}\left( \mu _{0},x\right) .  \label{liminf}
\end{equation}

\begin{lemma}
\label{7.1}Assume Property \ref{prop-dcommunicate}.i) holds. Then there
exists a function $c:[0,\infty )\rightarrow \lbrack 0,\infty )$ and $%
C<\infty $ such that

i). $c\left( \varepsilon \right) \rightarrow 0$ as $\varepsilon \rightarrow
0 $, and

ii). given any $\varepsilon >0$ and $x,y\in \mathcal{S}$ such that $%
\left\Vert x-y\right\Vert <\varepsilon $, one can construct a path $\gamma
\in AC\left( \left[ 0,\varepsilon \right] :\mathcal{S}\right) $ such that $%
\gamma \left( 0\right) =x$, $\gamma \left( \varepsilon \right) =y$, $%
\sup_{s\in \lbrack 0,\varepsilon ]}||\gamma (s)-x||\leq C\varepsilon $ and $%
J_{\varepsilon }\left( x,y\right) \leq I_{\varepsilon }\left( \gamma \right)
\leq c\left( \varepsilon \right) $, where $I_{\varepsilon },J_{\varepsilon }$
are defined in (\ref{I}), (\ref{J}), respectively.
\end{lemma}

Before proving Lemma \ref{7.1}, we first describe how it can be used to
prove Lemma \ref{uc}. By Lemma \ref{move}, Assumption \ref{ue} and
Assumption \ref{g1} (which are the conditions of Lemma \ref{uc}) imply
Property \ref{prop-dcommunicate}.i), thus the condition of Lemma \ref{7.1}
is satisfied.

\begin{proof}[Proof of Lemma \protect\ref{uc}]
For any $\varepsilon >0$, take $t\in (0,\infty )$ and $\gamma \in AC\left( %
\left[ 0,t\right] :\mathcal{S}\right) $ such that $\gamma \left( 0\right) =x$%
, $\gamma \left( t\right) =y$, and $I_{t}\left( \gamma \right) \leq V\left(
x,y\right) +\varepsilon /2$. Given $\delta >0$, and any $y^{\delta }\in 
\mathcal{S}$ such that $\left\Vert y^{\delta }-y\right\Vert \leq \delta $,
by Lemma \ref{7.1}, there exists a path $\nu \in AC\left( \left[ 0,\delta %
\right] :\mathcal{S}\right) $ with $\nu \left( 0\right) =y$, $\nu \left(
\delta \right) =y^{\delta }$ with $I_{\delta }\left( \nu \right) \leq
c\left( \delta \right) $. Let $\bar{\gamma}$ be the concatenation of $\gamma 
$ and $\nu $. Then we have%
\begin{equation*}
V\left( x,y^{\delta }\right) \leq I_{t+\delta }\left( \bar{\gamma}\right)
=I_{t}\left( \gamma \right) +I_{\delta }\left( \nu \right) \leq V\left(
x,y\right) +\varepsilon /2+c\left( \delta \right) .
\end{equation*}%
It suffices to choose $\delta $ such that $c\left( \delta \right) \leq
\varepsilon /2$. The reverse inequality and the joint continuity with
respect to both variables can be proved using similar arguments.
\end{proof}

\medskip A similar construction leads to the proof of the following lemma,
which is used in the proof of Corollary \ref{cor-randin}. Notice Corollary %
\ref{cor-randin} also assumes Assumption \ref{ue} and Assumption \ref{g1},
which imply Property \ref{prop-dcommunicate}.i).

\begin{lemma}
\label{lem-ucont} Assume Property \ref{prop-dcommunicate}.i) holds. Given a
bounded and continuous function $h$, the function 
\begin{equation*}
U\left( y\right) \doteq \inf \left\{ I\left( \gamma \right) +h\left( \gamma
\right) :\gamma \in D\left( [0,1]:\mathcal{S}\right) ,\gamma \left( 0\right)
=y\right\}
\end{equation*}%
is continuous on ${\mathcal{S}}$.
\end{lemma}

\begin{proof}
Fix $\varepsilon >0$. Take $\gamma \in AC\left( [0,1]:\mathcal{S}\right) $
such that $I\left( \gamma \right) +h\left( \gamma \right) <U\left( y\right)
+\varepsilon /3$. Given $\delta >0$ such that $c\left( \delta \right) \leq
\varepsilon /3$, and any $y^{\delta }\in \mathcal{S}$ such that $\left\Vert
y^{\delta }-y\right\Vert \leq \delta $, by Lemma \ref{7.1}, there exists a
path $\nu \in AC\left( \left[ 0,\delta \right] :\mathcal{S}\right) $ with $%
\nu \left( 0\right) =y$, $\nu \left( \delta \right) =y^{\delta }$ such that $%
I_{\delta }\left( \nu \right) \leq c\left( \delta \right) $, and $\sup_{s\in %
\left[ 0,\delta \right] }\left\Vert \nu \left( s\right) -y\right\Vert \leq
C\delta $ for some $C<\infty $. We now rescale $\gamma $ to obtain a new
path $\gamma ^{\delta }$: for $c=\left( 1-\delta \right) ^{-1}$, define $%
\gamma ^{\delta }\in AC\left( \left[ 0,1-\delta \right] :\mathcal{S}\right) $
by $\gamma ^{\delta }\left( s\right) \doteq \gamma \left( cs\right) $. By
Proposition \ref{4.8}, we can take $\delta $ smaller if necessary such that $%
I_{1-\delta }\left( \gamma ^{\delta }\right) \leq I\left( \gamma \right)
+\varepsilon /3$. Let $\bar{\gamma}$ be the concatenation of $\nu $ and $%
\gamma ^{\delta }$. Then $\left\Vert \gamma -\bar{\gamma}\right\Vert
_{\infty }\rightarrow 0$ as $\delta \rightarrow 0$. Therefore, we have 
\begin{equation*}
U\left( y^{\delta }\right) \leq h(\bar{\gamma})+I\left( \bar{\gamma}\right)
\leq U\left( y\right) +\varepsilon /3+c\left( \delta \right) +h\left( \bar{%
\gamma}\right) -h\left( \gamma \right) .
\end{equation*}%
The other inequality is proved in the same way. Therefore,%
\begin{equation*}
\left\vert U\left( y^{\delta }\right) -U\left( y\right) \right\vert \leq
2\varepsilon /3+\left\vert h\left( \bar{\gamma}\right) -h\left( \gamma
\right) \right\vert ,
\end{equation*}%
by taking $\delta $ sufficiently small, the right hand side is less than $%
\varepsilon $.
\end{proof}

Assuming Lemma \ref{7.1}, we next show (\ref{liminf}) and therefore complete
the proof of the locally uniform LDP upper bound. For $\delta >0$, pick $%
\gamma \in AC\left( \left[ 0,1\right] :\mathcal{S}\right) $ such that $%
\gamma \left( 0\right) =\mu _{0},\gamma \left( t\right) \in \bar{B}\left(
x,\varepsilon \right) $, and $I_{t}\left( \gamma \right) \leq \bar{J}%
_{t}^{\varepsilon }\left( \mu _{0},x\right) +\delta $. By Lemma \ref{7.1}
there exists a path $\nu \in AC\left( \left[ 0,\varepsilon \right] :\mathcal{%
S}\right) $ with $\nu \left( 0\right) =\gamma \left( t\right) $, $\nu \left(
\varepsilon \right) =x$ with $I_{\varepsilon }\left( \nu \right) \leq
c\left( \varepsilon \right) $, where $c\left( \varepsilon \right)
\rightarrow 0$ as $\varepsilon \rightarrow 0$. Let $\bar{\gamma}$ be the
concatenation of $\gamma $ and $\nu $. We now rescale $\overline{\gamma }$
to obtain a new path: for $c=\left( t+\varepsilon \right) /t$, define $\bar{%
\gamma}_{c}\in AC\left( \left[ 0,t\right] :\mathcal{S}\right) $ by $\bar{%
\gamma}_{c}\left( s\right) =\bar{\gamma}\left( cs\right) $, $s\in \left[ 0,t%
\right] $. Then $\bar{\gamma}_{c}\left( 0\right) =\mu _{0}$, $\bar{\gamma}%
_{c}\left( t\right) =x$. Moreover, by Proposition \ref{4.8}, for $%
\varepsilon $ sufficiently small, $I_{t}\left( \bar{\gamma}_{c}\right) \leq
I_{t+\varepsilon }\left( \overline{\gamma }\right) +\delta $, and by the
construction above,%
\begin{equation*}
J_{t}\left( \mu _{0},x\right) \leq I_{t}\left( \bar{\gamma}_{c}\right) \leq
I_{t+\varepsilon }\left( \overline{\gamma }\right) +\delta =I_{t}\left(
\gamma \right) +I_{\varepsilon }\left( \nu \right) +\delta \leq \bar{J}%
_{t}^{\varepsilon }\left( \mu _{0},x\right) +2\delta +c\left( \varepsilon
\right) .
\end{equation*}%
Taking the limit inferior as $\varepsilon \rightarrow 0$ and then sending $%
\delta \rightarrow 0$, (\ref{liminf}) follows.

\begin{proof}[Proof of Lemma \protect\ref{7.1}]
By Property \ref{prop-dcommunicate}.i) and Remark \ref{flex}, there exists a
strongly communicating path $\gamma \in AC\left( \left[ 0,\varepsilon \right]
:\mathcal{S}\right) $ that satisfies $\gamma \left( 0\right) =x$, $\gamma
\left( \varepsilon \right) =y$, and has constant speed $U\leq c^{\prime
}\left\Vert x-y\right\Vert /\varepsilon \leq c^{\prime }$. Precisely, there
exist $F<\infty $ and $0=t_{0}<t_{1}<\cdots <t_{F}=1$, such that%
\begin{equation*}
\dot{\gamma}\left( t\right) =\sum_{m=1}^{F}Uv_{m}\mathbb{I}%
_{[t_{m-1}\varepsilon ,t_{m}\varepsilon )}\left( t\right) \text{ for a.e. }%
t\in \left[ 0,\varepsilon \right] .
\end{equation*}%
Since $I_{\varepsilon }\left( \gamma \right) =\sum_{m=1}^{F}\left(
I_{t_{m}\varepsilon }\left( \gamma \right) -I_{t_{m-1}\varepsilon }\left(
\gamma \right) \right) $, it suffices to bound each term from above.

Recall from (\ref{def-nv}) that for any $j\in \mathcal{N}_{v_{m}}$, $%
\left\langle e_{j},v_{m}\right\rangle <0$. Let $b_{1}\doteq
\min_{m=1,...,F}\min_{j\in \mathcal{N}_{v_{m}}}\left\vert \left\langle
e_{j},v_{m}\right\rangle \right\vert >0$. Note that for $s\in \lbrack
t_{m-1}\varepsilon ,t_{m}\varepsilon )$, and any $j\in \mathcal{N}_{v_{m}}$, 
$\gamma _{j}\left( t_{m}\varepsilon \right) -\gamma _{j}\left( s\right)
=\left\langle e_{j},v_{m}\right\rangle U\left( t_{m}\varepsilon -s\right) $,
and thus $\gamma _{j}\left( s\right) \geq b_{1}U\left( t_{m}\varepsilon
-s\right) $. Therefore, by Definition \ref{def1''}, there exist constants $%
c_{1}>0$, $p_{1}<\infty $, such that%
\begin{equation*}
\lambda _{v_{m}}\left( \gamma \left( s\right) \right) \geq c_{1}\left( %
\displaystyle\prod\limits_{j\in \mathcal{N}_{v_{m}}}\gamma _{j}\left(
s\right) \right) ^{p_{1}}\geq \tilde{c}_{1}U^{\kappa }\left(
t_{m}\varepsilon -s\right) ^{\kappa },
\end{equation*}%
where $\kappa \doteq dp_{1}<\infty $ and $\tilde{c}_{1}\doteq
c_{1}b_{1}^{dp_{1}}>0$. Thus, by taking $q_{v_{m}}=U$, and $q_{v}=0$ for $%
v\neq v_{m}$ in the first line below, we have 
\begin{eqnarray*}
L\left( \gamma \left( s\right) ,\dot{\gamma}\left( s\right) \right)
&=&\inf_{q \in [0,\infty)^{|{\mathcal{V}}|}:\sum_{v\in \mathcal{V}}vq_{v}=%
\dot{\gamma}\left( s\right) }\sum_{v\in \mathcal{V}}\lambda _{v}\left(
\gamma \left( s\right) \right) \ell \left( \frac{q_{v}}{\lambda _{v}\left(
\gamma \left( s\right) \right) } \right) \\
&\leq &\lambda _{v_{m}}\left( \gamma \left( s\right) \right) \ell \left( 
\frac{U}{\lambda _{v_{m}}\left( \gamma \left( s\right)\right)}\right)
+\sum_{v\in \mathcal{V}\setminus \{v_m\}} \lambda _{v}\left( \gamma \left(
s\right) \right) \\
&\leq &U\log \left( \frac{U}{\tilde{c}_{1}U^{\kappa }\left( t_{m}\varepsilon
-s\right) ^{\kappa }}\right) -U+C_{2} \\
&=&-\left( \kappa -1\right) U\log U-\kappa U\log \left( t_{m}\varepsilon
-s\right) -U\left( 1+\log \tilde{c}_{1}\right) +C_{2},
\end{eqnarray*}%
with $C_{2} \doteq R |{\mathcal{V}}|<\infty $, where $R$ is the bound in %
\eqref{m}. Therefore, 
\begin{eqnarray*}
I_{t_{m}\varepsilon }\left( \gamma \right) -I_{t_{m-1}\varepsilon }\left(
\gamma \right) &=&\int_{t_{m-1}\varepsilon }^{t_{m}\varepsilon }L\left(
\gamma \left( s\right) ,\dot{\gamma}\left( s\right) \right) ds \\
&\leq &\int_{t_{m-1}\varepsilon }^{t_{m}\varepsilon }\left( -\left( \kappa
-1\right) U\log U-\kappa U\log \left( t_{m}\varepsilon -s\right) -U\left(
1+\log \tilde{c}_{1}\right) +C_{2}\right) ds \\
&\leq &-C_{3}\left( U\right) \varepsilon \log \varepsilon +C_{4}\left(
U\right) \varepsilon
\end{eqnarray*}%
for some constants $C_{3}\left( U\right) ,C_{4}\left( U\right) $ such that $%
\sup_{U\in \left[ 0,c^{\prime }\right] }\left( C_{3}\left( U\right) \vee
C_{4}\left( U\right) \right) <\infty $. Summing over $m$, we have $%
J_{\varepsilon }\left( \mu _{0},y\right) \leq I_{\varepsilon }\left( \gamma
\right) \leq c\left( \varepsilon \right) $, where $c\left( \varepsilon
\right) =O\left( \varepsilon |\log \varepsilon |\right) $ as $\varepsilon
\rightarrow 0$.
\end{proof}

\subsection{Proof of the Lower Bound}

For the proof of the lower bound, take any $\varepsilon >0$ small. Then by
the Markov property for $\left\{ \mu ^{n}\right\} $, we have 
\begin{equation*}
\mathbb{P}_{\mu _{0}}\left( \mu ^{n}\left( t\right) =x_{n}\right) \geq 
\mathbb{P}_{\mu _{0}}\left( \mu ^{n}\left( t-\varepsilon \right) \in B\left(
x,\varepsilon \right) \right) \cdot \inf_{w_{n}\in B\left( x,\varepsilon
\right) \cap \mathcal{S}_{n}}\mathbb{P}_{w_{n}}\left( \mu ^{n}\left(
\varepsilon \right) =x_{n}\right) .
\end{equation*}%
The LDP lower bound in Corollary \ref{2.3} implies%
\begin{equation*}
\liminf_{n\rightarrow \infty }\frac{1}{n}\log \mathbb{P}_{\mu _{0}}\left(
\mu ^{n}\left( t-\varepsilon \right) \in B\left( x,\varepsilon \right)
\right) \geq -J_{t-\varepsilon }^{\varepsilon }\left( \mu _{0},x\right) ,
\end{equation*}%
where $J_{t}^{\varepsilon }$ is defined by \eqref{Jeps}. The proof of the
lower bound will be complete if we can show both of the following:

i) $\limsup_{\varepsilon \rightarrow 0}J_{t-\varepsilon }^{\varepsilon
}\left( \mu _{0},x\right) \leq J_{t}\left( \mu _{0},x\right) $.

ii) The \textbf{Local Communication Property}: There exist a function $%
c:[0,\infty )\rightarrow \lbrack 0,\infty )$ that satisfies $c\left(
\varepsilon \right) \rightarrow 0$ as $\varepsilon \rightarrow 0$ and is
such that for all $\varepsilon >0$ sufficiently small, 
\begin{equation*}
\inf_{w_{n}\in B\left( x,\varepsilon \right) \cap \mathcal{S}_{n}}\mathbb{P}%
_{w_{n}}\left( \mu ^{n}\left( \varepsilon \right) =x_{n}\right) \geq \exp
\left( -nc\left( \varepsilon \right) +o\left( n\right) \right) .
\end{equation*}

To prove the first property, we will use Proposition \ref{4.8}. For any $%
\delta >0$, take $\gamma \in AC\left( \left[ 0,1\right] :\mathcal{S}\right) $
such that $\gamma \left( t\right) =x$ and $I_{t}\left( \gamma \right) \leq
J_{t}\left( \mu _{0},x\right) +\delta $. Take $c=t/\left( t-\varepsilon
\right) $ and consider the path $\gamma _{c}\in AC\left( \left[
0,t-\varepsilon \right] :\mathcal{S}\right) $, such that $\gamma _{c}\left(
s\right) \doteq \gamma \left( cs\right) $, $s\in \left[ 0,t\right] $. Then $%
\gamma _{c}\left( 0\right) =\mu _{0}$, $\gamma _{c}\left( t-\varepsilon
\right) =x$. By Proposition \ref{4.8}, given $\delta >0$, for $\varepsilon $
sufficiently small, $I_{t/c}\left( \gamma _{c}\right) \leq I_{t}\left(
\gamma \right) +\delta $, and we have 
\begin{equation*}
J_{t-\varepsilon }^{\varepsilon }\left( \mu _{0},x\right) \leq
I_{t-\varepsilon }\left( \gamma _{c}\right) =I_{t/c}\left( \gamma
_{c}\right) \leq I_{t}\left( \gamma \right) +\delta \leq J_{t}\left( \mu
_{0},x\right) +2\delta .
\end{equation*}%
The conclusion follows on taking first $\varepsilon \rightarrow 0$ and then $%
\delta \rightarrow 0$.

To prove the local communication property, we start with a direct evaluation
of the hitting probability of jump Markov processes on a finite state space.

\begin{lemma}
\label{6.1}Let $\left\{ Y\left( t\right) \right\} _{t\geq 0}$ be a jump
Markov process with finite state space $\left\{
s_{0},s_{1},...,s_{N}\right\} $. For $i=0,...,N-1$, suppose that the jump
rate from state $s_{i}$ to $s_{i+1}$ is $b_{i+1}$, and the sum of jump rates
from state $s_{i}$ to all other states is bounded above by $c<\infty $. If $%
Y\left( 0\right) =s_{0}$, then 
\begin{equation*}
\mathbb{P}\left( Y\left( t\right) =s_{N}\right) \geq \frac{1}{N!}\left( \Pi
_{i=1}^{N}b_{i}\right) t^{N}\exp \left( -ct\right) .
\end{equation*}
\end{lemma}

\begin{proof}
Let $p\left( t\right) $ be the probability distribution of the process at
time $t$: $p_{i}\left( t\right) =\mathbb{P}\left( Y\left( t\right)
=s_{i}\right) $. Then the Kolmogorov forward equation takes the form $\dot{p}%
=Ap$, where $A$ is the $N\times N$ rate matrix for $Y$. Let $r$ be the
unique solution to the system of linear ODEs given by 
\begin{equation*}
\left\{ 
\begin{array}{rcll}
\dot{r}_{0} & = & -cr_{0}, &  \\ 
\dot{r}_{i} & = & b_{i}r_{i-1}-cr_{i}, & i=1,...,N, \\ 
r\left( 0\right) & = & e_{s_{0}}. & 
\end{array}%
\right.
\end{equation*}%
Solving this equation explicitly gives $r_{N}\left( t\right) =\frac{1}{N!}%
\left( \Pi _{i=1}^{N}b_{i}\right) t^{N}\exp \left( -ct\right) .$ Since $%
r(0)=p(0)$, the comparison principle for ODEs shows that $p_{i}\left(
t\right) \geq r_{i}\left( t\right) $ for all $i=1,\ldots ,N,$ and the lemma
is proved.
\end{proof}

\begin{proof}[Proof of the local communication property]
We will use Property \ref{prop-dcommunicate} and Lemma \ref{6.1}. Fix some $%
w_{n}\in B\left( x,\varepsilon \right) \cap \mathcal{S}_{n}$, note that the
probability of $\mu ^{n}\left( \varepsilon \right) =x_{n}$ is no less than
the probability that $\mu ^{n}$ hitting $x_{n}$ at $\varepsilon $ by passing
through the states of a given discrete strongly communicating path $\phi $
that connects $w_{n}$ and $x_{n}$.

By Property \ref{prop-dcommunicate}, there exists $F<\infty $, $0=t_{0}\leq
t_{1}\leq \cdots \leq t_{F}=T$, $\left\{ v_{m}\right\} _{m=1}^{F}$, and
constants $c_{1}>0$, $c^{\prime },p_{1}<\infty $, such that $\phi _{0}=w_{n}$%
, $\phi _{T}=x_{n}$, and%
\begin{equation*}
\phi _{s+1}-\phi _{s}=\frac{1}{n}v_{m},\text{ \ for }s\in \lbrack
t_{m-1},t_{m})\cap \mathbb{Z},
\end{equation*}%
with $T\leq c^{\prime }n\left\Vert x_{n}-w_{n}\right\Vert \leq c^{\prime
}n\varepsilon $. Also, for $s\in \lbrack t_{m-1},t_{m})\cap \mathbb{Z}$ and
large $n$, $\lambda _{v_{m}}^{n}\left( \phi _{s}\right) >c_{1}(\textstyle%
\prod\limits_{j\in \mathcal{N}_{m}}\left( \phi _{s}\right) _{j})^{p_{1}}$.
Let $z^{\left( m\right) }\doteq \phi _{s}\left( t_{m}\right) $. By the
Markov property, 
\begin{equation*}
\mathbb{P}_{w_{n}}\left( \mu ^{n}\left( \varepsilon \right) =x_{n}\right)
\geq \textstyle\prod\limits_{m=1}^{F}\mathbb{P}_{z^{\left( m\right) }}\left(
\mu ^{n}\left( \left( \frac{t_{m+1}-t_{m}}{T}\right) \varepsilon \right)
=z^{\left( m+1\right) }\right) ,
\end{equation*}%
and it suffices to give a lower bound for each term in the product. This
will be proved by comparison with another Markov process $Z^{n}$. Thus,
without modifying the notation, we let $\mu ^{n}\left( t\right) $ denote the
process starting at $z^{\left( m\right) }$ and stopped when it first leaves
the set of points $\left\{ \phi _{s}:s\in \lbrack t_{m-1},t_{m})\cap \mathbb{%
Z}\right\} $. For each $m$ and $t\in \lbrack 0,(t_{m+1}-t_{m})\varepsilon
/T) $, define $Z^{n}$ to be the jump Markov process with $Z^{n}\left(
0\right) =z^{\left( m\right) }$, with the same set $\frac{1}{n}\mathcal{V}$
of jump directions, and jump rates%
\begin{equation*}
\bar{\lambda}_{v}\left( x\right) \doteq \left\{ 
\begin{array}{cc}
nc_{1}\left( \displaystyle\prod\limits_{j\in \mathcal{N}_{v_{m}}}x_{j}%
\right) ^{p_{1}} & \text{if }v=v_{m}, \\ 
nR & \text{if }v\in \mathcal{V}\setminus \{v_{m}\},%
\end{array}%
\right.
\end{equation*}%
as long as $Z^{n}$ stays in the set $\left\{ \phi _{s}:s\in \lbrack
t_{m-1},t_{m})\cap \mathbb{Z}\right\} $, and with the process stopped when
it jumps off the line segment. Note that $\bar{\lambda}_{v}\left( x\right) $
bounds $\lambda _{v_{m}}^{n}\left( x\right) $ from below in the set, while $%
nR$ is an upper bound on all jump rates.

It follows by the comparison principle in Lemma \ref{6.1} that $\mu ^{n}$
has a higher probability to reach $x_{m+1}$ at time $(t_{m+1}-t_{m})%
\varepsilon /T$ than $Z^{n}$ does: 
\begin{equation*}
\mathbb{P}_{x_{m}}\left( \mu ^{n}\left( \left( \frac{t_{m+1}-t_{m}}{T}%
\right) \varepsilon \right) =x_{m+1}\right) \geq \mathbb{P}_{x_{m}}\left(
Z^{n}\left( \left( \frac{t_{m+1}-t_{m}}{T}\right) \varepsilon \right)
=x_{m+1}\right) .
\end{equation*}%
Let $l\doteq l\left( n,\varepsilon \right) =t_{m+1}-t_{m}$. Then by
Definition \ref{def1}, $l\leq C_{2}n\varepsilon $ for some $C_{2}<\infty $.
The product of the jump rates of $Z^{n}$ along this segment satisfies%
\begin{equation*}
\displaystyle\prod_{x:x\in \left\{ \phi _{s}:s\in \lbrack t_{m-1},t_{m})\cap 
\mathbb{Z}\right\} }\left( nc_{1}\left( \displaystyle\prod\limits_{j\in 
\mathcal{N}_{v_{m}}}x_{j}\right) ^{p_{1}}\right) \geq \frac{c_{1}^{l}\left(
l!\right) ^{\kappa _{m}p_{1}}}{n^{\left( \kappa _{m}p_{1}-1\right) l}},
\end{equation*}%
where $\kappa _{m}\doteq \left\vert \mathcal{N}_{m}\right\vert \leq d$. The
lower bound in the last inequality is achieved when $\left\{ \phi _{s}:s\in
\lbrack t_{m-1},t_{m})\cap \mathbb{Z}\right\} $ is a segment that ends at $%
x_{m+1}\in \partial \mathcal{S}$, and for all $j\in \mathcal{N}_{m}$, $%
x_{j}=1,...,l$ along the segment. Then it follows from Lemma \ref{6.1} that
for $\varepsilon >0$ sufficiently small,%
\begin{align*}
& \mathbb{P}_{x_{m}}\left( \mu ^{n}\left( \left( \frac{t_{m+1}-t_{m}}{T}%
\right) \varepsilon \right) =x_{m+1}\right) \\
& \quad \geq \mathbb{P}_{x_{m}}\left( Z^{n}\left( \left( \frac{t_{m+1}-t_{m}%
}{T}\right) \varepsilon \right) =x_{m+1}\right) \\
& \quad \geq \frac{1}{l!}\left[ \frac{c_{1}^{l}\left( l!\right) ^{k_{m}p_{1}}%
}{n^{\left( k_{m}p_{1}-1\right) l}}\right] \left( \left( \frac{t_{m+1}-t_{m}%
}{T}\right) \varepsilon \right) ^{l}\exp \left( -nR\left\vert \mathcal{V}%
\right\vert \left( \frac{t_{m+1}-t_{m}}{T}\right) \varepsilon \right) \\
& \quad \geq c_{1}^{l}\left( \frac{l!}{n^{l}}\right) ^{dp_{1}-1}\varepsilon
^{T}\exp \left( -nR\left\vert \mathcal{V}\right\vert \varepsilon \right) .
\end{align*}%
To obtain the last inequality, we write $x=$ $(t_{m+1}-t_{m})/T$, and use
the fact that $x\leq 1$, and for $\varepsilon <e^{-1}$, the function%
\begin{equation*}
\left( x\varepsilon \right) ^{xT}\exp \left( -nR\left\vert \mathcal{V}%
\right\vert \varepsilon x\right) =\exp \left( xT\log \left( x\varepsilon
\right) -nR\left\vert \mathcal{V}\right\vert \varepsilon x\right)
\end{equation*}%
is decreasing for $x\in (0,1]$. Applying Stirling's approximation and
noticing $T\leq c^{\prime }n\varepsilon $, we have%
\begin{align*}
& \frac{1}{n}\log \mathbb{P}_{x_{m}}\left( \mu ^{n}\left( \left( \frac{%
t_{m+1}-t_{m}}{T}\right) \varepsilon \right) =x_{m+1}\right) \\
& \quad \geq c^{\prime }\varepsilon \log c_{1}+\left( dp_{1}-1\right) \frac{l%
}{n}\log \frac{l}{n}-\left( dp_{1}-1\right) c^{\prime }\varepsilon \log
e+c^{\prime }\varepsilon \log \varepsilon -R\left\vert \mathcal{V}%
\right\vert \varepsilon +o\left( \varepsilon \right) \\
& \quad \geq dp_{1}c^{\prime }\varepsilon \log \varepsilon +O(\varepsilon
)+o\left( 1\right) ,
\end{align*}%
where $o\left( 1\right) $ tends to zero as $n\rightarrow \infty $. Taking
the product in $m$, we conclude $\frac{1}{n}\log \mathbb{P}_{w}\left( \mu
^{n}\left( \varepsilon \right) =x_{n}\right) \geq -c\left( \varepsilon
\right) +O(\varepsilon )+o\left( 1\right) $ with $c\left( \varepsilon
\right) =O\left( \varepsilon \log \varepsilon \right) $, as desired.
\end{proof}

\appendix

\section{Proof of Theorem \protect\ref{3.2}}

\label{sec-ap}

We now present the proof of Theorem \ref{3.2}. Recall that $h_n$ maps a
controlled PRM into a controlled process, and is defined in (\ref{hn}).
Recall also the definitions of $\mathcal{A}_{b}^{\otimes \left\vert \mathcal{%
V}\right\vert }$ and $\mathcal{\bar{A}}_{b}^{\otimes \left\vert \mathcal{V}%
\right\vert }$ in Definition \ref{ab} and Definition \ref{ab-},
respectively. The claim of Theorem \ref{3.2} is essentially that the
additional dependence of controls in $\mathcal{\bar{A}}_{b}^{\otimes
\left\vert \mathcal{V}\right\vert }$ on the \textquotedblleft
type\textquotedblright\ of jump is not needed, and that the variational
representation is valid with the simpler controls $\mathcal{A}_{b}^{\otimes
\left\vert \mathcal{V}\right\vert }$. We recall that the controls in $%
\mathcal{\bar{A}}_{b}^{\otimes \left\vert \mathcal{V}\right\vert }$ modulate
the intensity of the driving PRM in an $s,x$ and $\omega $ dependent
fashion, while the controls in $\mathcal{A}_{b}^{\otimes \left\vert \mathcal{%
V}\right\vert }$ multiply the jump rates $\lambda _{v}^{n}$ in an $s$ and $%
\omega $ dependent way.

The proof of Theorem \ref{3.2} will follow from Lemma \ref{rep}, and the
results Corollary \ref{d=} and Lemma \ref{inf} established below. For
simplicity we assume $T=1$. We start with two lemmas that elucidate the
relation between elements of $\mathcal{A}_{b}^{\otimes \left\vert \mathcal{V}%
\right\vert }$ and $\mathcal{\bar{A}}_{b}^{\otimes \left\vert \mathcal{V}%
\right\vert }$.

\begin{lemma}
\label{c1}There exists a map $\Theta ^{n}:\mathcal{\bar{A}}_{b}^{\otimes
\left\vert \mathcal{V}\right\vert }\rightarrow \mathcal{A}_{b}^{\otimes
\left\vert \mathcal{V}\right\vert }\times D\left( \left[ 0,1\right] :%
\mathcal{S}\right) \times \mathcal{\bar{A}}_{b}^{\otimes \left\vert \mathcal{%
V}\right\vert }$ that takes $\varphi \in \mathcal{\bar{A}}_{b}^{\otimes
\left\vert \mathcal{V}\right\vert }$ into a triple $\left( \hat{\alpha}^{n},%
\hat{\mu}^{n},\hat{\varphi}^{n}\right) $, such that for any $v\in \mathcal{V}
$, the following is true: \newline
1. $\hat{\alpha}_{v}^{n}\left( s\right) =\int_{0}^{\lambda _{v}^{n}\left( 
\hat{\mu}^{n}(s)\right) }\varphi _{v}\left( s,y\right) dy$, \newline
2. $\hat{\mu}^{n}=h_{n}\left( \frac{1}{n}N^{n\hat{\varphi}^{n}},\mu
^{n}\left( 0\right) ,\lambda ^{n}\right) .$ \newline
3. $\hat{\varphi}_{v}^{n}\left( s,y\right) =\frac{\hat{\alpha}_{v}^{n}(s)}{%
\lambda _{v}^{n}\left( \hat{\mu}^{n}(s)\right) }\mathbb{I}_{[0,\lambda
_{v}^{n}\left( \hat{\mu}^{n}(s)\right) ]}(y)+\mathbb{I}_{[0,\lambda
_{v}^{n}\left( \hat{\mu}^{n}(s)\right) ]^{c}}(y)$.
\end{lemma}

Note that given any control $\varphi \in \mathcal{\bar{A}}_{b}^{\otimes
\left\vert \mathcal{V}\right\vert }$, this lemma identifies a structurally
simpler control $\hat{\varphi}^{n}\in \mathcal{\bar{A}}_{b}^{\otimes
\left\vert \mathcal{V}\right\vert }$.

\begin{proof}
We prove the claim by a recursive construction.

1. To begin the recursion set $t_{0}=0$, and given any $\varphi \in \mathcal{%
\bar{A}}_{b}^{\otimes \left\vert \mathcal{V}\right\vert }$, define for $%
s\geq t_{0}$ and $v\in \mathcal{V}$, 
\begin{eqnarray*}
\hat{\mu}^{n,0}\left( s\right) &=&\mu ^{n}\left( 0\right) \text{, } \\
\hat{\alpha}_{v}^{n,0}\left( s\right) &=&\int_{0}^{\lambda _{v}^{n}(\hat{\mu}%
^{n,0}(s))}\varphi _{v}\left( s,y\right) dy\text{, } \\
\hat{\varphi}_{v}^{n,0}\left( s,y\right) &=&\frac{\hat{\alpha}%
_{v}^{n,0}\left( s\right) }{\lambda _{v}^{n}\left( \hat{\mu}^{n,0}(s)\right) 
}\mathbb{I}_{\left[ 0,\lambda _{v}^{n}(\hat{\mu}^{n,0}(s))\right] }\left(
y\right) +\mathbb{I}_{\left[ 0,\lambda _{v}^{n}(\hat{\mu}^{n,0}(s))\right]
^{c}}\left( y\right) .
\end{eqnarray*}%
In other words, for $s>0$ and $y$ inside the compact set $\left[ 0,\lambda
_{v}^{n}\left( \hat{\mu}^{n,0}(s)\right) \right] $, we set $\hat{\varphi}%
_{v}^{n,0}$ to be the average of $\varphi _{v}\left( s,\cdot \right) $ over
the set, while for $y$ in the complement we set $\hat{\varphi}_{v}^{n,0}=1$.
We see that by construction $\left\Vert \hat{\varphi}_{v}^{n,0}\right\Vert
_{\infty }\leq \left\Vert \varphi _{v}\right\Vert _{\infty }\vee 1$.

2. Assume now that for some $k\in \mathbb{N}_{0}$, $t_{k}$ is well defined, $%
(\{\hat{\varphi}_{v}^{n,k}\left( s\right) \},\{\hat{\alpha}_{v}^{n,k}\left(
s\right) \},\{\hat{\mu}^{n,k}\left( s\right) \})$ is well defined for $s\in %
\left[ 0,1\right] $, and 
\begin{equation*}
\left\Vert \hat{\varphi}_{v}^{n,k}\right\Vert _{\infty ,[t_{k},\infty )}\dot{%
=}\sup_{\left( s,y\right) \in \lbrack t_{k},\infty )\times \mathbb{R}%
}\left\vert \hat{\varphi}_{v}^{n,k}\left( s,y\right) \right\vert \leq
\left\Vert \varphi _{v}\right\Vert _{\infty }\vee 1.
\end{equation*}%
For any $t\geq t_{k}$ and $v\in \mathcal{V}$, define%
\begin{equation*}
\hat{B}_{k,v}\left( t\right) =\left\{ \left( s,y,r\right) :s\in \left[
t_{k},t\right] ,y\in \left[ 0,\lambda _{v}^{n}(\hat{\mu}^{n,k}(s))\right]
,r\in \left[ 0,\hat{\varphi}_{v}^{n,k}\left( s,y\right) \right] \right\}
\end{equation*}%
and 
\begin{equation*}
t_{k+1}=\inf \left\{ t>t_{k}\text{ such that for some }v\in \mathcal{V},\bar{%
N}_{v}^{n}(\hat{B}_{k,v}\left( t\right) )>0\right\} \wedge 1.
\end{equation*}%
We define $\hat{\mu}^{n,k+1}$ on $\left[ 0,1\right] $ by first setting $\hat{%
\mu}^{n,k+1}\left( s\right) =\hat{\mu}^{n,k}\left( s\right) $ for $s\in %
\left[ 0,t_{k+1}\right) $. Then, at $t_{k+1}$, we update $\hat{\mu}^{n,k+1}$
by setting 
\begin{equation*}
\hat{\mu}^{n,k+1}\left( t_{k+1}\right) =\hat{\mu}^{n,k}\left( t_{k}\right)
+\sum_{v\in \mathcal{V}}v\int_{[t_{k},t_{k+1}]}\int_{\mathcal{Y}}\mathbb{I}_{%
\left[ 0,\lambda _{v}^{n}\left( \hat{\mu}^{n,k}(s-)\right) \right]
}(y)\int_{[0,\infty )}\mathbb{I}_{[0,\hat{\varphi}_{v}^{n,k}(s-,y)]}(r)\frac{%
1}{n}\bar{N}_{v}^{n}(dsdydr),
\end{equation*}%
and set $\hat{\mu}^{n,k+1}\left( s\right) =\hat{\mu}^{n,k+1}\left(
t_{k+1}\right) $ for $s\geq t_{k+1}$. Define 
\begin{eqnarray*}
\hat{\alpha}_{v}^{n,k+1}\left( s\right) &=&\int_{0}^{\lambda _{v}^{n}(\hat{%
\mu}^{n,k+1}(s))}\varphi _{v}\left( s,y\right) dy\text{, } \\
\hat{\varphi}_{v}^{n,k+1}\left( s,y\right) &=&\frac{\hat{\alpha}%
_{v}^{n,k+1}\left( s\right) }{\lambda _{v}^{n}\left( \hat{\mu}%
^{n,k+1}(s)\right) }\mathbb{I}_{\left[ 0,\lambda _{v}^{n}(\hat{\mu}%
^{n,k+1}(s))\right] }\left( y\right) +\mathbb{I}_{\left[ 0,\lambda _{v}^{n}(%
\hat{\mu}^{n,k+1}(s))\right] ^{c}}\left( y\right) .
\end{eqnarray*}

We also have $\left\Vert \hat{\varphi}_{v}^{n,k+1}\right\Vert _{\infty
,[t_{k+1},\infty )}\leq \left\Vert \varphi _{v}\right\Vert _{\infty }\vee 1$.

3. Recall $\bar{R}$ defined as in (\ref{m'}). Since $\bar{N}_{v}^{n}$ has
a.s. finitely many atoms on $\left[ 0,1\right] \times \lbrack 0,\bar{R}%
]\times \left[ 0,\left\Vert \varphi _{v}\right\Vert _{\infty }\right] $, the
construction will produce functions defined on $\left[ 0,1\right] $ in $%
L<\infty $ steps. Then set%
\begin{equation*}
\hat{\mu}^{n}\left( s\right) =\hat{\mu}^{n,L}\left( s\right) \text{, }\hat{%
\alpha}_{v}^{n}\left( s\right) =\hat{\alpha}_{v}^{n,L}\left( s\right) \text{%
, }\hat{\varphi}_{v}^{n}\left( s\right) =\hat{\varphi}_{v}^{n,L}\left(
s\right) \text{, if }s\in \lbrack 0,1]\text{.}
\end{equation*}%
Then $\hat{\varphi}^{n}\in \mathcal{\bar{A}}_{b}^{\otimes \left\vert 
\mathcal{V}\right\vert }$. Furthermore, by construction 
\begin{equation*}
\hat{\mu}^{n}=h_n\left( \frac{1}{n}N^{n\hat{\varphi}^{n}},\mu ^{n}\left(
0\right) ,\lambda ^{n}\right) .
\end{equation*}
\end{proof}

The next lemma shows that from controls in $\mathcal{A}_{b}^{\otimes
\left\vert \mathcal{V}\right\vert }$ we can produce corresponding controls
in $\mathcal{\bar{A}}^{\otimes \left\vert \mathcal{V}\right\vert }$.

\begin{lemma}
\label{exst}There exists a map $\Xi ^{n}:\mathcal{A}_{b}^{\otimes \left\vert 
\mathcal{V}\right\vert }\rightarrow D\left( \left[ 0,1\right] :\mathcal{S}%
\right) \times \mathcal{\bar{A}}^{\otimes \left\vert \mathcal{V}\right\vert
} $ which takes $\bar{\alpha}\in \mathcal{A}_{b}^{\otimes \left\vert 
\mathcal{V}\right\vert }$ into a pair $\left( \bar{\mu}^{n},\bar{\varphi}%
\right) $, such that $\bar{\mu}^{n}=h\left( \frac{1}{n}N^{n\bar{\varphi}%
},\mu ^{n}\left( 0\right) ,\lambda ^{n}\right) $, where for $v\in \mathcal{V}
$, $\bar{\varphi}_{v}\left( s,y\right) =\frac{\bar{\alpha}_{v}(s)}{\lambda
_{v}^{n}\left( \bar{\mu}^{n}(s)\right) }\mathbb{I}_{[0,\lambda
_{v}^{n}\left( \bar{\mu}^{n}(s)\right) ]}(y)+\mathbb{I}_{[0,\lambda
_{v}^{n}\left( \bar{\mu}^{n}(s)\right) ]^{c}}(y).$
\end{lemma}

\begin{proof}
1. Define $t_{0}=0$ and for any $\bar{\alpha}\in \mathcal{A}_{b}^{\otimes
\left\vert \mathcal{V}\right\vert }$, $s\geq t_{0}$ and $v\in \mathcal{V}$,
define%
\begin{eqnarray*}
\bar{\mu}^{n,0}\left( s\right) &=&\mu ^{n}\left( 0\right) \text{, } \\
\bar{\varphi}_{v}^{0}\left( s,y\right) &=&\frac{\bar{\alpha}_{v}\left(
s\right) }{\lambda _{v}^{n}\left( \bar{\mu}^{n,0}(s)\right) }\mathbb{I}_{%
\left[ 0,\lambda _{v}^{n}\left( \bar{\mu}^{n,0}(s)\right) \right] }\left(
y\right) +\mathbb{I}_{\left[ 0,\lambda _{v}^{n}\left( \bar{\mu}%
^{n,0}(s)\right) \right] ^{c}}\left( y\right) .
\end{eqnarray*}

2. Assume now that for some $k\in \mathbb{N}$, $t_{k}$ is well defined, and
that $\left( \bar{\mu}^{n,k}\left( s\right) ,\left\{ \bar{\varphi}%
_{v}^{k}\left( s\right) \right\} \right) $ is well defined for $s\geq t_{k}$%
. For any $t\geq t_{k}$, define%
\begin{equation*}
\bar{A}_{k,v}\left( t\right) =\left\{ \left( s,y,r\right) :s\in \left[
t_{k},t\right] ,y\in \left[ 0,\lambda _{v}^{n}\left( \bar{\mu}%
^{n,k}(s)\right) \right] ,r\in \left[ 0,\bar{\varphi}_{v}^{k}\left(
s,y\right) \right] \right\}
\end{equation*}%
and 
\begin{equation*}
t_{k+1}=\inf \left\{ t>t_{k}\text{ such that for some }v\in \mathcal{V},\bar{%
N}_{v}^{n}\left( \bar{A}_{k,v}\left( t\right) \right) >0\right\} \wedge 1.
\end{equation*}%
We define $\bar{\mu}^{n,k+1}$ on $\left[ 0,1\right] $ by first setting $\bar{%
\mu}^{n,k+1}\left( s\right) =\bar{\mu}^{n,k}\left( s\right) $ for $s\in %
\left[ 0,t_{k+1}\right) $. Then, at $t_{k+1}$, we update $\bar{\mu}^{n,k+1}$
by 
\begin{equation*}
\bar{\mu}^{n,k+1}\left( t_{k+1}\right) =\bar{\mu}^{n,k}\left( t_{k}\right)
+\sum_{v\in \mathcal{V}}v\int_{[t_{k},t_{k+1}]}\int_{\mathcal{Y}}\mathbb{I}_{%
\left[ 0,\lambda _{v}^{n}\left( \bar{\mu}^{n,k}(s-)\right) \right]
}(y)\int_{[0,\infty )}\mathbb{I}_{[0,\bar{\varphi}_{v}^{k}(s-,y)]}(r)\frac{1%
}{n}\bar{N}_{v}^{n}(dsdydr),
\end{equation*}%
and set $\bar{\mu}^{n,k+1}\left( s\right) =\bar{\mu}^{n,k+1}\left(
t_{k+1}\right) $ for $s\geq t_{k+1}$. Define%
\begin{equation*}
\bar{\varphi}_{v}^{k+1}\left( s,y\right) =\frac{\bar{\alpha}_{v}\left(
s\right) }{\lambda _{v}^{n}\left( \bar{\mu}^{n,k+1}(s)\right) }\mathbb{I}_{%
\left[ 0,\lambda _{v}^{n}\left( \bar{\mu}^{n,k+1}(s)\right) \right] }\left(
y\right) +\mathbb{I}_{\left[ 0,\lambda _{v}^{n}\left( \bar{\mu}%
^{n,k+1}(s)\right) \right] ^{c}}\left( y\right) .
\end{equation*}%
\newline

3. Since $\bar{N}_{v}^{n}$ has a.s. finitely many atoms on $\left[ 0,1\right]
\times \left[ 0,\left\Vert \bar{\alpha}_{v}\right\Vert _{\infty }\right] $,
the construction will produce functions defined on $\left[ 0,1\right] $ in $%
L<\infty $ steps. Then set%
\begin{equation*}
\bar{\mu}^{n}\left( s\right) =\bar{\mu}^{n,L}\left( s\right) \text{, }\bar{%
\varphi}_{v}\left( s\right) =\bar{\varphi}_{v}^{L}\left( s\right) \text{, }%
s\in \lbrack 0,1].
\end{equation*}%
Note that 
\begin{equation*}
\bar{\varphi}_{v}\left( s,y\right) =\frac{\bar{\alpha}_{v}(s)}{\lambda
_{v}^{n}\left( \bar{\mu}^{n}(s)\right) }\mathbb{I}_{[0,\lambda
_{v}^{n}\left( \bar{\mu}^{n}(s)\right) ]}(y)+\mathbb{I}_{[0,\lambda
_{v}^{n}\left( \bar{\mu}^{n}(s)\right) ]^{c}}(y)
\end{equation*}%
and $\bar{\mu}^{n}$ satisfies 
\begin{equation*}
\bar{\mu}^{n}=h_n\left( \frac{1}{n}N^{n\bar{\varphi}},\mu ^{n}\left(
0\right) ,\lambda ^{n}\right) .
\end{equation*}
\end{proof}

The next result is a corollary to the construction in Lemma \ref{exst}. Let $%
\Sigma_1^n: \mathcal{A}_{b}^{\otimes \left\vert \mathcal{V}\right\vert }
\mapsto D([0,1]:{\mathcal{S}})$ denote the first component of the map in
Lemma \ref{exst}.

\begin{corollary}
\label{d=}Take any $\left\{ \bar{\alpha}_{v}\right\} \in \mathcal{A}%
_{b}^{\otimes \left\vert \mathcal{V}\right\vert }$ such that $\bar{\alpha}%
_{v}\left( t\right) =0$ when $\bar{\mu}^{n}(t)=x\in \partial \mathcal{S}$
and $x+\frac{1}{n}v$ is taken outside $\mathcal{S}$. For any $t\in \left[ 0,1%
\right] $, $\Xi _{1}^{n}\left( \bar{\alpha}\right) \left( t\right) $ has the
same distribution as $\Lambda ^{n}\left( \bar{\alpha},\mu ^{n}\left(
0\right) \right) \left( t\right) $, where $\Lambda ^{n}$ is as defined in (%
\ref{lam}).
\end{corollary}

\begin{proof}
Recall that $\bar{\mu}^{n}=\Xi _{1}^{n}\left( \bar{\alpha}\right) $. We have 
$\Xi _{1}^{n}\left( \bar{\alpha}\right) \left( 0\right) =\mu ^{n}\left(
0\right) $. Given $s\in \left[ 0,1\right] $, the total jump intensity of $%
\bar{\mu}^{n}\left( s\right) $ in the direction $v$ is 
\begin{equation*}
\int_{0}^{\lambda _{v}^{n}\left( \bar{\mu}^{n}(s)\right) }\bar{\varphi}%
_{v}\left( s,y\right) dy=\int_{0}^{\lambda _{v}^{n}\left( \bar{\mu}%
^{n}(s)\right) }\frac{\bar{\alpha}_{v}(s)}{\lambda _{v}^{n}\left( \bar{\mu}%
^{n}(s)\right) }dy=\bar{\alpha}_{v}(s)
\end{equation*}%
which is the same as that of $\Lambda ^{n}\left( \bar{\alpha},\mu ^{n}\left(
0\right) \right) \left( s\right) $.
\end{proof}

\begin{lemma}
\label{inf}Let $\mathcal{A}_{b}$, $\mathcal{\bar{A}}_{b}$ and $\mathcal{\bar{%
A}}$ be as defined in Definitions \ref{ab}, \ref{ab-} and \ref{a-}
respectively, and let $L_{1}$ be defined as in (\ref{L_T}) with $T=1$. Then
for $F\in M_{b}\left( D\left( \left[ 0,T\right] :\mathcal{S}\right) \right) $%
,%
\begin{align*}
& \inf_{\varphi \in \mathcal{\bar{A}}_{b}^{\otimes \left\vert \mathcal{V}%
\right\vert }}\mathbb{\bar{E}}\left[ \sum_{v\in \mathcal{V}}L_{1}(\varphi
_{v})+F(\bar{\mu}^{n}):\bar{\mu}^{n}=h_{n}\left( \frac{1}{n}N^{n\varphi
},\mu ^{n}\left( 0\right) ,\lambda ^{n}\right) \right] \\
& \quad =\inf_{\varphi \in \mathcal{\bar{A}}^{\otimes \left\vert \mathcal{V}%
\right\vert }}\mathbb{\bar{E}}\left[ \sum_{v\in \mathcal{V}}L_{1}(\varphi
_{v})+F(\bar{\mu}^{n}):\bar{\mu}^{n}=h_{n}\left( \frac{1}{n}N^{n\varphi
},\mu ^{n}\left( 0\right) ,\lambda ^{n}\right) \right] \\
& \quad =\inf_{\bar{\alpha}\in \mathcal{A}_{b}^{\otimes \left\vert \mathcal{V%
}\right\vert }}\mathbb{\bar{E}}\left[ \sum_{v\in \mathcal{V}%
}\int_{0}^{1}\lambda _{v}^{n}\left( \bar{\mu}^{n}(t)\right) \ell \left( 
\frac{\bar{\alpha}_{v}(t)}{\lambda _{v}^{n}\left( \bar{\mu}^{n}(t)\right) }%
\right) dt+F(\bar{\mu}^{n}):\bar{\mu}^{n}=\Xi _{1}^{n}\left( \bar{\alpha}%
\right) \right] ,
\end{align*}%
where $\Xi ^{n}$ is as defined in Lemma \ref{exst}.
\end{lemma}

\begin{proof}
The first equality is a consequence of Theorem 2.4 of \cite{BCD}. To prove
the rest of the claim, fix $\bar{\alpha}\in \mathcal{A}_{b}^{\otimes
\left\vert \mathcal{V}\right\vert }$ such that $\bar{\alpha}_{v}\left(
t\right) =0$ when $\bar{\mu}^{n}(t)=x\in \partial \mathcal{S}$ and $x+\frac{1%
}{n}v$ is taken outside $\mathcal{S}$. Let $\left( \bar{\mu}^{n},\bar{\varphi%
}\right) =\Xi ^{n}\left( \bar{\alpha}\right) $. Then by definition $\bar{%
\varphi}\in \mathcal{\bar{A}}^{\otimes \left\vert \mathcal{V}\right\vert }$,
and since $\nu $ in (\ref{L_T}) is Lebesgue measure 
\begin{equation*}
L_{1}(\bar{\varphi}_{v})=\sum_{v\in \mathcal{V}}\int_{0}^{\infty
}\int_{0}^{1}\ell \left( \bar{\varphi}_{v}\left( t,y\right) \right)
dtdy=\sum_{v\in \mathcal{V}}\int_{0}^{1}\lambda _{v}^{n}\left( \bar{\mu}%
^{n}(t)\right) \ell \left( \frac{\bar{\alpha}_{v}(t)}{\lambda _{v}^{n}\left( 
\bar{\mu}^{n}(t)\right) }\right) dt.
\end{equation*}%
Now it follows from Lemma \ref{exst} that 
\begin{align*}
& \inf_{\varphi \in \mathcal{\bar{A}}^{\otimes \left\vert \mathcal{V}%
\right\vert }}\mathbb{\bar{E}}\left[ \sum_{v\in \mathcal{V}}L_{1}(\varphi
_{v})+F(\bar{\mu}^{n}):\bar{\mu}^{n}=h\left( \frac{1}{n}N^{n\varphi },\mu
^{n}\left( 0\right) ,\lambda ^{n}\right) \right]  \\
& \quad \leq \mathbb{\bar{E}}\left[ \sum_{v\in \mathcal{V}}L_{1}(\bar{\varphi%
}_{v})+F\circ h\left( \frac{1}{n}N^{n\bar{\varphi}},\mu ^{n}\left( 0\right)
,\lambda ^{n}\right) \right]  \\
& \quad =\mathbb{\bar{E}}\left[ \sum_{v\in \mathcal{V}}\int_{0}^{1}\lambda
_{v}^{n}\left( \bar{\mu}^{n}(t)\right) \ell \left( \frac{\bar{\alpha}_{v}(t)%
}{\lambda _{v}^{n}\left( \bar{\mu}^{n}(t)\right) }\right) dt+F(\bar{\mu}^{n})%
\right] .
\end{align*}%
The reverse inequality is proved by a convexity argument. Recall the
definition of $\Theta ^{n}$ given in Lemma \ref{c1}. For given $\varphi \in 
\mathcal{\bar{A}}_{b}^{\otimes \left\vert \mathcal{V}\right\vert }$, let $%
\left( \bar{\alpha},\bar{\mu}^{n}\right) \dot{=}\left( \Theta _{1}^{n}\left(
\varphi \right) ,\Theta _{2}^{n}\left( \varphi \right) \right) $. Then $\bar{%
\alpha}\in \mathcal{A}_{b}^{\otimes \left\vert \mathcal{V}\right\vert }$. By
convexity of $\ell \left( \cdot \right) $ and Jensen's inequality,%
\begin{align*}
& \int_{0}^{\infty }\int_{0}^{1}\ell \left( \varphi _{v}\left( s,y\right)
\right) dsdy \\
& \quad \geq \int_{0}^{1}\lambda _{v}^{n}\left( \bar{\mu}^{n}(s)\right)
\left( \frac{1}{\lambda _{v}^{n}\left( \bar{\mu}^{n}(s)\right) }%
\int_{0}^{\lambda _{v}^{n}\left( \bar{\mu}^{n}(s)\right) }\ell \left(
\varphi _{v}\left( s,y\right) \right) dy\right) ds \\
& \quad \geq \int_{0}^{1}\lambda _{v}^{n}\left( \bar{\mu}^{n}(s)\right) \ell
\left( \frac{1}{\lambda _{v}^{n}\left( \bar{\mu}^{n}(s)\right) }%
\int_{0}^{\lambda _{v}^{n}\left( \bar{\mu}^{n}(s)\right) }\varphi _{v}\left(
s,y\right) dy\right) ds \\
& \quad =\int_{0}^{1}\lambda _{v}^{n}\left( \bar{\mu}^{n}(s)\right) \ell
\left( \frac{\bar{\alpha}_{v}\left( s\right) }{\lambda _{v}^{n}\left( \bar{%
\mu}^{n}(s)\right) }\right) ds.
\end{align*}%
Summing over $v\in \mathcal{V}$, applying Lemma \ref{exst} and infimizing
over $\bar{\alpha}\in \mathcal{A}_{b}^{\otimes \left\vert \mathcal{V}%
\right\vert }$ we obtain the desired result.
\end{proof}

\bibliographystyle{plain}
\bibliography{LDPref}

\end{document}